\documentclass[12pt,oneside,4apaper]{article}
\usepackage{setspace}
\usepackage[margin=1in]{geometry}
\usepackage[colorlinks=true,citecolor=blue!50!black,linkcolor=red!50!black]{hyperref}
\usepackage[mathscr]{euscript}
\usepackage[utf8]{inputenc}

\usepackage{amsmath}%
\usepackage{amsfonts}%
\usepackage{amsthm}
\usepackage{graphicx}
\usepackage{units}
\usepackage{verbatim}
\usepackage{enumerate}
\usepackage{mathtools}
\usepackage{needspace}

\usepackage{tikz}
\usetikzlibrary{matrix,arrows,arrows.meta,calc,intersections}
\tikzset{>={Stealth[width=1.5mm,length=1.5mm]}}

\usepackage{parskip}

\begingroup
    \makeatletter
    \@for\theoremstyle:=definition,remark,plain\do{%
        \expandafter\g@addto@macro\csname th@\theoremstyle\endcsname{%
            \addtolength\thm@preskip\parskip
            }%
        }
\endgroup

\newtheorem{theorem}{Theorem}[section]
\newtheorem{lemma}[theorem]{Lemma}
\newtheorem{corollary}[theorem]{Corollary}

\theoremstyle{remark}

\newtheorem{remark}[theorem]{Remark}

\newtheorem{observation}[theorem]{Observation}
\newtheorem{justification}[theorem]{Justification}
\newtheorem{warning}[theorem]{Warning}
\newtheorem{question}[theorem]{Question}

\newtheorem{claim}{Claim}[subsubsection]

\newtheorem*{remark*}{Remark}
\newtheorem*{motivation*}{Motivation}
\newtheorem*{observation*}{Observation}
\newtheorem*{justification*}{Justification}
\newtheorem*{warning*}{Warning}

\newenvironment{claimproof}[1][Proof]{\begin{proof}[#1]}{\end{proof}}

\usepackage{xspace}
\newcommand{\newcmd}[3]{\newcommand{#1}[#2]{#3\xspace}}

\newcommand{\mc}{\mathcal}

\newcommand{\mb}{\mathbb}
\newcmd{\df}{1}{\textbf{\textit{\color{cyan!10!black} #1}}}
\newcommand{\dfeq}{=}

\newcommand{\eps}{\varepsilon}
\renewcommand{\phi}{\varphi}

\newcommand{\compl}[1]{{{#1}^{\text{c}}}}

\newcommand{\vv}{{\vvv}}
\DeclareMathOperator{\vvv}{v}

\DeclareMathOperator{\sign}{sign}

\DeclareMathOperator{\wei}{wt}%{radii}

\DeclareMathOperator{\real}{PsV}
\DeclareMathOperator{\realo}{PsG}
\DeclareMathOperator{\dist}{dist}
\DeclareMathOperator{\disth}{dist_h}

\DeclareMathOperator{\pg}{PsG}
\DeclareMathOperator{\pog}{Ps\widetilde G}
\DeclareMathOperator{\pe}{PsE}
\DeclareMathOperator{\g}{G}
\DeclareMathOperator{\og}{\widetilde G}
\DeclareMathOperator{\e}{E}
\let\oe\relax
\DeclareMathOperator{\oe}{\widetilde E}
\DeclareMathOperator{\poe}{Ps\widetilde{E}}

\DeclareMathOperator{\orth}{O}
\DeclareMathOperator{\sorth}{SO}
\DeclareMathOperator{\stief}{V}
\DeclareMathOperator{\pstief}{PsV}
\DeclareMathOperator{\sphere}{S}

\DeclareMathOperator{\cell}{cell}
\DeclareMathOperator{\Cell}{Cell}
\DeclareMathOperator{\conf}{conf}
\DeclareMathOperator{\upto}{upto}

\newcommand{\Rpol}[1]{\mb{R}\mathrlap{{}^{#1}}{}_\text{pol}}
\DeclareMathOperator{\pol}{pol}
\DeclareMathOperator{\aim}{aim}

\newcommand{\Ball}{\text{\rm Ball}}

\DeclareMathOperator{\proj}{proj}

\DeclareMathOperator{\interp}{interp}
\DeclareMathOperator{\coord}{coord}
\DeclareMathOperator{\ho}{ho}

\DeclareMathOperator{\id}{id}

\DeclareMathOperator{\cov}{cov}
\DeclareMathOperator{\omg}{om_G}
\DeclareMathOperator{\ompsg}{om_{PsG}}
\DeclareMathOperator{\ot}{ot}
\DeclareMathOperator{\mcp}{MacP}

\DeclareMathOperator{\curves}{PsL}

\newcommand{\projplane}{\mb{RP}^2}

\begin{document}
%\raggedright

\begin{center}
{\Large \noindent
Grassmannians and Pseudosphere Arrangements}

\bigskip

%\begin{spacing}{1.5}
{\large \noindent
Michael~Gene~Dobbins%\textsuperscript{1}
}
%\end{spacing}

%\smallskip

\begin{minipage}{0.85\textwidth}
\raggedright \footnotesize \singlespacing
\noindent
%\llap{\textsuperscript{1}}
Department of Mathematical Sciences, Binghamton University (SUNY), Binghamton, \newline New York, USA. 
\texttt{mdobbins@binghamton.edu} 
\end{minipage}

\end{center}

\bigskip

\begin{abstract}
We extend vector configurations to more general objects that have nicer combinatorial and topological properties, called weighted pseudosphere arrangements.
These are defined as a weighted variant of arrangements of pseudospheres, as in the Topological Representation Theorem for oriented matroids.
We show that in rank 3, the real Stiefel manifold, Grassmannian, and oriented Grassmannian are homotopy equivalent to the analogously defined spaces of weighted pseudosphere arrangements. 
As a consequence, this gives a new classifying space for rank 3 vector bundles and for rank 3 oriented vector bundles where the difficulties of real algebraic geometry that arrise in the Grassmannian can be avoided.
In particular, we show for all rank 3 oriented matroids, that the subspace of weighted pseudosphere arrangements realizing that oriented matroid is contractible. 
This is a sharp contrast with vector configurations where the space of realizations can have the homotopy type of any primary real semialgebraic set.
%This work is the first half of a project to show homotopy equivalence of Grassmannians and MacPhersonians in rank 3.
\end{abstract}

Math Subject Classification (2010) \ \ 52C30 \ \ 52C40 \ \ 14M15 \ \ 57R22

\section{Introduction}

If we record all the possible ways a given vector configuration or affine point set can be partitioned by a hyperplane, the resulting combinatorial representation will be an oriented matroid \cite{bjorner1999oriented}.
From this data, we can determine such information as, what points appear on the boundary of the convex hull of a point set, the faces of the resulting polytope, the solutions to a linear programming optimization problem, whether polytopes intersect, and the visibility between points around or on the boundary of a polytope.
Such representations can be used for proving theorems and developing algorithms for finite point sets in Euclidean space \cite{oswin2013,cardinal2018subquadratic,DobHolHub2014Erdos,habert2012,knuth1992,novick1,novick2,shnurnikov2016tk}.

Oriented matroids are more general objects, however, since not all oriented matroids arise from a vector configuration.  Indeed, it is $\exists \mb{R}$-complete, which is at least NP-hard, to determine whether a given oriented matroid can be realized by a vector configuration.  
In contrast, the Topological Representation Theorem says that any oriented matroid can be realized by a pseudosphere arrangement \cite{folkman1978oriented}.

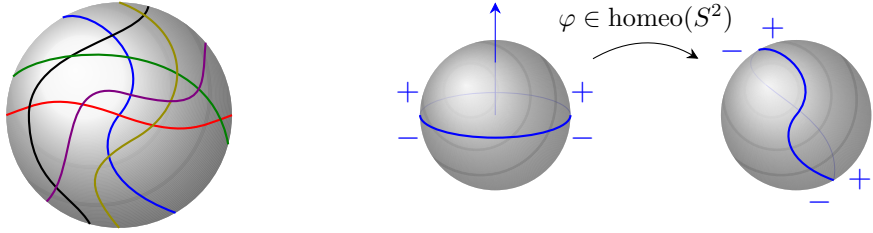
\begin{figure}
\begin{center}

\begin{tikzpicture}

\begin{scope}[rotate = 0]
\draw[blue!30]
(1,0) arc (0:180:1 and .3)
;\end{scope}

\draw[blue!50]
(0,0) -- (90:.7)
;

\shade[ball color=black!30, white, opacity=.3]
circle (1)
;

\draw[blue,->]
(90:.7) -- (90:1.5)
;

\begin{scope}[rotate = 0]
\draw[thick,blue]
(1,0) arc (0:-180:1 and .3)
;\end{scope}

\draw[blue]
(15:1.2) node {$+$} (-15:-1.2) node {$+$}
(-15:1.2) node {$-$} (15:-1.2) node {$-$}
;

\path (30:1.5) coordinate (a);

%\end{scope}

\begin{scope}[shift={(4cm,0cm)}]

\draw[blue!30]
(-60:1) .. controls +(80:1) and +(220:.7) .. (120:1)
;

\shade[ball color=black!30, white, opacity=.3]
circle (1)
;

\draw[thick,blue]
(-60:1) .. controls +(160:.7) and +(50:-.3) .. 
(0,0) .. controls +(50:.6) and +(20:.3) .. (120:1)
;

\draw[blue]
(-45:1.2) node {$+$} (105:1.2) node {$+$}
(-75:1.2) node {$-$} (-225:1.2) node {$-$}
;

%\node at (1.75,0) {$\theta \circ \varphi^{-1}$};

\path (-30:-1.5) coordinate (b);

\end{scope}

\path[->]
(a) edge[bend left] node[auto] {\footnotesize $\varphi \in \text{homeo}(S^2)$} (b)
;

\begin{scope}[shift={(-5cm,0cm)},scale=1.5]

\shade[ball color=black!5, white, opacity=.5]
circle (1)
;

\draw[thick,blue]
(-60:1) .. controls +(160:.7) and +(50:-.3) .. 
(0,0) .. controls +(50:.6) and +(20:.3) .. (120:1)
;

\draw[thick,red]
(0:1) .. controls +(200:.3) and +(-20:.6) .. 
(0,0) .. controls +(-20:-.6) and +(20:.3) .. (0:-1)
;

\draw[thick]
(75:1) .. controls +(260:.3) and +(90:.6) .. 
(0:-.8) .. controls +(90:-.6) and +(120:.3) .. (75:-1)
;

\draw[thick,green!50!black]
(-20:-1) .. controls +(45:.4) and +(100:1) .. (-15:1)
;

\draw[thick,olive]
(270:1) .. controls +(130:.5) and +(210:.5) .. 
(-55:.25) .. controls +(210:-.6) and +(-20:.7) .. (90:1)
;

\draw[thick,violet]
(40:1) .. controls +(85:-.4) and +(-15:.8) .. 
(80:.2) .. controls +(-15:-.6) and +(45:.4) .. (50:-1)
;

\end{scope}

\end{tikzpicture}
\end{center}

\caption{Left: a pseudosphere arrangement.  Right: a pseudosphere.}
\label{figure-pseudosphere}

\end{figure}

A pseudosphere arrangement is a topological analog to data representing the direction of each vector in a vector configuration; see Firgure \ref{figure-pseudosphere}.
More precisely, given a vector configuration in $\mb{R}^k$, we can define an oriented hyperplane arrangement by associating each vector to its orthogonal complement, and then define an oriented sphere arrangement by intersecting this hyperplane arrangement with the unit sphere.  From this oriented sphere arrangement, we know the direction of each vector.  Also, each cell of the sphere arrangement corresponds to a way of partitioning the given vector configuration by a hyperplane, and so the sphere arrangement defines an oriented matroid.  
A pseudosphere arrangement is a collection of oriented topological embeddings of $(k-2)$-spheres in the $(k-1)$-sphere satisfying certain conditions, and the cells of a pseudosphere arrangement define an oriented matroid; see Sections \ref{section-polar} and \ref{section-oriented-matroid}.
In this paper we deal with weighted pseudosphere arrangements, where each pseudosphere gets a weight analogous to the norm of a vector.

The hardness of deciding whether a given oriented matroid can be realized by a vector configuration is a consequence of Mnëv's Universality Theorem, which says that for any primary semialgebraic set $X$, there is a rank 3 oriented matroid $\mc{M}$, such that the quotient by isometry of the subspace of vector configurations realizing $\mc{M}$ has the homotopy type of $X$  \cite{mnev1988universality}.  
In other words, if you were hoping for the set of vector configurations that can be partitioned by a hyperplane in a certain fixed set of ways to define a `nice' topological space, then you may be disappointed to know that such a space can be as horrible as any space that you can define algebraically. 
The Topological Representation Theorem may be understood as saying that the space of realizations by pseudosphere arrangements is always at least non-empty.
Here we take this farther by showing that in contrast to Mnëv's theorem, the space of realizations by pseudosphere arrangements is always contractible up to isometry in rank 3; see Theorem~\ref{theorem-pseudoreal-contract}\footnote{This is actually proved for weighted pseudosphere arrangements, but here the weights play no role.}.

We also consider the space of \emph{all} $n$-element rank $k$ weighted pseudosphere arrangements, which we call the pseudolinear Stiefel manifold $\pstief_{k,n}$.
One way to define the real Stiefel manifold $\stief_{k,n}$ is as the space of all $n \times k$ matrices with orthonormal columns, or equivalently up to diffeomorphism, as the space of all spanning configurations of $n$ vectors in $\mathbb{R}^k$ considered up to symmetric positive definite linear transformation. 
In rank 3, we show that each pseudolinear Stiefel manifold is homotopy equivalent to the corresponding real Stiefel manifold, $\pstief_{3,n} \simeq \stief_{3,n}$.
Moreover, there is a natural embedding of $\stief_{3,n}$ in $\pstief_{3,n}$, and we provide a strong deformation retraction from $\pstief_{3,n}$ to $\stief_{3,n}$ that is equivariant with respect to the orthogonal group $\orth_3$; see Theorem \ref{theorem-grassman-pseudograssman}.
This holds even in the case where $n = \infty$.
We also consider the quotient of the pseudolinear Stiefel manifold by the orthogonal and special orthogonal groups, which we call the pseudolinear Grassmannians and oriented pseudolinear Grassmannians.  We show that these are homotopy equivalent to the corresponding real Grassmannians and oriented Grassmannians. 
That is, we the have the homotopy equivalences, $\g_{k,n} = \stief_{k,n} / \orth_k \simeq \pg_{k,n} = \pstief_{k,n} / \orth_k$ and $\og_{k,n} = \stief_{k,n} / \sorth_k \simeq \pog_{k,n} = \pstief_{k,n} / \sorth_k$.
This means that weighted pseudosphere arrangements effectively have the same global topology as vector configurations in rank 3. 

%Recall that the quotient space of $\g_{k,n} = \stief_{k,n} / \orth_k$ is the $(k,n)$ real Grassmannian, and $\og_{k,n} = \stief_{k,n} / \sorth_k$ is the $(k,n)$ real oriented Grassmannian. 
%We also consider the quotient spaces $\pg_{k,n} = \pstief_{k,n} / \orth_k$ and $\pog_{k,n} = \pstief_{k,n} / \sorth_k$, which we call the pseudolinear Grassmannian and oriented pseudolinear Grassmannian.
%As a consequence of Theorem \ref{theorem-grassman-pseudograssman}, we also have the homotopy equivalences $\pg_{3,n} \simeq \g_{3,n}$ and $\pog_{3,n} \simeq \og_{3,n}$.

These homotopy equivalences motivate the association of weights to pseudosphere arrangements. 
%We may regard a pseudosphere as indicating the direction of a vector and we add a weight to represent its norm.  
As a vector configuration moves along on a path though a Stiefel manifold, some of the vectors may pass through the origin.  For a weighted pseudosphere arrangement, this would correspond to a pseudosphere vanishing from the arrangement and then reappearing somewhere else.  Since this must happen continuously, we include weights that go to zero as a pseudosphere vanishes. 

One source of interest in Grassmannians is as classifying spaces for vector bundles.  
Recall that a real rank $k$ vector bundle is a space that locally has the structure of a product of $\mb{R}^k$ with a space $B$, called the base space; for a precise definition see \cite[page 24]{husemoller1994fibre}. 
The infinite Grassmannian $\g_{k,\infty}$ is a classifying space for rank $k$ vector bundles in the following sense. 
Every rank $k$ vector bundle with base space $B$ can be defined up to isomorphism by a map from $B$ to $\g_{k,\infty}$, and two maps define isomorphic vector bundles if and only if the maps are homotopy equivalent, provided that $B$ is paracompact.
We show that the infinite pseudolinear Grassmannian $\pg_{k,\infty}$ is a classifying space for rank 3 vector bundles, the infinite oriented pseudolinear Grassmannian $\pog_{k,\infty}$ is a classifying space for rank 3 oriented vector bundles, and the pseudolinear Stiefel manifold is universal for $\orth_3$ fiber bundles; see Corollary~\ref{corollary-classifying}.

%One way to obtain a vector bundle is as the quotient space $(X \times \mb{R}^k)/\orth_3$ where $X$ is a locally trivial principal $\orth_3$-bundle \cite[Capter~5]{husemoller1994fibre}.

Oriented Matroids have a natural partial order, and the poset of all rank $k$ oriented matroids on $n$ elements is the MacPhersonian, $\mcp_{k,n}$.
This work is part of a project by the author to resolve the MacPhersonian conjecture in rank 3, which says that the Grassmannian is homotopy equivalent to the polyhedral chain complex of $\mcp_{k,n}$.  %
One potential consequence of this conjecture is that we would have a representation of the homotopy type of the Grassmannian as a simplicial complex defined by purely combinatorial conditions.  
Another important consequence of the MacPhersonian conjecture is that we could represent a vector bundle over a simplicial complex as a poset map to the MacPhersonian in a way that gives a bijection between isomorphism classes of vector bundles and matroid bundles \cite{anderson2002mod,anderson1999matroid}.
This conjecture was motivated by the use of oriented matroids to compute Pontrjagin classes of triangulated manifolds \cite{gelfand1992combinatorial}.  An erroneous proof of the conjecture has appeared and was subsequently retracted \cite{biss2003homotopy,biss2009erratum}. 
The present paper does not deal with the MacPhersonian, but may be regarded intuitively as evidence for this conjecture.

%Here we show the analogous representation in the infinite pseudolinear Grassmannian $\pg_{3,\infty}$.
%That is, homotopy classes of maps from a paracompact space $B$ to $\pg_{3,\infty}$ are mapped bijectively as the pull-back bundle to isomorphism classes of rank 3 vector bundles over the base space $B$; see Corollary \ref{corollary-classifying}.

In the face of Mnëv's Universality Theorem, the MacPhersonian conjecture may seem overly optimistic for two reasons. 
First, we have a natural map $\omg$ from the Grassmannian $\g_{3,n}$ to the MacPhersonian $\mcp_{3,n}$ sending each vector configuration to its associated oriented matroid, 
but this map is not surjective.  Second, the preimage $\omg^{-1}(\mc{M})$ of an oriented matroid $\mc{M}$ can have the homotopy type of any primary semialgebraic set, even in rank 3.
That is, the Grassmannian may be decomposed into realizations of oriented matroids, but the resulting components have highly complicated topology, as do their intersections.  
The conjecture would suggest that we can simply ignore the topology of the components and their intersections, and this will have no bearing on the topology of the Grassmannian as a whole.
The MacPhersonian conjecture may already seem intuitively more reasonable in light of Theorem \ref{theorem-pseudoreal-contract} and Theorem \ref{theorem-grassman-pseudograssman}, which factor $\omg$ into maps $\iota$ and $\ompsg$ through the pseudolinear Grassmannian $\pg_{3,n}$ such that the following diagram commutes,
\[
\begin{tikzpicture}
\matrix (m) [matrix of math nodes, row sep=4em,
column sep=5em, text height=1.5ex, text depth=0.25ex]
{ 
& \pg_{3,n} &  \\
\g_{3,n} & & \mcp_{3,n} \\
};
\path[->]%,font=\scriptsize]
(m-1-2) edge node[auto] {$\ompsg$} (m-2-3)
(m-2-1) edge node[auto] {$\iota$} (m-1-2)
(m-2-1) edge node[auto] {$\omg$} (m-2-3)
;
\end{tikzpicture} 
\]
the map $\iota$ is a homotopy equivalence, the map $\ompsg$ is surjective, and for every rank 3 oriented matroid $\mc{M} \in \mcp_{3,n}$, the preimage $\ompsg^{-1}(\mc{M}) \subset \pg_{3,n}$ is contractible.
%What remains is to show that $\pg_{3,n}$ and $\mcp_{3,n}$ are homotopy equivalent.

Other obstacles may raise doubt on generalizing these results to rank higher than 3.  For example, the proofs of the main results here use the fact that the quotient by isometries of the space of self homeomorphisms of the 2-sphere is contractible \cite{kneser1926}, which may fail in higher dimensions \cite[Section 10.12]{kirby1977foundational}.
Also, extension spaces of higher rank oriented matroids can be disconnected, even in the realizable case \cite{liu2017counterexample}. 
The fact that universality holds for realizations of oriented matroids even when restricted to rank 3, however, shows how crucial the rank 3 case is.

Several ideas used in this paper were developed based 
on discussions with Laura Anderson and 
on the author's prior work with Andreas Holmsen and Alfredo Hubard on generalizing Mnëv's Universality Theorem to arrangements of convex sets in the plane \cite{DobHolHub2015realization}, which in turn stemmed from work generalizing the Erdős-Szekeres thoerem from point sets to arrangements of convex sets \cite{DobHolHub2014Erdos,DobHolHub2016regular}.

\section{Main definitions}

\subsection{Basic notation}

Here we briefly give some basic notation which is mostly standard, but may not be universally consistent in the literature. 
We use round brackets 
$(x_1,\dots,x_n) \in X^n$ to denote a sequence and curly brackets $\{x_1,\dots,x_n\} \subset X$ to denote a set.
To append an entry $y$ to sequence $I = (x_1,\dots,x_n)$, we write $I \cdot y = (x_1,\dots,x_{n},y)$. 
We use the notation $[n]_{\mb{N}} \dfeq \{1,\dots,n\}$ and $[a,b]_{\mb{R}} \dfeq \{x \in \mb{R} : a \leq x \leq b\}$,
with round brackets for strict inequalities. 
The unit sphere in $\mb{R}^k$ is denoted by $\sphere^{k-1} \dfeq \left\{ x\in \mb{R}^k : \|x\|^2 = 1 \right\}$, and the closed unit ball by $\Ball^k$. 
We call the intersection of a pointed convex cone in $\mb{R}^k$ with the sphere $\sphere^{k-1}$, a convex subset of the sphere, or we say a spherically convex set for emphasis.
For a convex subset $C$ of the sphere or the projective plane and points $x_i \in C$, we use $[x_1,\dots,x_n]_C$ to denote the polygonal path in $C$ with vertices $x_1,\dots,x_n$ in that order and geodesic edges. 
%The standard basis vectors in $\mb{R}^r$ are denoted $e_1,\dots,e_r$.

We denote the set of self homeomorphisms between topological spaces by 
\[\hom(X,Y) = \{ \phi: X \to Y :\ \phi \text{ is a homeomorphism }  \},\] 
and $\hom(X) \dfeq \hom(X,X)$. 
Let $\hom^+(\sphere^d)$ be 
the component of $\hom(\sphere^d)$ that contains the identity map,
i.e.\ positively oriented reparameterizations of the sphere.
$\orth_k$ denotes the orthogonal group of rank $k$, which we may regard as a subset of $\hom(\sphere^{k-1})$ or as a set of ($k\times k$)-matrices when convenient, and $\sorth_k$ denotes the special orthogonal group. 

%, and $\g_{k,n}$ with $k \leq n$ denotes the Grassmannian.

We use $x \mapsto \varphi$ for the function defined by substituting a value for $x$ into a formula $\varphi$.
We may implicitly curry functions, and we separate arguments by `;' to emphasis when this is done.  That is, given a function $f : X \times Y \to Z$, we denote by $f(c) : Y \to Z$ the function $y \mapsto f(c;y)$.
We denote partial functions by $f:X \not\to Y$.

\subsection{Vector configurations and Grassmannians}\label{subsec-vectconfig}

Let $X$ be a $k$ dimensional real vector space with an inner product $\langle,\rangle$, an orthonormal basis $e_1,\dots,e_k$, and unit sphere $\sphere(X) = \{u \in X: \|u\|=1\}$.  For the time being, we could regard $X$ to be $\mb{R}^k$, but we will mostly work in the vector space $X = \Rpol{k}$ defined later.

A vector configuration $(a_1,\dots,a_n)$ in $X$ is a \df{Parseval frame} when any of the following equivalent conditions are satisfied: 

\vbox{
\begin{itemize}
\item
For all $u \in \sphere(X)$, $\langle a_1,u \rangle^2 + \dots +\langle a_n,u \rangle^2 = 1$. 
\item
For all $x \in X$, $\langle a_1,x \rangle^2 + \dots +\langle a_n,x \rangle^2 = \|x\|^2$.
\item
The linear map $x \mapsto (\langle a_1,x\rangle,\dots,\langle a_n,x\rangle)$ from $X$ to $\mb{R}^n$ is an isometry.
\item
The $(n\times k)$-matrix $A$ with entries $A_{i,j} = \langle a_i, e_j\rangle$ has orthonormal columns. 
\end{itemize}
}

\begin{observation}
Each class of vector configurations spanning $X$ related by symmetric positive definite transformations contains a unique Parseval frame. 
\end{observation}

The Stiefel manifold and Grassmannian can be defined in several ways that are homeomorphic to each other. Here, we define the \df{Stiefel manifold}${}$ $\stief_{k,n} = \stief_{k,n}(X)$ for $k \leq n \in \mb{N}$ to be the space of Parseval frames of $X$ indexed by $[n]_\mb{N}$. 
We identify a Parseval frame that ends with a tail of zeros with the shorter frame where the trailing zero are removed so that $\stief_{k,k} \subset \stief_{k,k+1} \subset \stief_{k,k+2} \subset \cdots$,
and we define the infinite Stiefel manifold as the union of this ascending chain of spaces,
\[ \stief_{k,\infty} = \bigcup_{n=k}^\infty \stief_{k,n}. \]

%, and $\stief_{k,\infty} = \bigcup_{n\in \mb{N}} \stief_{k,n}$.
% We may drop the subscripts and simply write $\stief = \stief_{k,n}$ when the values of the parameters $k,n$ are understood.

We define the \df{Grassmannian}${}$ as the quotient by the orthogonal group $\g_{k,n} = \stief_{k,n}/\orth_k$.
We define the \df{oriented Grassmannian}${}$ as the quotient by the special orthogonal group $\og_{k,n} = \stief_{k,n}/\sorth_{k}$. 
Similarly, we define the infinite Grassmannian $\g_{k,\infty}$ and oriented Grassmannian $\og_{k,\infty}$ as the quotient of $\stief_{k,\infty}$, which is also the union the ascending chain of Grassmannians.

We define a metric on $\stief_{k,n}$ for $A = (a_1,\dots,a_n)$ and $B = (b_1,\dots,b_n) $ by
\[ \dist(A,B) \dfeq \max \{ \|a_i - b_i\| :\ i \in [n]_\mb{N} \}, \]
and we define a metric on $\mc{A},\mc{B} \in \g_{k,n}$ by
\[ \dist(\mc{A},\mc{B}) \dfeq \inf\{ \dist(A,B) : A \in \mc{A}, B \in \mc{B} \}. \]
%This is a metric, since orthogonal transformations are isometric. 

\begin{observation}
Since $\g_{k,n}$ and $\og_{k,n}$ are defined as the quotient of $\stief_{k,n}$ by a group of isometries, $\dist$ is a metric and the metric topology is the same as the quotient topology.
\end{observation}

On the infinite Stiefel manifold and infinite (oriented) Grassmannian, we use the direct limit topology, which is the finest topology such that the inclusion maps $\stief_{k,n} \hookrightarrow \stief_{k,\infty}$ are continuous.
Equivalently, this topology is defined by the universal property that a function $\phi : \stief_{k,\infty} \to Y$ to any topological space $Y$ is continuous if and only if the restriction of $\phi$ to $\stief_{k,n}$ is continuous for all $n \geq k$.

\begin{warning}
While $\dist$ defines a metric on $\stief_{k,\infty}$, the resulting metric topology is not the same as the direct limit topology on $\stief_{k,\infty}$, and likewise for the Grassmannians.  
\end{warning}

%The canonical bundle on the Grassmannian is then given by the quotient of the product of Parseval frames with $\mb{R}^k$ by the orthogonal group, with a projection to the Grassmannian.
%\[ \e_{k,n} \dfeq \left\{ \left\{(AQ,Q^*v)  :\ Q \in \orth_k \right\}:\ A \in \stief_{k,n}, v \in \mb{R}^k \right\}, \]
%\[ \gproj : \e_{k,n} \to \g_{k,n}, \quad \gproj(\mc{X}) \dfeq \{A \in \stief_{k,n} :\ \exists v.\ (A,v) \in \mc{X} \}, \]
%where the vector configurations are represented as the rows of a matrix $A$ as in the definition of Parseval frame above.  
%The canonical bundle on the oriented Grassmannian $\ogproj: \oe_{k,n} \to \og_{k,n}$ is defined similarly using $\sorth_k$.

%The vector bundle $\e_{k,n}$ is chosen so that for $Q \in \orth_k$, 
%$Av = (AQ)(Q^*v) = ( \langle a_1, v\rangle, \dots, \langle a_n, v \rangle ) $
%is the same vector for all pairs $(A,v) \in \mc{X}$ in the same class $\mc{X} \in\e_{k,n}$.

There are various equivalent ways the Grassmannian could be defined, depending on how a point in the Grassmannian is represented, and likewise for the Stiefel manifold.  
A point in $\stief_{k,n}$ is commonly represented as an orthonormal sequence of $k$ vectors in $\mb{R}^n$, i.e.\ the columns of the matrix $A$ in the last definition of a Parseval frame.
Likewise, a point in $\g_{k,n}$ may alternatively be represented as a $k$-dimensional vector subspace of $\mb{R}^n$, i.e.\ the column space of the matrix $A$.  Our choice of representation, the row vectors of $A$, is a matter of convenience.

\begin{justification}
For us it is more convenient to consider the rows of $A$ rather than the columns so that later we will be able to extend these vector configurations to a larger space defined by weighted pseudospheres.
We choose to represent the elements of our vector configuration as row vectors of $A$ so that the orthogonal group $\orth_k$ acts on the right.  Later $\orth_k$ will act by precomposition as a subspace of $\hom(\sphere^2)$, which makes this choice consistent with the convention that composition is written right to left.
\end{justification}

\subsection{Polar representation and pseudosphere arrangements}\label{section-polar}

We represent a vector $a \in \mb{R}^k$ by the pair $\pol(a) \dfeq (\|a\|,\aim (a))$ where 
\[ \aim (a) : \sphere^{k-1} \to \{+,0,-\}, \quad \aim (a; x) \dfeq \sign \langle a,x \rangle.\]
This is effectively a polar representation of $a$ where the direction of $a$ is represented by $\aim(a)$. 
We denote by $\Rpol{k} \dfeq \pol(\mb{R}^k)$ the vector space isomorphic to $\mb{R}^k$ by $\pol$, and inheriting the usual scalar multiplication, vector addition, inner product, norm, standard basis vectors, and the action of matrices as linear transformations.  Note that we do not have a nice formula for adding vectors $a,b \in \Rpol{k}$, other than $a+b = \pol(\pol^{-1}(a)+\pol^{-1}(b))$.

We define weighted pseudospheres as an extension of $\Rpol{k}$.  
A rank $k$ oriented \df{pseudosphere}${}$ is a map $\theta : \sphere^{k-1} \to \{+,0,-\}$ such that there is some orientation preserving self-homeomorphism $\phi \in \hom^+(\sphere^{k-1})$ such that $\theta \circ \phi = \aim (e_k)$; see Figure \ref{figure-pseudosphere} Left.  
We may simply call $\theta$ a pseudosphere, with the understanding that it is oriented and has some rank.  
%Let 
%\[ \param(\theta) = \{ \phi \in \hom^+(\sphere^{k-1}): \theta \circ \phi = \aim (e_k) \} \]
%be the set of all such homeomorphisms. 

%
A \emph{non-trivial} \df{weighted pseudosphere}${}$ is a pair $\alpha = (r,\theta)$ consisting of a positive real number $r>0$ and a pseudosphere $\theta$.  Additionally, there is the \emph{trivial} weighted pseudosphere $0 =\pol(0) = (0,\, x \mapsto 0)$, which is the origin of $\Rpol{k}$.
We let 
\[\|\alpha\| \dfeq r \quad \text{ and } \quad \aim (\alpha) \dfeq \theta.\]
We denote the kernel by $S_\alpha = S_\theta = \theta^{-1}(0)$.
We can scale weighted pseudospheres by $s \in \mb{R}$ by $s \alpha = (|s|r,\sign(s)\theta)$ for $s\neq 0$ and $0\alpha = 0$.

A \df{pseudosphere arrangement}${}$ is a sequence of pseudospheres $\Theta = (\theta_1,\dots,\theta_n)$ that satisfies the following.
For all $I \subseteq [n]_\mb{N}$, 
$S_I = \bigcap_{i \in I} S_{\theta_i} $ is either empty or a topological sphere, meaning there is a homeomorphism $\phi_I : \sphere^{k_I} \to S_I$ for some $k_I \leq k$, 
and if $S_I$ is non-empty then $(\theta_1 \circ \phi_I,\dots, \theta_1 \circ \phi_I)$ is again a pseudosphere arrangement; see Figure \ref{figure-pseudosphere} Right.
A \df{weighted pseudosphere arrangement}${}$ is a sequence of rank $k$ weighted pseudospheres ${A} = (\alpha_1,\dots,\alpha_n)$ such that $(\aim(\alpha_1),\dots,\aim(\alpha_n))$ is a pseudosphere arrangement. 
We may simply write $S_i = S_{\{i\}} = S_{\theta_i}$.
We say $A$ and $\Theta$ are \df{spanning}${}$ when $S_{1} \cap \dots \cap S_{n} = \emptyset$.  In other words, for every $x \in \sphere^{k-1}$, there is some $\alpha_i$ that does not vanish at $x$.

\subsection{Pseudolinear Grassmannians}

Throughout the rest of the paper, let the Stiefel manifold be $\stief_{k,n} = \stief_{k,n}(\Rpol{k})$, i.e.\ the space of Parseval frames in $\Rpol{k}$, and similarly let the Grassmannians and oriented Grassmannians consist of equivalence classes of Parseval frames in $\Rpol{k}$ as in Subsection \ref{subsec-vectconfig}.  We choose the vector space $\Rpol{k}$ for convenience in extending vector configurations to weighted pseudosphere arrangements.

We define the \df{pseudolinear Stiefel manifold}${}$ $\pstief_{k,n}$ to be the set of all rank $k$ spanning weighted pseudosphere arrangements indexed by $[n]_\mb{N}$.
We define a $\hom(\sphere^{k-1})$-action on $\pstief_{k,n}$ as follows.
For $A = (\alpha_1,\dots,\alpha_n) \in \pstief_{k,n}$ with $\alpha_i = (r_i,\theta_i)$, let
$\alpha_i * \psi \dfeq (r_i,\ \theta_i \circ \psi)$, and let 
$A*\psi \dfeq (\alpha_1*\psi,\dots,\alpha_n*\psi)$.  This includes an extension of the $\orth_k$-action on $\stief_{k,n}$ to $\pstief_{k,n}$ as $\orth_k \subseteq \hom(\sphere^{k-1})$ and for $a \in \mb{R}^k$, we have  
\[ \pol(a)*Q = (\|a\|,\, \aim (a) \circ Q) = \pol(Q^* a).\]  
%We can scale weighted pseudospheres by $s \in \mb{R}$ by $s \alpha = (|s|r,\sign(s)\theta)$ for $s\neq 0$ and $0\alpha = 0$, and we scale arrangements by $sA = (s\alpha_1,\dots,s\alpha_n)$.  
We say that $A$ is \df{symmetric} when $-A = A*(-\id)$.

The \df{pseudolinear Grassmannian}${}$ $\pg_{k,n}$ is the quotient by the orthogonal group, $\pg_{k,n} = \pstief_{k,n} / \orth_k$, and the \df{oriented pseudolinear Grassmannian}${}$ $\pog_{k,n}$ is the quotient by the special orthogonal group, $\pog_{k,n} = \pstief_{k,n} / \sorth_k$.

We extend the metrics on $\stief_{k,n}$, $\g_{k,n}$, and $\og_{k,n}$ to metrics on $\pstief_{k,n}$, $\pg_{k,n}$, and $\pog_{k,n}$ as follows.
We first define distance between weighted pseudospheres by a weighted analog of Fréchet distance.
For weighted pseudospheres $\alpha_i = (r_i,\theta_i)$ let
 \[ \dist(\alpha_1,\alpha_0) \dfeq \inf{}_{\phi_1, \phi_0} \sup{}_{x} \left\|  r_1 \phi_1(x) - r_0 \phi_0(x)  \right\| \]
where $ \phi_i \in \hom^+(\sphere^{k-1})$ such that $\theta_i \circ \phi_i = \aim(e_k)$
and $x \in \sphere^{k-1}$ such that $\langle e_k, x\rangle = 0$. 
Note that $\phi_i$ always exists by the definition of a pseudosphere, and also $\phi_i$ outside the equator of $\sphere^{k-1}$ has no bearing on $\dist$, so we may regard $\phi_i$ as a positively oriented parameterization of the kernel $S_{i}$.
For weighted pseudosphere arrangements ${A} = (\alpha_1,\dots,\alpha_n)$ and ${B} = (\beta_1,\dots,\beta_n) \in \pstief$, let
\[ \dist({A},{B}) \dfeq \max_{i\in [n]_{\mb{N}}} \dist(\alpha_i,\beta_i). \]
For a pair $\mc{A},\mc{B}$ in $\pg_{k,n}$ or in $\pog_{k,n}$, let 
\[ \dist(\mc{A},\mc{B}) \dfeq \inf\{ \dist(A,B):\ A \in \mc{A},\ B \in \mc{B} \}. \]

\begin{observation}
Since $\pg_{k,n}$ and $\pog_{k,n}$ are defined as the quotient of $\pstief_{k,n}$ by a group of isometries, $\dist$ is a metric and the quotient topology is the same as the metric topology.
\end{observation}

\begin{observation}
For $a,b \in \Rpol{k}$, we have $\dist(a,b) = \|a-b\|$. Hence, $\dist$ is an extension of the metrics on $\stief_{k,n}$, $\g_{k,n}$, and $\og_{k,n}$, and the subspace topology is the same as the metric topology.
\end{observation}

Again we identify weighted pseudosphere arrangements that only differ by a tail of all zeros, so that $\pstief_{k,k} \subset \pstief_{k,k+1} \subset \pstief_{k,k+2} \subset \cdots$, and we define spaces $\pstief_{k,\infty}$, $\pg_{k,\infty}$, $\pog_{k,\infty}$ as the union of the corresponding ascending chain of spaces with the direct limit topology.

We define the canonical bundles over the rank 3 pseudolinear Grassmannians to be the spaces
\[ \pe_{3,n} \dfeq (\pstief_{3,n} \times \mb{R}^3)/{\orth_3}  = \left\{ \{(A*Q, Q^* x) : Q \in \orth_3 \} : A \in \pstief_{3,n}, x \in \mb{R}^3 \right\} \]
with the projection map $ \xi_{3,n}: \pe_{3,n} \to \pg_{3,n} $ induced by $(A,x) \mapsto A$.
Similarly, we define the canonical bundles over the oriented pseudolinear Grassmannians to be 
$\poe_{3,n} \dfeq (\pstief_{3,n} \times \mb{R}^3)/{\sorth_3}$ with projetion $ \widetilde\xi_{3,n}: \poe_{3,n} \to \pog_{3,n} $.
We will prove that these are a fiber bundles, indeed vector bundles, in rank 3; see Lemma~\ref{lemma-principal-bundle}.

\subsection{Sign hyperfield-vector sets and chirotopes}\label{section-oriented-matroid}

For a weighted pseudosphere arrangement $A = (\alpha_1,\dots,\alpha_n) = ((r_1,\theta_1),\dots,(r_n,\theta_n))$, let
\begin{align*}
 \wei(A) &\dfeq (r_1,\dots,r_n) \in \mb{R}^n, \\
 \aim (A) &\dfeq \theta_1\times \dots\times \theta_n : \mb{R}^k \to \{+,0,-\}^n,  \\
 \cov (A) &\dfeq \left\{ \aim (A; v) :\ v \in \mb{R}^{k} \right\} \subseteq \{+,0,-\}^n.
\end{align*}
We call $\cov (A)$ the \df{covector set} of $A$. 
We call a set of sequences of signs $\mc{X}$ a \df{sign hyperfield-vector set} when $\mc{X}$ satisfies the vector axioms for oriented matroids,
or equivalently when $\mc{X}$ the set of covectors of an oriented matroid \cite{}. 
%Although it is used inconsistently in the liturature, we may also refer to $\mc{X}$ as an oriented matroid.

\begin{remark}
The Topological Representation Theorem essentially says that for every weighted pseudosphere arrangement $A$, $\cov(A)$ is a sign hyperfield-vector set, and every sign hyperfield-vector set $\mc{X}$ is realized by a weighted pseudosphere arrangement $A$ such that $\mc{X} = \cov(A)$ \cite{folkman1978oriented}.  
\end{remark}

A \df{basis} of $\mc{X}$ is a minimal subset $I \subset [n]_\mb{N}$ such that for every $\sigma \in \mc{X} \setminus 0$ there is $i \in I$ such that $\sigma(i) \neq 0$, and a subset of a basis is said to be \df{independent}. 
Equivalently, a set $\{i_1,\dots,i_m\} \subset [n]_\mb{N}$ is an independent set of $\mc{X}$ when 
\[ \{ (\sigma(i_1),\dots,\sigma(i_m)) : \sigma \in \mc{X} \} = \{+,0,-\}^m, \]
and a maximal independent set is a basis of $\mc{X}$.  
All bases have the same size, and this is called the \df{rank} of $\mc{X}$.

\begin{warning}
The independent sets, bases, and rank of a sign hyperfield-vector set $\mc{X}$ are the same as those of the oriented matroid with \emph{covector set} $\mc{X}$.
There is, however, a dual oriented matroid $\mc{M}^*$ that has $\mc{X}$ as its \emph{vector set}, and has an associated rank, bases, and independent sets that are different from what we use in this paper.
\end{warning}

\begin{justification}
A hyperfield is a generalization of a field, and their study has lead to a vast generalization of matroids, oriented matroids, and vector spaces as an analog of vector spaces where the scalar field is replaced with a hyperfield.  One hyperfield is the sign hyperfield, and the analog of a vector space in this case is a sign hyperfield-vector set \cite{anderson2019vectors,baker2019matroids}.
\end{justification}

The \df{\hbox{$\stief$-realization} space} of a rank $k$ sign hyperfield-vector set $\mc{X}$ on $[n]_\mb{N}$ is 
\[ \stief(\mc{X}) \dfeq 
\left\{ A \in \stief_{k,n} \ :\  \cov (A) = \mc{X} \right\},
\]
and the \df{$\pstief$-realization space} is 
\[ \pstief(\mc{X}) \dfeq 
\left\{ A \in \pstief_{k,n} \ :\  \cov (A) = \mc{X} \right\}.
\]
Similarly, we define the \df{$\g$-realization} and \df{$\pg$-realization spaces} by
\[ \g(\mc{X}) = \stief(\mc{X})/\orth_k \subset \g_{k,n}, \quad \pg(\mc{X}) = \pstief(\mc{X})/\orth_k \subset \pg_{k,n}. \]
We also regard these as realization spaces of the oriented matroid with covector set $\mc{X}$.

We order the set of sign sequences $\{+,0,-\}^n$ by the product of the relation $(\leq_\vv)$ where ${0 <_\vv ({+})}$, and ${0 <_\vv ({-})}$, and the pair $({+}),({-})$ are incomparable. 
This makes each sign hyperfield-vector set $\mc{X}$ a graded poset and the \df{dimension} of $\sigma \in \mc{X}$ is its height minus 1.  In particular, $0$ always has dimension $-1$.

For $\sigma \in \cov (A)$, let 
\[ \cell(A,\sigma) \dfeq \left\{u \in \sphere^{k-1}:\ \aim (A; u) = \sigma \right\}, \]
and $\Cell(A,\sigma)$ be the closure of $\cell(A,\sigma)$.

\begin{remark}\label{remark-cell}
Edmonds and Mandel have shown that the subdivision of $\sphere^{k-1}$ by a weighted pseudosphere arrangement is a regular cell decomposition.
In particular, the map $\sigma \mapsto \Cell(A,\sigma)$ is a poset isomorphism from $\cov(A)$ to the closed cells of $A$ ordered by inclusion, with $\Cell(A,0) = \emptyset$ \cite{edmonds1978topology}.
%Note that as a cell of this decomposition, the dimension of $\Cell(A,\sigma)$ is the same as that of $\sigma$, but as a subset of $\sphere^{k-1}$ the Hausdorff dimension of $\Cell(A,\sigma)$ might be higher.
\end{remark}

%Let $\dim_\mc{V}(\sigma) = ({}$poset rank of $\sigma$ in $\mc{V}) -1$, which is also the dimension of $\cell(A,\sigma)$.
%Let $\mc{V}_{\dim=d} = \{ \sigma \in \mc{V}:\ \dim_\mc{V}(\sigma) = d \}$.

%For a partial sign vector $\sigma \in \{+,0,-\}^I$ for some $I \subset [n]_\mb{N}$, we denote the corresponding subset of the sphere by
%\[ \cell(A,\sigma) \dfeq \left\{u \in \sphere^{k-1}:\ \forall i\in I.\  \aim (\alpha_i, u) = \sigma_i \right\}. \]

For $I \subseteq [n]_\mb{N}$, let $\proj_I(A) \dfeq (\beta_1,\dots,\beta_n)$ where 
\[\beta_i = \left\{ \begin{array}{ll} \alpha_i & i\in I \\ 0 & i \not\in I \end{array} \right. . \]
We say $I$ is a \df{basis} of $A$ when $|I|$ is the rank of $A$ and $\proj_I(A)$ is spanning, or equivalently when $I$ is a basis of $\cov(A)$.

%\begin{observation}
%If $I$ is a basis for $A$, then $\proj_I(A)$ subdivides the sphere into the cross-polytope decomposition.
%\end{observation}

We now come to realizations in the oriented (pseudolinear) Grassmannian.
We associate to each weighted pseudosphere arrangement a sign valued function on $k$-tuples of indices
\[ \ot (A) : [n]_\mb{N}^k \to \{+,0,-\}, \]
called the \df{order type} of $A$.  Unless $I = \{i_1,\dots,i_k\}$ is a basis for $A$, $\ot(A;i_1,\dots,i_k) = 0$.
If $I$ is a basis, then $C = \cell(\proj_I(A),\sign(e_{i_1}+\dots+e_{i_k}))$ is parameterized by a map $s : \Delta \to C$ from the standard $(k{-}1)$-simplex $\Delta$ such that the $j$-th facet of $s(\Delta)$ is contained in the $i_j$-th kernel $S_{{i_j}}$.
We define $\ot(A;i_1,\dots,i_k) \in \{+,-\}$ to be the orientation of $s$. 
Given a rank $k$ chirotope $\chi$ on $[n]_\mb{N}$, the $\stief$-realization and $\og$-realization spaces of $\chi$ are 
\[ \stief(\chi) \dfeq \left\{ A \in \stief_{k,n} \ :\  \ot (A) = \chi \right\}, \quad \og(\chi) \dfeq \stief(\chi)/\sorth_k \subset \og_{k,n}, \]
and realization spaces $\pstief(\chi)$ and $\pog(\chi)$ are defined similarly from $\pstief_{k,n}$.

\begin{remark}
Again by the Topological Representation Theorem, $\ot(A)$ is always a chirotope, and for every chirotope $\chi$, $\pstief(\chi)$ is non-empty. 
\end{remark}

\begin{justification}
Every sign hyperfield-vector set $\mc{X}$ corresponds to a pair of chirotopes $\{\chi,-\chi\}$, and we have $\pstief(\mc{X})/\sorth_k = \pog(\chi) \cup \pog(-\chi)$.
On the other hand, $\pstief(\chi)$ is not closed under the action of $\orth_k$, since $\ot(A*Q) = -\ot(A)$ for $Q \in \orth_k$ with $\det(Q) = -1$.
Hence, we use sign hyperfield-vector sets to represent elements of the Grassmannian, while we use chirotopes to represent elements of the oriented Grassmannian.
We deliberately prefer talking about ``sign hyperfield-vector sets'' to ``oriented matroids'' since these correspond to the case where we do \emph{not} have an orientation.  
\end{justification}

\section{Rank 3}

The main goal of this section is to prove the following two theorems. 

\begin{theorem}\label{theorem-pseudoreal-contract}
The $\pg$-realization space of every rank 3 oriented matroid (i.e.\ sign hyperfield-vector set) is contractible.
Also, the $\pog$-realization space of every rank 3 chirotope is contractible.
\end{theorem}

\begin{theorem}\label{theorem-grassman-pseudograssman}
For $n \in \{3,\dots\}$ or $n = \infty$, 
the pseudolinear Stiefel manifold   $\pstief_{3,n}$ strongly and $\orth_3$-equivariantly deformation retracts to the Stiefel manifold $\stief_{3,n}$.  Hence the pseudolinear Grassmannian $\pg_{3,n}$ strongly deformation retracts to the Grassmannian $\g_{3,n}$, and the pseudolinear oriented Grassmannian $\pog_{3,n}$ strongly deformation retracts to the oriented Grassmannian $\og_{3,n}$.
\end{theorem}

\newcounter{grassmanpseudograssmansection}
\setcounter{grassmanpseudograssmansection}{\value{section}}
\newcounter{grassmanpseudograssman}
\setcounter{grassmanpseudograssman}{\value{theorem}}

\begin{corollary}\label{corollary-classifying}
$\pstief_{3,\infty}$, $\pe_{3,\infty}$, and $\poe_{3,\infty}$ are respectively universal for principal $\orth_3$-bundles, rank 3 vector bundles, and oriented rank 3 vector bundles.
Hence, $\pg_{3,\infty}$ and $\pog_{3,\infty}$ are classifying spaces.
\end{corollary}

\subsection{Tools}

%We will start with some basic properties.
%Recall that a space $X$ dominates a subspace $Y \subseteq X$ when there is a map $d : X \to Y$ such that $d$ restricted to $Y$ is homotopic to the identity on $Y$.

We call pseudospheres in rank 3, \df{pseudocircles}.
We start by showing that for $\mc{A} \in \pg_{3,n}$  we can always fix $A \in \mc{A}$ up to orthogonal transformation by fixing a coordinate system defined in terms of the pseudocircles in an $\orth_3$ invariant way.  Essentially, we pick three independent elements of $A$ to define a basis in $\mb{R}^3$.
To this end we will define a partial function 
\[\coord : [n]_\mb{N}^3 \times \pstief_{3,n} \not\to \orth_3, \]
defined for pairs $(I,A)$ where $I$ is a basis of $A$, that satisfies the following lemma.

\begin{lemma}\label{lemma-basis}%\n
For all $I = (i_1,i_2,i_3)$ of distinct indices in $[n]_\mb{N}$ and $Q \in \orth_3$, we have 
\begin{enumerate}
\item
$ \coord(I) : \{A \in \pstief_{3,n} : I \text{ is a basis of } A \} \to \orth_3 $
is continuous,
\item 
$\coord(I;A*Q) =Q^*\coord(I;A)$, 
\item
$ \coord((1,2,3);(e_1,e_2,e_3)) = \id$,
\item
$\coord(I;A*\coord(I;A)) = \id$.
\end{enumerate}
\end{lemma}

For $A\in \pstief_{3,n}$ and an ordered basis $I = (i_1,i_2,i_3)$ of $A$, 
let $\coord(I;A)\in \orth_3$ be given by the matrix with columns $(u_1,u_2,u_3)$ defined as follows.
Let $p_{\pm k}$ be the vertex $\cell(\proj_I(A),\, \pm \sign (e_{i_k}))$.
If $p_{- 1} = -p_{1}$, then let 
\begin{align*}
u_1 &= p_1, \\ 
u_2 &= \tilde u_2/ \|\tilde u_2\| \text{ where } \tilde u_2 = \proj_{u_1^{\perp}}(p_2) = {p_2 -\langle u_1, p_2 \rangle u_1}, \\
u_3 &= \ot(A;i_1,i_2,i_3) ( u_1 \times u_2).
\end{align*}
Otherwise, we define a map $\varphi \in \hom(\sphere^2)$ that sends $-p_{1}$ to $p_{-1}$ and let $\coord(I;A) = \coord(I;A * \varphi)$.
%Let $\varphi$ be defined as follows.  
Choose a polar coordinate system for $\mb{R}^3$ so that the 1st and 2nd coordinates are angle and radius in the plane spanned by $p_1 =(1,0,0)$ and $p_{-1} = (1,\omega,0)$ with $\omega \in (0,\pi)_\mb{R}$ and the 3rd coordinate is offset from this plane.  Define $\varphi$ by
\[ \varphi(r,\theta,h) = \begin{cases} (r,\theta \frac{\omega}{\pi},h) & \theta \in [0,\pi]_\mb{R} \\ (r,\theta (2-\frac{\omega}{\pi}),h) & \theta \in [-\pi,0]_\mb{R}. \end{cases} \]
%for $\psi \in [0,\pi]_\mb{R}$ by $\varphi(r,\psi,h) = (r,\psi \omega/\pi,h)$ and  
%for $\psi \in [-\pi,0]_\mb{R}$ by $\varphi(r,\psi,h) = (r,-\psi (\omega -2\pi)/\pi,h)$. 

\begin{remark}
The points $p_{1},p_{-1}$ are where $S_{i_2}$ and $S_{i_3}$ meet, and assuming these are antipodal, $Q = \coord(I;A)$ is defined as the orthogonal transformation that sends $e_{1}$ to $p_1$, and sends $e_2$ into the half-plane extending from the line though $p_1$ in the direction of $p_2$, and orients the sphere so that $S_{i_1},S_{i_2},S_{i_3}$ appear counterclockwise in that order around the boundary of the triangular cell that is on the positive side of all three curves.
When $p_{1},p_{-1}$ are not antipodal, we first deform the sphere to make these points antipodal to ensure that $p_1$ and $p_2$ are linearly independent.
Note also that $p_3$ may be in the plane spanned by $p_1,p_2$, so we do not use $p_3$ to find $u_3$.
\end{remark}

\begin{lemma}\label{lemma-vertex}
If $\Theta_k = (\theta_{k,1},\theta_{k,2},\theta_{k,3})$ is a spanning pseudocircle arrangement for $k\in \{1,\dots,\infty\}$ and $\Theta_k \to \Theta_\infty$, then $\cell(\Theta_k,e_1) \to \cell(\Theta_\infty,e_1)$.

Moreover, if $\{S_{k,1}, S_{k,2}\}$ is an unoriented pseudocircle arrangement for $k\in \{1,\dots,\infty\}$, and $S_{k,i} \to S_{\infty,i}$ in Fréchet distance, and $S_{\infty,1}\neq S_{\infty,2}$, then $(S_{k,1}\cap S_{k,2}) \to (S_{\infty,1}\cap S_{\infty,2})$ in symmetric Hausdorff distance.
\end{lemma}

\begin{proof}
Since $\Theta_k$ is spanning, the point $p_k = \cell(\Theta_k,e_1) \in \sphere^2$ is well defined.
Since $\sphere^2$ is compact, there is a convergent subsequence $p_{k_j} \to q$.
Since $\theta_{k,i} \to \theta_{\infty,i}$ for $i \in \{2,3\}$, $S_{k,i} \to S_{\infty,i}$ in Fréchet distance, so there are maps $\psi_{k,i}:S_{k,i} \to S_{\infty,i}$ such that $\forall x  \in \sphere^2$, $\hbox{$\|\psi_{k,i}(x)-x\|$} \leq \eps_k \to 0$.  In particular, this holds for $x=p_{k_j}$, so $\| \psi_{k_j,i}(p_{k_j}) -q\| \leq \hbox{$\eps_k + \|p_k -q\|$} \to 0$, 
so $\psi_{k_j,i}(p_{k_j}) \to q$, which implies $q \in (S_{\infty,2} \cap S_{\infty,3})$.
We just have to show that $\theta_{\infty,1}(q)=+$.

Since $p_\infty = \cell(\Theta_\infty,e_1)$ is bounded away from $\theta_{\infty,1}^{-1}\{0,-\}$, and $p^-_\infty = \cell(\Theta_\infty,-e_1)$ is bounded away from $\theta_{\infty,1}^{-1}\{0,+\}$, and $q \in (S_{\infty,2} \cap S_{\infty,3}) = \{p_\infty,p^-_\infty\}$, there is some $\delta > 0$ such that $\forall x \in \sphere^2$ if $\|x-q\| \leq \delta$ then $\theta_{\infty,1}(x) = \theta_{\infty,1}(q)$.
Also, since $\theta_{k,1} \to \theta_{\infty,1}$ there are maps $\psi_{k,1} \in \hom^+(\sphere^2)$ such that $\theta_{k,1} = \theta_{\infty,1} \circ \psi_{k,1}$ and $\forall x \in \sphere^2$, $\|\psi_{k,1}(x) -x\| \leq \eps'_k \to 0$.  Hence, if $\|x-q\| \leq \delta -\eps'_k$ then $\|\psi_{k,1}(x)-q\| \leq \delta$, so $\theta_{k,1}(x) = \theta_{\infty,1}(q)$.
For $j$ sufficiently large, we have $\|p_{k_j} -q\| \leq \delta -\eps'_{k_j}$, so $\theta_{\infty,1}(q) = \theta_{k_j,1}(p_{k_j}) = +$.  Thus, $q = p_\infty$.

Suppose that $p_k$ does not converge to $p_\infty$, then there is some other subsequence that is bounded away from $p_\infty$ but has a subsubsequence that converges by compactness.  By the same argument above this subsubsequence must converge to $p_\infty$, which is a contradiction.  Thus, $p_k \to p_\infty$

The second part of the lemma follows by a similar argument.
\end{proof}

\begin{comment}
Since $S_{\infty,1}\neq S_{\infty,2}$, $S_{\infty,1}\cap S_{\infty,2}$ is a pair of points. %$\{p_\infty,q_\infty\}$.
For $k$ large enough, we have $S_{k,1}\neq S_{k,2}$, so $S_{k,1}\cap S_{k,2} = \{p_k,q_k\}$ is a pair of points.
Since the space of pairs of points in $\sphere^2$ is compact, there is a subsequence $\{p_{k_j},q_{k_j}\}$ that converges to some pair $\{p_\infty,q_\infty\}$. 
We may label these points so that $p_{k_j} \to p_\infty$ and $q_{k_j} \to q_\infty$.
Moreover, we label each pair $\{p_k,q_k\}$ so that 
$\|p_k -p_\infty\|+\|q_k -q_\infty\| \leq \|q_k -p_\infty\|+\|p_k -q_\infty\|$.

Since $S_{k,i} \to S_{\infty,i}$ in Fréchet distance, there are maps $\psi_{k,i}:S_{k,i} \to S_{\infty,i}$ such that $\hbox{$\|\psi_{k,i}(x)-x\|$} \to 0$ uniformly.  In particular this holds for $x=p_{k_j}$ and for $x=q_{k_j}$, so $\psi_{k_j,i}(p_{k_j}) \to p_\infty$ and $\psi_{k_j,i}(q_{k_j}) \to q_\infty$, which implies that $\{p_\infty,q_\infty\} = (S_{\infty,1}\cap S_{\infty,2})$.

Suppose that $\{p_k,q_k\}$ does converge.  Then, there is some other subsequence $\{p_{k'{}_j},q_{k'{}_j}\}$ that is bounded away from $\{p_\infty,q_\infty\}$, and again by compactness, we may assume that this subsequence converges, but by the same argument as above, $\{p_{k'{}_j},q_{k'{}_j}\}$ converges to $\{p_\infty,q_\infty\}$, which is a contradiction.
\end{comment}

\begin{proof}[Proof of Lemma \ref{lemma-basis}]
%For \ref{item-basis}, 
%Consider $A\in \pstief_{3,n}$ and $I = (i_1,i_2,i_3) \subset [n]_\mb{N}$ such that $\tilde A = \proj_I(A)$ is spanning.

We have immediately that $\coord((1,2,3);(e_1,e_3,e_3))$ is the identity from the definition.

For any $Q \in \orth_3$ and any $\sigma \in \cov(A)$, we have 
\begin{align*}
\cell(A*Q, \sigma) &= \{x : \aim(A*Q;x) = \sigma\} \\
 &= \{x : \aim(A;Qx) = \sigma\}  \\
 &= \{Q^*y : \aim(A;y) = \sigma\} 
 = Q^*\cell(A,\sigma).
\end{align*}
In particular, $A \mapsto A*Q$ sends $p_i$ to $Q^*p_i$, and so sends $u_i$ to $Q^*u_i$.
Thus, ${\coord(I;A*Q) =} Q^*\coord(I;A)$.
Furthermore, if $Q = \coord(I;A)$, then $\coord(I;A*Q) =Q^*Q=\id$.

Finally, the vertices $p_{\pm 1}$ and $p_{2}$ depend continuously on $A$ by Lemma \ref{lemma-vertex}, and $u_1,u_2$ depend continuously on these vertices, and $u_3$ depends continuously on $u_1,u_2$ up to change of sign, which is constant on connected components of the domain of $\coord(I)$.  Thus, $\coord(I)$ is continuous.
\end{proof}

Let $\disth$ be the sup metric on $\hom(\sphere^2)$, 
and let $\|\psi\| = \disth(\psi,\id)$.
For each rank 3 sign hyperfield-vector set $\mc{X}$, this gives a topologically equivalent metric on $\pstief(\mc{X})$ by 
\[\disth(A,B) = \|\wei(A)-\wei(B)\| + \inf \{ \|\psi\| :\ \psi \in \hom(\sphere^2) ,\ A*\psi = B \}.  \]

\begin{theorem}\label{theorem-disth}
For every rank 3 sign hyperfield-vector set $\mc{X}$, the metrics $\dist$ and $\disth$ induce the same topology on $\real(\mc{X})$.
\end{theorem}

\begin{warning}%Beware! 
$\disth$ does not extend to a metric on $\pstief_{3,n}$ that is topologically equivalent to $\dist$, except for the case $n=3$.  Indeed, there are sign hyperfield-vector sets $\mc{X} \neq \mc{Y}$ and sequences $A_k \in \real(\mc{X})$ where $A_k \to A_\infty \in \real(\mc{Y})$ with respect to $\dist$, but $A_k$ diverges with respect to $\disth$.
\end{warning}

To prove Theorems \ref{theorem-pseudoreal-contract} and \ref{theorem-grassman-pseudograssman} will make use of the following.

\begin{theorem}[Kneser 1926 \cite{Kneser1926deformationssatze}, see also \cite{friberg1973topological}]\label{theorem-kneser}
% (Kneser 1926 \cite{Kneser1926deformationssatze}, see also \cite{friberg1973topological})
There is a strong\ ~$\orth_3$-equivariant deformation retraction %$\ho$
\[\ho : \hom(\sphere^2) \times [0,1]_\mb{R} \to \hom(\sphere^2)\] 
from homeomorphisms of the 2-sphere $\hom(\sphere^2)$ to the orthogonal group $\orth_3$,
where $Q \in \orth_3$ acts by precomposition, i.e.\ $fQ = f \circ Q$.
\end{theorem}

\begin{observation}\label{observation-kneser-equivariant}
Although the deformation in \cite{friberg1973topological} is not $\orth_3$-equivariant, 
once we have a strong deformation retraction from $\hom(\sphere^2)$ to $\orth_3$ the deformation can be modified to be equivariant using Lemma \ref{lemma-basis}
\end{observation}

%Kneser proved Theorem \ref{theorem-kneser} without the equivariant condition \cite{Kneser1926deformationssatze}, see also \cite{friberg1973topological}.  
%The $\mb{Z}_2$-equivariant deformation is given below.
%This is equivalent to a strong deformation retraction from homeomorphisms of the real projective plane $\hom(\projplane)$ to $\text{SO}(3)$.

\begin{proof}[Proof of Observation \ref{observation-kneser-equivariant}]
%(Proof of Theorem \ref{theorem-kneser} (equivariant version))
Let $\kappa$ be a strong deformation retraction from $\hom(\sphere^2)$ to $\orth(3)$.  
For $f \in \hom(\sphere^2)$ let 
\[ A(f) = ((1,\aim(e_1)\circ f), (1,\aim(e_2)\circ f), (1,\aim(e_3)\circ f)),
\quad
\gamma(f) = \coord((1,2,3);A(f)) \]

For $Q \in \orth_3$ we have $\gamma(f\circ Q) = Q^{-1} \circ \gamma(f)$ by Lemma \ref{lemma-basis}.  Furthermore, if $f \in \orth_3$, then $\gamma(f) = f^{-1} \circ \gamma(\id) = f^{-1}$.

Let
\[ \ho(f,t) = \kappa(f \circ \gamma(f),t) \circ \gamma(f)^{-1}. \]
We have 
\begin{align*}
\ho(f \circ Q,t)
&= \kappa(f \circ Q \circ \gamma(f \circ Q),t) \circ \gamma(f \circ Q)^{-1} \\
&= \kappa(f \circ Q \circ  Q^{-1} \circ \gamma(f),t) \circ (Q^{-1} \circ \gamma(f))^{-1}  \\
&= \kappa(f \circ \gamma(f),t) \circ \gamma(f)^{-1} \circ Q \\
&= \ho(f ,t) \circ Q. 
\end{align*}
Thus, $\ho$ is equivariant with respect to the action of $\orth_3$ by precomposition.
To check $\ho$ is a strong deformation retraction, observe 
\[\ho(f,0) = f \circ \gamma(f) \circ \gamma(f)^{-1} = f, \]
\[\ho(f,1) = \kappa(f \circ \gamma(f),1) \circ \gamma(f)^{-1} \in \orth_3, \]
and if $f \in \orth_3$, then
\[\ho(f,t) = \kappa(f \circ f^{-1},t) \circ f = f. \qedhere \]
\end{proof}

\begin{question}
Note that we could alternatively make a deformation that is $\orth_3$-equivariant with respect to postcomposition in Theorem~\ref{theorem-kneser} by considering the inverse of the map being deformed, but we cannot do this on both sides at the same time.  Is there a deformation that is equivariant with respect to both postcomposition and precomposition?
\end{question}

\begin{theorem}[Radó 1923 \cite{rado1923representation}, reformulated]\label{theorem-conformal}
Fix $u,v,w \in \sphere^1$. 
Let $S_k$ be a simple closed curve in $\sphere^2$ and $a_k,b_k,c_k \in S_k$ distinct for each $k \in \{1,2,\dots,\infty\}$ such that $S_k \to S_\infty$ in Fréchet distance and $a_k \to a_\infty$, $b_k \to b_\infty$, $c_k \to c_\infty$. 
Then, there is a unique homeomorphism $h_k$ for each $k$ from the closed unit disk to the closed region bounded by $S_k$ that is conformal on the interior of the disk and sends $u,v,w$ 
respectively to $a_k,b_k,c_k$.
Furthermore, $h_k$ converges uniformly to $h_\infty$.
\end{theorem}

The above theorem is just useful to the present paper as a variation of the Canonical Schoenflies Theorem in dimension 2 \cite{gauld1971canonical}.  We do not make use of conformality.  An important distinction is that the Canonical Schoenflies Theorem provides a parameterization of a Jourdan region that depends on an embedding of a cylinder, which is to say that it depends on a parameterization of a pseudocircle along with a collar neighborhood of that pseudocircle.  In higher dimensions this may be necessary, but since we are working only in $\sphere^2$, we can make use of a parameterization that depends only on a pseudocircle and the image of 3 points on the pseudocircle.
Another important feature of this variation is that it is $\orth_3$-equivariant.

Theorem \ref{theorem-conformal} is essentially a reformulation of a theorem of Tibor Radó \cite{rado1923representation}, see also \cite[Section II.5 Theorem 2]{goluzin1969geometric}.  Briefly, this reformulation is as follows.  
The Riemann Mapping Theorem implies that every simply connected open region in $\sphere^2$ is conformally equivalent to the open unit disk. Carathéodory showed that if the region is bounded by a simple closed curve $S \subset \sphere^2$, then this conformal map extends to a homeomorphism $h$ from the closed unit disk to the closure of the region \cite{caratheodory1913gegenseitige}.  Radó showed that if simple closed curves $S_k$ have parameterizations that converge uniformly, then there exist such maps $h_k$ that converge uniformly \cite{rado1923representation}.  Finally, the image of 3 distinct points on the unit circle determines a conformal automorphism of the unit disk depending uniformly continuously on the 3 points.  One way to see this is as follows.  Conformal automorphisms of the unit disk correspond bijectively via conjugation by the map $z \mapsto \frac{z+i}{iz+1}$ to conformal automorphisms of the upper half plane of $\mb{C}$, 
which are precisely the complex extensions of projective automorphisms of the real projective line.  Projective automorphisms of the line are determined bijectively by the image of any 3 points.  
Specifically, the automorphism $\varphi$ that respectively sends $0,1,\infty$ to the distinct values $x_0,x_1,x_\infty \in \mb{R}\cup\{\infty\}$ is given by the linear fractional transformation 
$\varphi(z) = \frac{(z-x_0)(x_1-x_\infty)}{(z-x_\infty)(x_1-x_0)}$,  which is also the cross-ratio $(z,x_1;x_0,x_\infty)$. 
%Therefore, we may assume $h_k$ sends $u,v,w$ to $a_k,b_k,c_k$.  Otherwise, compose $h_k$ with the conformal automorphism that sends $u,v,w$ to the preimages of $a_k,b_k,c_k$.

\subsection{Interpolation}

Given a rank 3 sign hyperfield-vector set $\mc{X}$, we define a map 
\[ \interp : \real(\mc{X}) \times \real(\mc{X}) \to \hom(\sphere^2) \]
such that $A*\interp(A,B) = B$. 
Equivalently, for each $\sigma \in \mc{X}$, $\interp(A,B;\cell(B,\sigma)) = \cell(A,\sigma)$.

To define $\interp$, we use Theorem \ref{theorem-conformal} to parameterize the 2-cells of a given pseudocircle arrangement. 
We define $\interp(A,B)$ first on 0-cells then 1-cells then 2-cells; see Figure \ref{figure-interp}.  Let $\mc{V}_d$ be the set of all $d$-dimensional sign hyperfield-vectors of $\mc{X}$.  Equivalently, $\mc{V}_d$ is the set of all $\sigma$ such that $\Cell(A;\sigma)$ is a $d$-dimensional cell of $A$.
Note that as a subset of the sphere, the Hausdorff dimension of $\Cell(A;\sigma)$ might be higher than $d$.

\begin{figure}

\begin{center}
\begin{tikzpicture}

\begin{scope}[shift={(-4.3cm,-.2cm)},scale=1.3]

\shade[ball color=black!5, white, opacity=.5]
circle (1)
;

\path 
(20:1.2) coordinate (X)
(270:1.2) coordinate (X1)
;

\draw[thick,blue]
(-60:1) .. controls +(160:.7) and +(50:-.3) .. 
(0,0) .. controls +(50:.6) and +(20:.3) .. (120:1)
;

\draw[thick,red]
(0:1) .. controls +(200:.3) and +(-20:.6) .. 
(0,0) .. controls +(-20:-.6) and +(20:.3) .. (0:-1)
;

\draw[thick]
(75:1) .. controls +(260:.3) and +(90:.6) .. 
(0:-.8) .. controls +(90:-.6) and +(120:.3) .. (75:-1)
;

\draw[thick,green!50!black]
(-20:-1) .. controls +(45:.4) and +(100:1) .. (-15:1)
;

\draw[thick,olive]
(270:1) .. controls +(130:.5) and +(210:.5) .. 
(-55:.25) .. controls +(210:-.6) and +(-20:.7) .. (90:1)
;

\draw[thick,violet]
(40:1) .. controls +(85:-.4) and +(-15:.8) .. 
(80:.2) .. controls +(-15:-.6) and +(45:.4) .. (50:-1)
;

%\draw[thick,fill=black]
%(0,0) circle (1pt)
%;

\begin{scope}

\clip 
(50:-1) arc (230:40:1)
 .. controls +(85:-.4) and +(-15:.8) .. 
(80:.2) .. controls +(-15:-.6) and +(45:.4) .. (50:-1)
;

\clip
(75:1) .. controls +(260:.3) and +(90:.6) .. 
(0:-.8) .. controls +(90:-.6) and +(120:.3) .. (75:-1)
arc (-105:75:1)
;

\clip
(-20:-1) .. controls +(45:.4) and +(100:1) .. (-15:1)
arc (-15:-200:1)
;

\clip
(-60:1) .. controls +(160:.7) and +(50:-.3) .. 
(0,0) .. controls +(50:.6) and +(20:.3) .. (120:1)
arc (120:300:1)
;

%\only<13-17>{
%\draw[line width=2.5pt,densely dotted,red]
%(0:1) .. controls +(200:.3) and +(-20:.6) .. 
%(0,0) .. controls +(-20:-.6) and +(20:.3) .. (0:-1)
%;}

\clip 
(0:-1) arc (180:0:1)
 .. controls +(200:.3) and +(-20:.6) .. 
(0,0) .. controls +(-20:-.6) and +(20:.3) .. (0:-1)
;

\fill[orange,opacity=0.3]
(0,0) circle (1);

\end{scope}

\draw[thick,fill=white]
(81:.52) circle (1pt)
;

\draw[thick,fill=yellow]
%(128:.67) circle (1pt)
(175:.8) circle (1pt)
;

\draw[thick,fill=orange]
(165:.34) circle (1pt)
;

\end{scope}

\begin{scope}[shift={(-.2cm,-.2cm)},scale=1.3, rotate=-60]

\shade[ball color=black!5, white, opacity=.5]
circle (1)
;

\begin{scope}[rotate=60]
\path 
(160:1.2) coordinate (Y)
(270:1.2) coordinate (Y1)
;
\end{scope}

\begin{scope}[rotate = 0]
\draw[thick,blue]
(1,0) arc (0:180:1 and .3)
;\end{scope}

\begin{scope}[rotate = 42]
\draw[thick,green!50!black]
(1,0) arc (0:180:1 and .5)
;\end{scope}

\begin{scope}[rotate = 144]
\draw[thick]
(1,0) arc (0:180:1 and .6)
;\end{scope}

\begin{scope}[rotate = 256]
\draw[thick,red]
(1,0) arc (0:180:1 and .4)
;\end{scope}

\begin{scope}[rotate = 288]
\draw[thick,violet]
(1,0) arc (0:180:1 and .35)
;\end{scope}

\begin{scope}
\clip
(1,0) arc (0:180:1 and .3) arc (180:360:1);
\begin{scope}[rotate = 42]
\clip
(1,0) arc (0:180:1 and .5) arc (180:360:1);
\begin{scope}[rotate = 102]
\clip
(1,0) arc (0:180:1 and .6) arc (180:360:1);
\begin{scope}[rotate = 144]
\clip
(1,0) arc (0:180:1 and .35) arc (180:360:1);
\begin{scope}[rotate = -32]

%\draw[line width=2.5pt,densely dotted,red]
%(1,0) arc (0:180:1 and .4) arc (180:360:1)
%;
\clip
(1,0) arc (0:180:1 and .4) arc (180:360:1);

\fill[orange,opacity=0.3]
(0,0) circle (1);

\end{scope}
\end{scope}
\end{scope}
\end{scope}
\end{scope}

\draw[thick,olive]
(170:1) .. controls +(60:.3) and +(0:-.3) .. 
(80:.45) .. controls +(0:.7) and +(80:.3) .. (-20:1)
;

%\draw[thick,fill=black]
%(0,0) circle (1pt)
%;

\draw[thick,fill=white]
(149:.51) circle (1pt)
;

\draw[thick,fill=yellow]
%(189:.72) circle (1pt)
(280:.73) circle (1pt)
;

\draw[thick,fill=orange]
(166:-.4) circle (1pt)
;

\end{scope}

\begin{scope}[shift={(-4.3cm,-3.2cm)},scale=.8]

\fill[green!50!black!25] (0,0) -- (0:1) arc (0:70:1);
\fill[black!25] (0,0) -- (70:1) arc (70:120:1);
\fill[red!25] (0,0) -- (120:1) arc (120:240:1);
\fill[violet!25] (0,0) -- (240:1) arc (240:330:1);
\fill[blue!25] (0,0) -- (330:1) arc (330:360:1);

\draw[ultra thick,green!50!black] (0:1) arc (0:70:1);
\draw[ultra thick] (70:1) arc (70:120:1);
\draw[ultra thick,red] (120:1) arc (120:240:1);
\draw[ultra thick,violet] (240:1) arc (240:330:1);
\draw[ultra thick,blue] (330:1) arc (330:360:1);

%\draw[red,->] (130:1.2) arc (130:230:1.2);

\draw[thick,fill=white]
(0:1) circle (1.5pt)
;

\draw[thick,fill=yellow]
(120:1) circle (1.5pt)
;

\draw[thick,fill=orange]
(240:1) circle (1.5pt)
;

\path
(90:1.2) coordinate (X2)
(0:1.2) coordinate (X3)
;

\end{scope}

\begin{scope}[shift={(-.2cm,-3.2cm)},scale=.8]

\fill[green!50!black!25] (0,0) -- (0:1) arc (0:30:1);
\fill[black!25] (0,0) -- (30:1) arc (30:120:1);
\fill[red!25] (0,0) -- (120:1) arc (120:240:1);
\fill[violet!25] (0,0) -- (240:1) arc (240:270:1);
\fill[blue!25] (0,0) -- (270:1) arc (270:360:1);

\draw[ultra thick,green!50!black] (0:1) arc (0:30:1);
\draw[ultra thick] (30:1) arc (30:120:1);
\draw[ultra thick,red] (120:1) arc (120:240:1);
\draw[ultra thick,violet] (240:1) arc (240:270:1);
\draw[ultra thick,blue] (270:1) arc (270:360:1);

%\draw[red,->] (130:1.2) arc (130:230:1.2);

\draw[thick,fill=white]
(0:1) circle (1.5pt)
;

\draw[thick,fill=yellow]
(120:1) circle (1.5pt)
;

\draw[thick,fill=orange]
(240:1) circle (1.5pt)
;

\path
(90:1.2) coordinate (Y2)
(180:1.2) coordinate (Y3)
;

\end{scope}

\path[->]
 (X) edge[bend left] node[auto] {\footnotesize $\text{interp}(A,B)$} (Y)
;

\path[<->]
 (X1) edge node[auto] {\footnotesize conformal} (X2)
 (Y1) edge node[auto] {\footnotesize conformal} (Y2)
 (Y3) edge node[auto,align=center] {\footnotesize interpolate \\ \footnotesize  from boundary} (X3)
;

\end{tikzpicture}
\end{center}

\caption{$\interp(A,B)$ defined on a 2-cell.}
\label{figure-interp}

\end{figure}
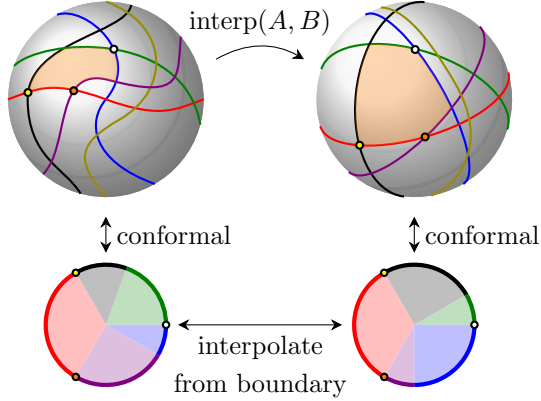

For each $\sigma \in \mc{V}_{2}$, fix a choice of 3 distinct elements $\upsilon_1,\upsilon_2,\upsilon_3 \in \mc{V}_{0}$ with $\upsilon_i <_\vv \sigma$; lets say the 3 lexicographically smallest elements.  These are the sign hyperfield-vectors of 3 vertices of the boundary of $\cell(A,\sigma)$ for $A \in \real(\mc{X})$.
Let $p_i = (\cos (2\pi i/3), \sin (2\pi i/3)) \in \sphere^1$ for $i \in \{1,2,3\}$ be the 3 points spaced uniformly about the unit circle so $p_3 = e_1$.
Let
\[ \conf_{A,\sigma}: \Ball^2 \to \cell(A,\sigma) \]
be the unique homeomorphism as in Theorem \ref{theorem-conformal} that is conformal on the interior and respectively sends $p_i$ to the vertex with sign hyperfield-vector $\upsilon_i$, so $\cell(A,\upsilon_i) = \{\conf_{A,\sigma}(p_i)\}$.

For the 0-cells there is only one possibility, since a 0-cell is just a single point, so 
for $\upsilon \in \mc{V}_{0}$, 
let $\interp(A,B;\cell(B,\upsilon)) = \cell(A,\upsilon)$.

For each $\tau \in \mc{V}_{1}$, fix a sign hyperfield-vector $\sigma \in \mc{V}_{2}$ such that $\sigma >_\vv \tau$ and let 
\[\phi_{A,\tau} : [0,1]_\mb{R} \to \partial{}\Ball^2\] be the positively oriented constant speed parameterization of the arc $\conf_{A,\sigma}^{\,-1}(\cell(A,\tau))$.
Let 
\[ \interp(A,B) = \conf_{A,\sigma}^{\phantom1} \circ \phi_{A,\tau}^{\phantom1} \circ \phi_{B,\tau}^{-1} \circ \conf_{B,\sigma}^{-1} \quad 
\text{on } \cell(B,\tau).\]

For each $\sigma \in \mc{V}_{2}$, let 
$ \phi_{A,B,\sigma} : \Ball^2 \to \Ball^2 $ by 
\[ \phi_{A,B,\sigma}(x) = \|x\| \left(\conf_{A,\sigma}^{-1} \circ \interp(A,B)^{\phantom1} \circ \conf_{B,\sigma}^{\phantom1}(x/\|x\|) \right) \]
for $x \neq 0$, let $\phi_{A,B,\sigma}(0)=0$, 
and let 
\[ \interp(A,B) = \conf_{A,\sigma}^{\phantom1} \circ \phi_{A,B,\sigma}^{\phantom1} \circ \conf_{B,\sigma}^{-1} \quad
\text{on } \cell(B,\sigma).\]

\begin{lemma}\label{lemma-interp}
For any rank 3 sign hyperfield-vector set $\mc{X}$, $\interp$ is continuous on $\real(\mc{X}) \times \real(\mc{X})$ and $A* \interp(A,B) = B$.
\end{lemma}

\begin{observation}\label{remark-interp-id}
The restriction of $\interp$ to the diagonal is the identity, $\interp(A,A;x)=x$.  
\end{observation}

\begin{observation}\label{remark-orthinterp}
\[ \interp(A*Q_1,B*Q_2) = Q_1^* \circ \interp(A,B) \circ Q_2 \]
since $A*Q_1* (Q_1^* \circ  \interp(A,B) \circ Q_2) = B*Q_2 $.  
\end{observation}

\begin{lemma}\label{lemma-path}
Let $S_k$ be a sequence of oriented simple closed curves in $\sphere^2$ that converge to an oriented simple closed curve $S_\infty$ in Fréchet distance, and let $a_k,b_k \in S_k$ such that $a_k \to a_\infty$ and $b_k \to b_\infty$.  Then, the oriented paths $P_k \subset S_k$ from $a_k$ to $b_k$ converge to the oriented path $P_\infty \subset S_\infty$ from $a_\infty$ to $b_\infty$ in Fréchet distance.
\end{lemma}

\begin{proof}
For $x,y \in \sphere^1$, let $[x,y]_{\sphere^1}$ be the counter-clockwise arc from $x$ to $y$.

Since $S_k \to S_\infty$, there is a sequence of embeddings $f_k : \sphere^1 \to \sphere^2$ with $f_k(\sphere^1) = S_k$ %and $g_k(\sphere^1) = S$ 
and $\eps_k \to 0$ 
such that for all $u \in \sphere^1$, $\|f_k(u) - f_\infty(u)\| < \eps_k $.  
Fix $w \in \sphere^1$ such that ${f_\infty(w) \in P_\infty \setminus \{a_\infty,b_\infty\}}$.
%, and let the sequence start at some $k_0 \geq 3$ such that for $k \geq k_0$ we have 
By choosing $k$ sufficiently large, $\eps_{k}$ is sufficiently small, so that we may assume following,
\[
\|f_\infty(w) -a_\infty\| > 2\eps_{k}, \quad 
\|f_\infty(w) -b_\infty\| > 2\eps_{k}, \quad \text{and} \quad
\|b_\infty - a_\infty\| > 4\eps_k.
\] 
Let $x_k \in \sphere^1$ be the first point clockwise of $w$ such that 
$\|f_\infty(x_k) - a_\infty\| = 2\eps_k$ 
and $y_k \in \sphere^1$ be the first point counter-clockwise of $w$ such that 
$\|f_\infty(y_k) - b_\infty\| = 2\eps_k$.  
Starting from $f_\infty(w)$ and traversing $P_\infty$ we reach $f_\infty(y_k)$ before $b_\infty$, so 
$f_\infty(y_k) \in P_\infty \setminus \{a_\infty,b_\infty\}$.
For $u \in [w,y_k]_{\sphere^1}$, we have 
\[ \|f_k(u) - b_\infty\| \geq  \| f_\infty(u) - b_\infty \| - \|f_k(u) -f_\infty(u)\| > 2\eps_k -\eps_k = \eps_k, \]
but $\|b_k - b_\infty\| < \eps_k$, so starting from $f_k(w)$ and traversing $P_k$ we reach $f_k(y_k)$ before $b_k$, which implies
$f_k(y_k) \in P_k \setminus \{a_\infty,b_\infty\}$.
Similarly, $f_k(x_k) \in P_k \setminus \{a_\infty,b_\infty\}$.

Let $\tilde x_k, \tilde y_k \in \sphere^1$ where $f_k (\tilde x_k) = a_k$ and $f_k (\tilde y_k) = b_k$, and 
let 
\[ A_k = [\tilde x_\infty,x_k]_{\sphere^1} \cup [\tilde x_k,x_k]_{\sphere^1} \quad \text{and}\quad 
   B_k = [y_k, \tilde y_\infty]_{\sphere^1} \cup [y_k, \tilde y_k]_{\sphere^1} 
\]
\[ \alpha_k = \sup \{ \|f_\infty(u)-a_\infty\| :\ u \in A_k \} \quad \text{ and} \quad 
\beta_k = \sup \{ \|f_\infty(u)-b_\infty\| :\ u \in B_k \}.
\]
Since $x_k,\tilde x_k \to \tilde x_\infty$ and $y_k, \tilde y_k \to \tilde y_\infty$, we have $\alpha_k,\beta_k \to 0$. Observe that
% \[ \begin{array}{l}
% \sup \{ \|f_k(u)-f_\infty(v)\| :\ u \in (\tilde x_k,x_k),v \in (\tilde x_\infty,x_k) \} \\
% \quad \leq \sup \{ \|f_k(u) -a_\infty \| : u \in A_k \} + \sup \{ \|f_\infty(v) -a_\infty \| : v \in A_k \} \\ 
% \quad < \eps_k +2\alpha_k,  
% \end{array}
% \]
\begin{align*}
&\sup \{ \|f_k(u)-f_\infty(v)\| :\ u \in (\tilde x_k,x_k),v \in (\tilde x_\infty,x_k) \} \\
&\quad \leq \sup \{ \|f_k(u) -a_\infty \| : u \in A_k \} + \sup \{ \|f_\infty(v) -a_\infty \| : v \in A_k \} \\ 
&\quad < \eps_k +2\alpha_k,  
\end{align*}
and similarly
$\sup \{ \|f_k(u)-f_\infty(v)\| :\ u \in (y_k,\tilde y_k),v \in (y_k,\tilde y_\infty) \} < \eps_k +2\beta_k $.
%\eps_k +2\alpha_k > \sup{}_{u,v} \|f_k(u)-f_\infty(v)\| :\ u,v \in A_k \quad \text{ and} \quad 

For $k\geq 3$, 
let $u_k : [\nicefrac{1}{k},1-\nicefrac{1}{k}]_\mb{R} \to \sphere^1$ be a counter-clockwise parameterization of the arc from $x_k$ to $y_k$, and define $g_k,h_k : [0,1]_\mb{R} \to \sphere^2$ as follows.
For $t \in [\nicefrac{1}{k},1-\nicefrac{1}{k}]_\mb{R}$, let 
\[ g_k(t) = f_k \circ u_k(t) \quad \text{and} \quad h_k(t) = f_\infty \circ u_k(t),\] 
and on the rest of $[0,1]_\mb{R}$ complete $g_k$ and $h_k$ to respective parameterizations of $P_k$ and $P_\infty$.
We now have 
\[\begin{array}{r@{\ <\ }l@{\quad \text{for}\quad }r@{\ \in\ }l}
\|g_k(t)-h_k(t)\| & \eps_k +2\alpha_k & t & [0,\nicefrac{1}{k}]_\mb{R}      \\
\|g_k(t)-h_k(t)\| & \eps_k & t & [\nicefrac{1}{k},1-\nicefrac{1}{k}]_\mb{R} \\
\|g_k(t)-h_k(t)\| & \eps_k +2\beta_k & t & [1-\nicefrac{1}{k},1]_\mb{R}
\end{array}\] 
and $\alpha_k,\beta_k,\eps_k \to 0$.  Therefore, $P_k \to P_\infty$ in Fréchet distance.
\end{proof}

\begin{proof}[Proof of Lemma \ref{lemma-interp}]
First, we show that $A* \interp(A,B) = B$, since $\interp(A,B)$ is defined by building up a map that sends each cell of $B$ to the corresponding cell of $A$.
In the case $\upsilon \in \mc{V}_0$, this is immediate from the definition.
In the case $\tau \in \mc{V}_1$, $\phi_{B,\tau}^{-1}$ is defined so that $\hbox{$\phi_{B,\tau}^{-1} \circ \conf_{B,\sigma}^{-1}(\cell(B,\tau))$} = [0,1]_\mb{R}$, and this is bijective, which also means $\conf_{A,\sigma}^{\phantom1} \circ \phi_{A,\tau}^{\phantom1}([0,1]_\mb{R}) = \cell(A,\tau)$.
In the case $\sigma \in \mc{V}_2$, $\interp(A,B)$ is a composition of bijections sending $\cell(B,\sigma)$ to the unit ball, which is then sent to itself, and then to $\cell(A,\sigma)$.

Next, we show $\interp(A,B) \in \hom(\sphere^2)$.
By definition $\conf_{A,\sigma}$ and $\conf_{A,\sigma}^{-1}$ and $\phi_{A,\tau}$ are continuous. 
And, $\phi_{A,B,\sigma}$ is a composition of continuous functions except at the origin, where $\phi_{A,B,\sigma}$ is also continuous, since $\|\phi_{A,B,\sigma}(x)\| = \|x\|$.  
We just have to check that the definition of $\interp(A,B)$ in a 1-cell agrees with the value $\interp(A,B)$ approaches on the boundary of a 2-cell.  Let $\sigma \in \mc{V}_2$, $\tau \in \mc{X}$ such that $\tau <_\vv \sigma$, and let $x_k \in \cell(B,\sigma)$ such that $x_k \to x \in \cell(B,\tau)$.  Then,
\begin{align*}
\interp(A,B;x_k)
& \to \conf_{A,\sigma}^{\phantom1} \circ \phi_{A,B,\sigma}^{\phantom1} \circ \conf_{B,\sigma}^{-1}(x) \\
& = \conf_{A,\sigma}^{\phantom1} \circ \conf_{A,\sigma}^{-1} \circ \interp(A,B)^{\phantom1} \circ \conf_{B,\sigma}^{\phantom1} \circ \conf_{B,\sigma}^{-1}(x) \\
& = \interp(A,B;x),
\end{align*}
since $\| \conf_{B,\sigma}^{-1}(x) \| = 1$.  Thus, $\interp(A,B) \in \hom(\sphere^2)$.

Finally, we show that $\interp(A,B)$ depends continuously on $A$ and $B$.
The vertices of $A$ depend continuously on $A$ by Lemma~\ref{lemma-vertex}, and likewise for $B$. 
By Lemma \ref{lemma-path}, this implies that each 1-cell of $A$ and $B$ varies continuously, and therefore the conformal maps defining $\interp(A,B)$ vary continuously by Theorem \ref{theorem-conformal}.
\end{proof}

\subsection{Topology of pseudolinear realization spaces}

\begin{proof}[Proof of Theorem \ref{theorem-disth}]
Let $A_k \in \real(\mc{X})$ for $k \in \{1,2,\dots,\infty\}$. 
If $A_k \to A_\infty$ with respect to $\dist$, then $\interp(A_k,A) \to \id_{\sphere^2}$ and $\wei(A_k) \to \wei(A_\infty)$, so $A_k \to A_\infty$ with respect to $\disth$.
If $A_k \to A_\infty$ with respect to $\disth$, then each pseudosphere of $A_k$ converges to the corresponding pseudosphere of $A_\infty$ and $\wei(A_k) \to \wei(A_\infty)$, so $A_k \to A_\infty$ with respect to $\dist$.
\end{proof}

\begin{proof}[Proof of Theorem \ref{theorem-pseudoreal-contract}]
By the Topological Representation Theorem, $\realo(\mc{X})$ is non-empty \cite{folkman1978oriented}. 
For $A \in \mc{A} \in \realo(\mc{X})$, we define a strong deformation retraction $\rho$ from $\realo(\mc{X})$ to $\mc{A}$.
For each $\mc{B} \in \realo(\mc{X})$, fix $B \in \mc{B}$. Let  
\Needspace{2cm}
\[\rho : \realo(\mc{X}) \times [0,1]_\mb{R} \to \realo(\mc{X}) \] 
\[ \rho(\mc{B},t) = A*\ho(\interp(A,B),t)*\orth_3  \]
Observe that $\rho(\mc{B},t)$ does not depend on the choice of $B \in \mc{B}$ by Observation \ref{remark-orthinterp} and since $\ho$ is $\orth_3$-equivariant with respect to precomposition (right action).  Indeed, for any other choice $B' \in \mc{B}$, there is $Q \in \orth_3$ such that $B' = B*Q$, and we have 
\begin{align*}
A*\ho(\interp(A,B'),t)*\orth_3 
&= A*\ho(\interp(A,B) \circ Q,t)*\orth_3 \\
&= A*\ho(\interp(A,B),t)*Q*\orth_3 \\
&= \rho(\mc{B},t).
\end{align*}  
Furthermore, $\rho$ is continuous by Lemma \ref{lemma-interp} and is a strong deformation retraction by Theorem \ref{theorem-kneser}.

If $\interp(A,B) \in \hom^+(\sphere^2)$, then $\ho(\interp(A,B),t) \in \hom^+(\sphere^2)$ for all $t \in [0,1]_\mb{R}$,
so the same deformation also shows that $\pog(\chi)$ is contractible.
\end{proof}

\begin{remark}
Note that if $\ho$ were $\orth_3$-equivariant with respect to postcomposition (left action), then $\rho$ would not depend on a choice of $A \in \mc{A}$, but the dependence on $A$ is not a problem for our purposes.
\end{remark}

\subsection{Deforming weighted pseudocircle arrangements}

We now define the deformation of Theorem \ref{theorem-grassman-pseudograssman}.

We first define several intermediate spaces and deformation retractions between these space.  We then combine these deformations to produce a deformation retraction from $\pstief_{3,n}$ to $\stief_{3,n}$.

\subsubsection{Intermediate spaces}
\label{subsubsec-spaces}

For each non-repeating non-empty sequence $I = (i_1,\dots,i_m)$ %of length $m \in \{0,1,\dots,n\}$ 
with entries among $[n]_\mb{N}$ we define spaces $X_I,Y_I,Z_I \subseteq \pstief_{3,n}$ by giving conditions for a weighted pseudosphere arrangement $A = (\alpha_1,\dots,\alpha_n)$ to be in the space.
For this, let $S_i = (\aim \alpha_i)^{-1}(0)$. 
%denote the $i$-th \df{unoriented pseudocircle} of $A$.  

For $m = |I| \in \{0,1,2,3\}$, let $X_I$ be the set of arrangements $A$ such that $I$ is an independent set of $\cov(A)$.  That is, 

\vbox{
\begin{itemize}
\item
$X_{()} = \pstief_{3,n}$; 
\item
$X_{(i_1)} \subset X_{()}$ where $\|\alpha_{i_1}\| \neq 0$;
\item
$X_{(i_1,i_2)} \subset X_{(i_1)}$ where $\|\alpha_{i_2}\| \neq 0$ and $S_{i_1} \neq S_{i_2}$; 
\item
$X_{(i_1,i_2,i_3)} \subset X_{(i_1,i_2)}$ where $\|\alpha_{i_3}\| \neq 0$ and $S_{i_1} \cap S_{i_2} \cap S_{i_3} = \emptyset$.
\end{itemize}
}
For $m >3$, let $X_I = X_{(i_1,\dots,i_m)} \subset Y_{(i_1,\dots,i_{m-1})}$ such that $\|\alpha_{i_m}\| \neq 0$.
Informally, we use $X_{I}$ to first pick a basis and then subsequently pick non-zero weighted pseudocircles as long as one is available.

For $m \in \{0,1,2\}$, let $Y_I = X_I$. 
%Let $Y_{(i_1)} = X_{(i_1)}$ and $Y_{(i_1,i_2)} = X_{(i_1,i_2)}$.
For $m=3$, let $Y_{(i_1,i_2,i_3)} \subset X_{(i_1,i_2,i_3)}$ where $S_{i_1},S_{i_2},S_{i_3}$ are great circles. 
For $m >3$, let $Y_I \subset X_I$ such that $S_{i_m}$ is antipodally symmetric and is geodesic in each of the 2-cells of the subdivision of $\sphere^2$ by $S_{i_1},\dots,S_{i_{m-1}}$.  That is, the pseudocircles indicated by the sequence $I$ are symmetric and piecewise geodesic with corners only on the pseudocircles appearing earlier in the sequence.

For $m \in \{0,1,2\}$, let $Z_I$ be the set of spanning weighted great circle arrangements, i.e.\ vector configurations spanning $\Rpol{3}$. 
For $m\geq 3$, let $Z_I \subseteq Y_I$ such that every pseudocircle of $A$ is antipodally symmetric and for all $j \not\in I$, if $\|\alpha_j\| \neq 0$ then $S_{j}$ is geodesic on each 2-cell of $\proj_I(A)$.  That is, all pseudocircles are piecewise geodesic with corners only on the pseudocircles appearing in the sequence, and those that appear in the sequence have corners only on those earlier in the sequence, and $-\proj_I(A) = \proj_I(A) * (-\id)$.

We will define strong equivariant deformation retractions $f_I$ from $X_I$ to $Y_I$ and $h_{I} = h_{(i_1,\dots,i_m)}$ from $Z_{(i_1,\dots,i_{m})}$ to $Z_{(i_1,\dots,i_{m-1})}$.  
%For $m \in \{0,1,2\}$, $f_I$ and $h_I$ are just the trivial deformation.
We then combine these to get strong deformation retractions $g_I$ from $Y_I$ to $Z_I$.
%, which are defined recursively by equation \ref{equation-g} on page \pageref{equation-g} in the case $n \neq \infty$; the infinite case is similar.  
%Equation \ref{equation-g} may be viewed as a blueprint for how 
%along with the proof of Claim \ref{claim-g} as the central argument that we do indeed obtain a deformation to great circles.   
The deformations $f_I$ and $h_I$ for each $I$ ultimately fit together in equation \ref{equation-g} on page \pageref{equation-g} (or equation \ref{equation-g-infty} on page \pageref{equation-g-infty} for $n = \infty$) to give a deformation $g_{()}$ from $Y_{()} = \pstief_{3,n}$ to the space $Z_{()}$ of vector configurations that span $\Rpol{3}$.
To complete the deformation of Theorem~\ref{theorem-grassman-pseudograssman}, we then perform a strong equivariant deformation retraction from great circles to Parseval frames by a continuous orthonormalization process.

Before defining the deformations we start with some properties of these spaces that will be needed.

\begin{claim}\label{claim-pseudocircle-dist}
Let $(\alpha_1,\alpha_2)$ be a weighted pseudocircle arrangement.  Then, $\dist(\alpha_1,\alpha_2)\leq d= {\|\alpha_1\|+\|\alpha_2\|}$.  Also, $\dist(\alpha_1,\alpha_2) = d$ if and only if $\alpha_1 = -\alpha_2$ and these are antipodally symmetric.
\end{claim}

\begin{proof}
We have $\dist(\alpha_1,\alpha_2)\leq d$ immediately from the definition of $\dist$.
If $\alpha_1 = -\alpha_2$ are antipodally symmetric, then any parameterizations of $S_1 = S_2$ with opposite orientations must pass though an antipodal pair of points, so $\dist(\alpha_1,\alpha_2) = d$.  It remains to show the `only if' implication.

Suppose there were a point $x_1 \in S_1$ such that $-x_1 \not\in S_2$.
Then, there would have to be $x_2 \in S_2$ such that $-x_2 \not\in S_1$, since $S_2$ cannot be a proper subset of $-S_1$.  In this case, we would be able to define parameterizations $\psi_i : \sphere^1 \to S_i$ such that $\psi_1$ stays at $x_1$ while $\psi_2$ traverses $S_2$ starting at $x_2$ and then vise versa, so we would have $\dist(\alpha_1,\alpha_2) < d$. 

Hence, if $\dist(\alpha_1,\alpha_2) = d$ then $S_1 = -S_2$ with the antipodal orientations.  Since $(S_1,S_2)$ is a pseudocircle arrangement, they must either coincide or cross transversally at exactly two points with opposite orientation.  If they crossed transversally at a point $x$, then they would also have to cross at $-x$ with the same orientation, which is a contradiction.  Thus, we have $S_1 = S_2$ and $\alpha_1 = -\alpha_2$.
\end{proof}

\begin{claim}\label{claim-sphericonvex}
For $m = |I| \geq 3$ and $A \in Y_I$, the 2-cells of $\proj_I({A})$ are spherical convex polygons such that each edge is a 1-cell $\proj_I({A})$.
\end{claim}

\begin{claimproof}
We proceed by induction.  For $m=3$, each cell is an orthant of the sphere, which is a spherical triangle with 1-cell edges.  Since $Y_I \subset Y_{(i_1,\dots,i_{m-1})}$, the 2-cells of $\proj_{(i_1,\dots,i_{m-1})}(A)$ are spherical convex polygons with edges that are 1-cells by inductive assumption.  At each 2-cell $C$ such that $S_{i_m}$ intersects the interior of $C$, $S_{i_m}$ must intersect $C$ in a path $P$, since the restriction of $\proj_I({A})$ is homeomorphic to a 1-dimensional pseudosphere arrangement and $P = C \cap S_{i_m}$ is a 1-cell of that arrangement.
From the definition of $Y_I$, $P$ must be a geodesic arc through $C$, thereby subdividing $C$ into a pair of spherical convex polygons with edges that are 1-cells. 
\end{claimproof}

\begin{claim}\label{claim-open-X}
For $m > 0$, $X_I$ is an open proper subset of $Y_{(i_1,\dots,i_{m-1})}$ in the induced metric topology on $Y_{(i_1,\dots,i_{m-1})}$.
\end{claim}

Note that $X_I$ is not generally open in the topology on $\pstief_{3,n}$.

\begin{claimproof}
Consider some $A \in Y_{(i_1,\dots,i_{m-1})}$, a generic configuration in $(\Rpol{3})^n \subset Y_{(i_1,\dots,i_{m-1})}$ will suffice.  Then, $\proj_{[n]_\mb{N}\setminus i_m}(A) \in Y_{(i_1,\dots,i_{m-1})} \setminus X_I$, so $X_I$ is a proper subset of $Y_{(i_1,\dots,i_{m-1})}$.

To show $X_I$ is open, for each $A \in X_I$, we find a radius $r>0$ sufficiently small that every $B \in Y_{(i_1,\dots,i_{m-1})}$ within distance $r$ of $A$ is in $X_I$.
For $m =1$ or $m>3$, $r = \|\alpha_{i_m}\|$ suffices.
For $m \in \{2,3\}$, we may use 
\[ r = \min\left\{\|\alpha_i\|:\ i\in I\right\} \cdot \inf\left\{ (\nicefrac{1}{2}) \|x -y\|:\ x \in P,\ y \in (S_{i_1} \cup \dots \cup S_{i_m}) \right\} \]
where $P$ consists of a point in the interior of each 2-cell of the subdivision of the sphere by $S_{i_1},\dots,S_{i_m}$.  
For every $B \in Y_{(i_1,\dots,i_{m-1})}$ within distance $r$ from $A$, we have that the points of $P$ are each in the corresponding 2-cell of $B$, which means $B \in X_I$.
\end{claimproof}

\subsubsection{The deformation $f_I$ from $X_I$ to $Y_I$}
\label{subsubsec-f}

%We now define a deformation retraction $f_I$ from $X_I$ to $Y_I$.
For $m \in \{0,1,2\}$, let $f_I(A,t) = A$. 
For $m \geq 3$, consider $A \in X_I$ and assume that $Q = \coord((i_1,i_2,i_3);A)$ is the identity, otherwise let $f_I(A,t) = f_I(A*Q,t)*Q^{-1}$.

let 
\[ f_{I}(A,t) = A*\ho(\interp(B,C),1-t)\]
where $B = (\beta_1,\dots,\beta_n) = \proj_I(A)$ and $C = (\gamma_1,\dots,\gamma_n)$ with $\wei(C)=\wei(B)$ depends on $m$ as follows.  

For $m=3$, let $\gamma_{i_1} = \|\alpha_{i_1}\|e_1$, while for $k \in \{2,3\}$, $\gamma_{i_k} \in \Rpol{3}$ is the unique vector such that 
$\|\gamma_{i_k}\| = \|\alpha_{i_k}\|$,
$\dist(\gamma_{i_k},\gamma_{i_1}) = \dist(\alpha_{i_k},\alpha_{i_1})$,
and $\aim(\gamma_{i_k},x) = \aim(\alpha_{i_k},x)$ for $x \in \cell(B,\sign(e_{i_2})) \cup \cell(B,\sign(e_{i_3}))$.

For $m > 3$ and for $k<m$ let $\gamma_{i_k}= \beta_{i_k}$.
Let $\gamma_{i_m}$ be a copy of $\beta_{i_k}$ in the upper hemisphere that is straightened in each cell keeping the end points fixed, and a reflected copy in the lower hemisphere.
That is, we let 
\[ H \dfeq \left\{(x_1,x_2,x_3) \in \sphere^2 :\ x_1 > 0 \text{ or } (x_1=0 \text{ and } x_2 \geq 0)\right\}, \]
\[ V \dfeq \bigcup_{k=1}^{m-1} S_{i_k}. \]
For $v \in V$, 
\[ \aim (\gamma_{i_m}; v) = \begin{cases} 
\phantom{-}\aim (\beta_{i_m}; v) & v \in H \\
{-}\aim (\beta_{i_m}; -v) & v \not\in H 
\end{cases}\]
Let $\gamma_{i_m}$ be defined on $\sphere^2 \setminus V$ such that $S_{i_m}$ is geodesic on all 2-cells of $\proj_{(i_1,\dots,i_{m-1})} A$.

\begin{claim}\label{claim-deform-f}
$f_I$ is a well defined strong equivariant deformation retraction from $X_I$ to $Y_I$.
\end{claim}

\begin{claimproof}
We start with the case $m=3$.  Since $A \in X_I$, $I$ is a basis of $A$, so $\{{i_1},{i_2},{i_3}\}$ is an independent set, which implies $\dist(\alpha_{i_k},\alpha_{i_1}) < \|\alpha_{i_k}\| +\|\alpha_{i_1}\|$ by Claim \ref{claim-pseudocircle-dist}.
There is a circle in $\Rpol{3}$ of points that are distance $\|\alpha_{i_k}\|$ from $0$ and distance $\dist(\alpha_{i_k},\alpha_{i_1})$ from $\gamma_{i_1} = \|\alpha_{i_1}\|e_1$, and only an antipodal pair among this circle is orthogonal to $p_{5-k} = \cell(B,\sign(e_{i_{5-k}}))$, and only one can be the vector $\gamma_{i_k}$ toward $p_k = \cell(B,\sign(e_{i_{k}}))$, in the sense that 
$\langle \gamma_{i_k}, p_k\rangle > 0$, so $C$ is well defined.
We have $\cov(C) = \cov(B)$ is the set of all sign sequences with support in $I$, so $f_I$ is a well defined deformation by Lemma \ref{lemma-interp}, Theorem \ref{theorem-kneser}, and Remark \ref{remark-orthinterp} as in the proof of Theorem \ref{theorem-pseudoreal-contract}. 
Since $\{\gamma_{i_1},\gamma_{i_2},\gamma_{i_3}\}\subset \Rpol{3}$ is a basis, $C \in Y_I$, so $f_I$ deforms $X_I$ to $Y_I$.  Since $\proj_I(f_I(A,t))$ remains in the same $\pstief$-realization space throughout the deformation, $\{i_1,i_2,i_3\}$ is an independent set of $f_I(A,t)$, so $f_I(A,t) \in X_I$, so $f_I$ is a deformation retraction from $X_I$ to $Y_I$.  
If $A \in Y_I$ to start, then $\{\alpha_{i_1},\alpha_{i_2},\alpha_{i_3}\}\subset \Rpol{3}$ is a basis, which means $\alpha_{i_k} = \beta_{i_k} = \gamma_{i_k}$, so $B = C$, so $\interp(B,C)$ is the identity, which implies $f_I(A,t) = A$ is trivial, since $\ho$ is strong.  Thus, $f_I$ is strong.

Now consider the case $m>3$. 
Since the 2-cells of $\proj_{(i_1,\dots,i_{m-1})}(A)$ are spherical convex polygons by Claim \ref{claim-sphericonvex}, the geodesic arc between any pair of points is in that 2-cell, so $C$ is well defined and $\sign(B)=\sign(C)$, so $f_I$ is a well defined deformation.  Since $C \in Y_I$ by definition, $f_I$ is a deformation from $X_I$ to $Y_I$. 
Since the initial sequence of weighted pseudocircles $\alpha_{i_1},\dots,\alpha_{i_{m-1}}$ remain fixed throughout the deformation, we have for all $A \in X_I$ and $t \in [0,1]_\mb{R}$ that $f_I(A,t) \in Y_{(i_1,\dots,i_{m-1})}$, and since the norm of $\alpha_{i_m}$ also remains fixed, we have $f_I(A,t) \in X_I$, which implies that $f_I$ is a deformation retraction.  
If $A \in Y_I$ to start, then $f_I(A,t) = A$ is trivial, since $\beta_{i_m} = \gamma_{i_m}$, which implies $B=C$, so the deformation is strong.

Finally, the deformation is equivariant since $f_I(A,t) = f_I(A*Q,t)*Q^{-1}$ by definition.
\end{claimproof}

\begin{claim}\label{claim-stable-f}
For all $p \in \{0,\dots,m{-}1\}$ and all $j \in [n]_\mb{N} \setminus \{i_1,\dots,i_p\}$, 
if $A \in X_I \cap X_{(i_1,\dots,i_{p},j)}$ then $f_I(A,1) \in X_{(i_1,\dots,i_{p},j)}$.
\end{claim}

\begin{claimproof}
For $p=0$, the claim holds since $f_I$ preserves norm.  In particular, for $A \in X_{(j)}$, the norm of the $j$-th pseudocircle never vanishes throughout the deformation $f_I(A,t)$.
Similarly for $p\geq 3$, the claim holds since $f_I(A,t) \in X_I \subset Y_{(i_1,\dots,i_{m-1})} \subseteq Y_{(i_1,\dots,i_k)}$, so $f_I(A,1) \in Y_{(i_1,\dots,i_k)}$, and $f_I$ preserves norm.

For $p=1$ (or $p=2$), the claim holds since $f_I$ preserves order type.  In particular, if $(i_1,j)$ (or $(i_1,i_2,j)$) is an ordered independent set of $\cov(A)$, then it remains independent throughout the deformation.
\end{claimproof}

\subsubsection{The deformation $h_{(i_1,\dots,i_m)}$ from $Z_{(i_1,\dots,i_m)}$ to $Z_{(i_1,\dots,i_{m-1})}$.}
\label{subsubsec-h}

We now define a deformation retraction $h_I = h_{(i_1,\dots,i_m)}$ from $Z_I = Z_{(i_1,\dots,i_m)}$ to $Z_{(i_1,\dots,i_{m-1})}$.  If $m < 3$, this is the trivial deformation, so assume $m \geq 3$.

Recall from the definition of $Z_I$ that for $A \in Z_I$, $S_{i_1},S_{i_2},S_{i_3}$ are great circles, which means $\alpha_{i_1}, \alpha_{i_2},\alpha_{i_3} \in \Rpol{3}$.
We project from $\sphere^2$ to the the real projective plane $\projplane$ (defined as a compactification of $\mb{R}^2$) by the 2-to-1 covering map $\proj_{\projplane}(\alpha_{i_1},\alpha_{i_2},\alpha_{i_3}) : \sphere^2 \to \projplane$ that sends $S_{i_1}$ to the horizon, and $S_{i_2}$ and $S_{i_3}$ to the horizontal and vertical axes respectively.
This map is given on $x \in \sphere^2 \setminus S_{i_1}$ by
\[ \proj_{\projplane}(\alpha_{i_1},\alpha_{i_2},\alpha_{i_3}; x) = \left( \tfrac{\langle \alpha_{i_2}, x \rangle}{\langle \alpha_{i_1}, x \rangle}, \tfrac{\langle \alpha_{i_3}, x \rangle}{\langle \alpha_{i_1}, x \rangle} \right), \]
where $\langle \cdot,\cdot \rangle$ is the inner product on $\Rpol{3}$ corresponding to the standard inner product on $\mb{R}^3$, 
and the map is defined for $x \in S_{i_1}$ by continuous extension.
We do this projection so that we can make use of the vector space structure of $\mb{R}^2 \subset \projplane$ when defining $h_I$.

Recall that a \df{pseudoline} in $\projplane$ is a simple closed curve in the real projective plane that cannot be deformed to a point.
Let $\curves$ denote the space of all pseudolines in $\projplane$. 
We give $\curves$ the metric that is the pull-back by $\proj_{\projplane}(e_1,e_2,e_3)$ of Fréchet distance on pseudocircles in $\sphere^2$.
By \df{pseudoline arrangement} we mean a collection of pseudolines such that any two either coincide or intersect at a single point.  
We say $L$ is a \df{pseudoline extension} of a pseudoline arrangement $(L_{i_1},\dots,L_{i_m})$ when $(L_{i_1},\dots,L_{i_m},L)$ is a pseudoline arrangement, i.e.\ either $L = L_{i_j}$ for some $j$, or $L$ is a pseudoline that intersects each $L_{i_j}$ exactly once.

Since the pseudocircles of an arrangement in $Z_I$ are antipodally symmetric, these project to pseudolines in $\projplane$, which we denote $L_i = L_i(A) = \proj_{\projplane}(\alpha_{i_1},\alpha_{i_2},\alpha_{i_3};S_i)$.  
We will define the deformation $h_I$ by deforming the pseudolines $L_i$ and lifting these to deformations of weighted pseudocircles.
An important feature of the deformations that we will use is that each pseudoline $L_j$ for $j \in \compl{I} = [n]_\mb{N} \setminus I$ deforms in a way that depends only on $L_{i_1},\dots,L_{i_m}$, and the initial position of $L_{j}$, but distinct pseudolines deform independently of each other.
Claim~\ref{claim-pseudoline-lift} below shows that this is a valid way to define such a deformation.

%This metric on $\curves$ is chosen so that $\proj_{\projplane}(\alpha_{i_1},\alpha_{i_2},\alpha_{i_3})$ induces continuous maps from $Z_I$ to $\curves^n$.
% by sending each $S_j$ to $L_j$.  
Let
\begin{align*}
L_{(j_1,\dots,j_k)}(A)  = (L_{j_1}(A),\dots,L_{j_k}(A)) 
\quad \text{and} \quad
\tilde Z_I  = L_I(Z_I). 
\end{align*}
We call a partial map
\[ \tilde \lambda : \tilde Z_I \times \curves \times [0,1]_\mb{R} \not\to \curves \]
an \df{extension deformation process} on $\tilde Z_I$ when, for all $\tilde A \in \tilde Z_I$, $\lambda_0 \in \curves$, and $t \in [0,1]_\mb{R}$, the following hold.
\begin{enumerate}
\item
$\tilde \lambda(\tilde A, \lambda_0, 0) = \lambda_0$.
\item
$\tilde \lambda(\tilde A, \lambda_0, t)$ is defined, 
%and is a pseudoline extension of $\tilde A$, 
provided that $\lambda_0$ is a pseudoline extension of $\tilde A$ and $\lambda_0$ is linear in each cell of $\tilde A$.
\item
$\tilde \lambda$ is continuous on the domain where it is defined.
\end{enumerate}
We say $\eta: Z_I \times [0,1]_\mb{R} \to \pstief_{3,n}$ is a deformation of $Z_I$ \df{induced} by the extension deformation process $\tilde \lambda$ on $Z_I$ when, for all $A \in Z_I$, $t \in [0,1]_\mb{R}$, and $j \in \compl{I}$, the following hold.
\begin{enumerate}
\item
$\wei(\eta(A,t)) = \wei(A)$. %, i.e.\ weights remain fixed throughout the deformation.
\item
$\proj_{I}(\eta(A,t)) = \proj_{I}(A)$. %, i.e.\ the pseudocircles $\alpha_{i}$ for $i\in I$ remain fixed throughout the deformation.
\item
$\eta(A,t)$ is antipodally symmetric and $L_{j}(\eta(A,t)) = \tilde \lambda(L_I(A),L_j(A),t)$.
\end{enumerate}

\begin{claim}\label{claim-pseudoline-lift}
{If} $\tilde \lambda$ is an extension deformation process on $\tilde Z_I$ such that
\[ \left(\tilde \lambda(\tilde A, L_1, t),\dots,\tilde \lambda(\tilde A, L_{n}, t) \right) \] 
is a pseudoline arrangement for every $\tilde A \in \tilde Z_I$, $t \in [0,1]_\mb{R}$, and every pseudoline arrangement $(L_1,\dots,L_n)$ where $(L_{i_1},\dots,L_{i_m}) = \tilde A$, 
{then} there is a unique deformation $\eta$ of $Z_I$ induced by $\tilde \lambda$ and $\eta$ is $\orth_3$-equivariant. 
\end{claim}

\begin{claimproof}
We first demonstrate that we do not need for $\tilde \lambda$ to keep track of the orientations of the pseudocircles, as these can be tracked throughout the deformation, and thereby show uniqueness.
Let 
\[ W_A = \left\{B \in \pstief_{3,n}:\ -B = B*(-\id),\ \wei(B) = \wei(A),\ \proj_I(B) = \proj_I(A) \right\} \]
be the space of symmetric weighted pseudocircle arrangements where the weight of each pseudocircle is fixed to match that of $A$,
and the pseudocircles indexed by $I$ coincide with those of $A$.
Let $J = \{j \in \compl{I}: \|\alpha_j\| \neq 0 \}$ be the indices of the non-zero elements of $A$.  
For $B = (\beta_1,\dots,\beta_n) \in W_A$, we have $\dist(\beta_j,- \beta_j) = 2\|\alpha_j\| >0$, so the $\mb{Z}_2$-action reversing the orientation of the $j$-th element is free on $W_A$, provided that $j \in J$.
Therefore, $W_A$ is a $(\mb{Z}_2)^{J}$-fiber bundle over the space of pseudoline arrangements extending $L_I(A)$ by $|J|$ pseudolines. 
Let $\tilde A = L_I(A)$ and 
\[ \Lambda(B,t) = \left( \tilde \lambda(\tilde A, L_1(B), t),\dots, \tilde \lambda(\tilde A, L_n(B),t) \right). \]
By the hypotheses of the claim, $\Lambda(B,t)$ is a pseudoline arrangement, so there is a non-empty subset of $W_A$ that projects to $\Lambda(B,t)$.
Since fiber bundles have the homotopy lifting property, the deformation $\Lambda$ lifts to a continuous deformation $\eta_A$ of $Z_I \cap W_A$ in $W_A$.  

Let $\eta$ be the function $\eta(A,t) = \eta_A(A,t)$ for each $A \in Z_I$.
By definition, if there is a deformation $\eta'$ of $Z_I$ induced by $\lambda$, then $\eta'(A,t) \in W_A$ and each pseudocircle projects to a pseudoline that deforms according to $\tilde \lambda$, so $\eta' = \eta$ is the lift of $\Lambda$.  
Hence $\eta$ is the unique deformation induced by $\tilde \lambda$, provided $\eta$ is continuous.

Next, we show that $\eta$ is continuous where it is defined.
Let $\eta_{A,i}$ be the $i$-th weighted pseudocircle of $\eta_A$ so that  
$\eta_A = {\eta_{A,1} \times \dots \times \eta_{A,n}}$.
Note that some of the weighted pseudocircles may vanish and reappear elsewhere as $A$ varies, but $L_{i_1},\dots,L_{i_m}$ cannot vanish by the definition of $Z_I$.
Consider $A_k \to A_\infty \in Z_I$ and $t_k \to t_\infty \in [0,1]_\mb{R}$.  Let
\[ A_k = (\alpha_{k,1},\dots,\alpha_{k,n}) = ((r_{k,1},\theta_{k,1}),\dots,(r_{k,n},\theta_{k,n})), \]
and let $M_k$ be the linear transformation sending $\alpha_{\infty,i_1},\alpha_{\infty,i_2},\alpha_{\infty,i_3}$ to $\alpha_{k,i_1},\alpha_{k,i_2},\alpha_{k,i_3}$, and let $N_k \in \hom(\sphere^2)$ be $M_k$ followed by projecting radially to the sphere.
Recall that $\{\alpha_{k,i_1},\alpha_{k,i_2},\alpha_{k,i_3}\}$ is a basis for the vector space $\Rpol{3}$ by definition of $Z_I$, so $M_k$ is a well defined invertible linear transformation and $M_k \to M_\infty=\id$ in the operator norm.

For each $i$ such that $r_{\infty,i} > 0$, we have $k$ large enough that $r_{k,i} > 0$, so we may let 
\[ B_k = (\beta_{k,1},\dots,\beta_{k,n}) \in W_{A_\infty} \]
where %$\|\tilde \alpha_{k,i}\| = \|\alpha_{\infty,i}\|$ and 
\[ \beta_{k,i} = \begin{cases}
(r_{\infty,i},\theta_{k,i} \circ N_k) & r_{\infty,i} > 0 \\
0 & r_{\infty,i} = 0. 
\end{cases}\]
Since $\theta_{k,i} \to \theta_{\infty,i}$ and $N_k \to \id$, we have $B_k \to A_\infty$, and since $\eta_{A_\infty}$ is a continuous deformation in $W_{A_\infty}$, we have $\eta(B_k,t_k) = \eta_{A_\infty}(B_k,t_k) \to \eta(A_\infty,t_\infty)$.

The distance between the $i$-th element $\alpha$ of $\eta(A_k,t_k)$ and $\beta$ of $\eta(B_k,t_k)$ is, by definition of Fréchet distance, the infimum over maps from $\|\alpha_{k,i}\|S \subset \mb{R}^3$ to $\|\beta_{k,i}\|T = \|\alpha_{\infty,i}\|T \subset \mb{R}^3$ respecting the orientations of $S$ and $T$, where $S$ and $T$ are the kernel of $\aim(\alpha)$ and $\aim(\beta)$.
Let $\|\cdot\|_\text{op}$ denote the operator norm of a linear transformation.
We can obtain $\|\alpha_{\infty,i}\|T$ from $\|\alpha_{k,i}\|S$ by first scaling by $\|\alpha_{\infty,i}\|/\|\alpha_{k,i}\|$, which moves each point at most $\left| \| \alpha_{\infty,i} \| - \| \alpha_{k,i} \| \right|$,  then applying $M_k$, which moves each point at most $\| \alpha_{\infty,i} \| \cdot \|M_k -\id \|_{\textrm{op}}$, and then projecting radially to $\|\alpha_{\infty,i}\| \sphere^2$, which moves each point at most $\| \alpha_{\infty,i} \| \cdot | \|M_k \|_{\textrm{op}} -1 |$.  Hence, we can bound the distance to $\beta$ from $\alpha$ as the sum of the distances for each of these three steps, which gives,
\[ \dist(\eta(B_k,t_k), \eta(A_k,t_k)) \leq \max_{i \in [n]_\mb{N}} \left( \left| \| \alpha_{\infty,i} \| - \| \alpha_{k,i} \| \right| +\| \alpha_{\infty,i} \| \left( \|M_k -\id \|_{\textrm{op}} +\left| \|M_k \|_{\textrm{op}}  -1 \right| \right) \right). \]
Since $\|\alpha_{k,i}\| \to \|\alpha_{\infty,i}\|$ and $M_k \to \id$, we have $\dist(\eta(B_k,t_k), \eta(A_k,t_k)) \to 0$, so 
$\eta(A_k,t_k) \to \eta(A_\infty,t_\infty)$, which means $\eta$ is continuous.
Thus, $\eta$ is the deformation of $Z_I$ induced by $\tilde \lambda$.

Finally, we show that $\eta$ is $\orth_3$-equivariant. For $Q \in \orth_3$, we have 
\begin{align*}
L_i(A*Q) 
= \left\{ \left( \tfrac{\langle Q^*\alpha_{i_2}, x \rangle}{\langle Q^*\alpha_{i_1}, x \rangle}, \tfrac{\langle Q^*\alpha_{i_3}, x \rangle}{\langle Q^*\alpha_{i_1}, x \rangle} \right) :\ x \in Q^* S_i \right\} 
= L_i(A).
\end{align*}
That is, applying an orthogonal transformation to a weighted pseudocircle arrangement in $A \in Z_I$, does not change its image in the projective plane, since the projection $\proj_{\projplane}(\alpha_{i_1},\alpha_{i_2},\alpha_{i_3})$ also changes by the same transformation.
Therefore, $L_j(\eta(A,t)*Q) = L_j(\eta(A,t))$,
and $\tilde \lambda(L_I(A*Q),L_j(A*Q),t) = \tilde \lambda(L_I(A),L_j(A),t)$, 
so $L_j(\eta(A*Q,t)) = L_j(\eta(A,t))$, which means $L_j(\eta(A*Q,t)) = L_j(\eta(A,t)*Q)$.
Also $\proj_I(\eta(A*Q,t)) = \proj_I(A*Q) = \proj_I(\eta(A,t)*Q)$, 
so $\eta(A*Q,t) = \eta(A,t)*Q$, which means $\eta$ is $\orth_3$-equivariant.
\end{claimproof}

We will define $h_I$ as the deformation of $Z_I$ induced by an extension deformation process $\tilde \lambda$ as in Claim~\ref{claim-pseudoline-lift}.
We first choose some arbitrary $A \in Z_{I}$ and let $\tilde A = L_I(A) \in \tilde Z_{I}$, and choose an arbitrary pseudoline extension $\lambda(0)$ that is linear on the cells of $\tilde A$.  To define $\tilde \lambda$, we define a deformation $\lambda(t)$ of the pseudoline $\lambda(0)$, and then let $\tilde \lambda(\tilde A,\lambda(0),t) = \lambda(t)$.  Note that $\lambda(t)$ may depend on $\tilde A = (L_{i_1},\dots,L_{i_m})$ and $\lambda(0)$.

\paragraph{The case m$>$3.}
For $m >3$, $h_I$ %\marginpar{$h_I$, $m>3$} 
is the deformation of $Z_I$ induced by the extension deformation process $\tilde \lambda$ where $\lambda(t) = \tilde \lambda(\tilde A, \lambda(0), t)$ is defined as follows.
Let $\lambda(t) = \lambda(0)$ be fixed unless $\lambda(0)$ intersects $L_{i_m}$ at a single point $p(0)$ that is in the interior of a 2-cell $C$ of the subdivision of $\mb{R}^2$ by $L_{i_2},\dots,L_{i_{m-1}}$.
Otherwise, let $\lambda$ be fixed on the complement of $C$, and let $\lambda$ evolve in $C$ as follows; see Figure \ref{figure-hI>3}.
Let $a,b$ be the points where $\lambda(0)$ meets the boundary of $C$, and
$p(1)$ be the point where the segment $[a,b]_{C}$ intersects $L_{i_m}$,
and 
\begin{gather*}
p(t) = t p(1) + (1-t) p(0), \\
 \lambda(t) \cap C = \left[a,p(t),b\right]_{C}. 
\end{gather*}

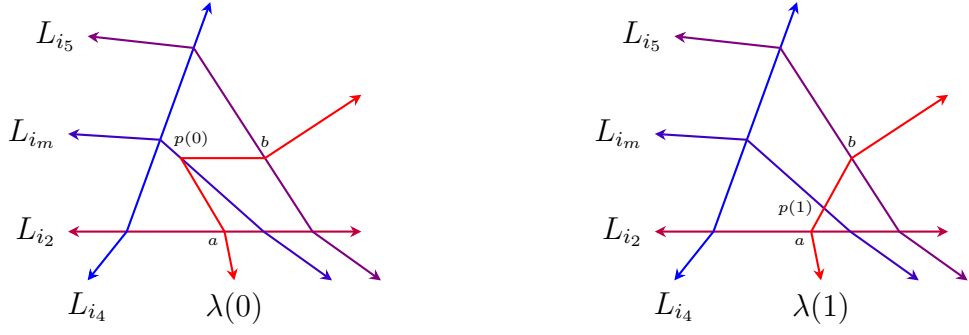
\begin{figure}
\begin{center}
\begin{tikzpicture}[scale=1.3,every node/.style={black}]

\draw[thick,purple,<->]
(-1,0) node[left] {$L_{i_2}$} -- (2,0);

\draw[blue,thick,<->]
(-.8,-.5) node[below] {$L_{i_4}$} -- (-.4,0) -- ++(70:1) coordinate (a) -- ++(70:1) coordinate (b) -- ++(70:.5);

\draw[violet,thick,<->]
(-.8,2) node[left] {$L_{i_5}$} -- (b) -- (1.5,0) coordinate (c) -- (2.2,-.5);

\draw[blue!50!violet,thick,<->]
(-1,1) node[left] {$L_{i_m}$} -- (a) -- (1,0) coordinate (d) -- (1.7,-.5);

\draw[red,thick,<->]
(.7,-.5) node[below] {$\lambda(0)$} -- (.6,0) node[below left,shift={(2pt,2pt)}] {\tiny $a$} -- ($(a)!.2!(d)$) node[above,shift={(4pt,0pt)}] {\tiny $p(0)$} -- ($(b)!.6!(c)$) node[above] {\tiny $b$} -- (2,1.4);

\begin{scope}[shift={(6,0)}]

\draw[thick,purple,<->]
(-1,0) node[left] {$L_{i_2}$} -- (2,0);

\draw[blue,thick,<->]
(-.8,-.5) node[below] {$L_{i_4}$} -- (-.4,0) -- ++(70:1) coordinate (a) -- ++(70:1) coordinate (b) -- ++(70:.5);

\draw[violet,thick,<->]
(-.8,2) node[left] {$L_{i_5}$} -- (b) -- (1.5,0) coordinate (c) -- (2.2,-.5);

\draw[blue!50!violet,thick,<->]
(-1,1) node[left] {$L_{i_m}$} -- (a) -- (1,0) coordinate (d) -- (1.7,-.5);

\draw[red,thick,<->]
(.7,-.5) node[below] {$\lambda(1)$} -- (.6,0) node[below left,shift={(2pt,2pt)}] {\tiny $a$} coordinate (e) -- ($(b)!.6!(c)$) node[above] {\tiny $b$} coordinate (f) -- (2,1.4);

\node[left] at (intersection of a--d and e--f) {\tiny $p(1)$};

\end{scope}

\end{tikzpicture}
\caption{Points used to define $\lambda(t)$ for $m>3$.}
\label{figure-hI>3}
\end{center}
\end{figure}

\Needspace{2cm}
\begin{claim}\label{claim-lambda-def}
For $m >3$, $\lambda(t)$ is well defined and is a pseudoline extension of $L_{i_1},\dots,L_{i_m}$.
\end{claim}

\begin{claimproof}
We may assume that $\lambda(0)$ is a pseudoline extension of $(L_{i_1},\dots,L_{i_{m-1}})$ that passes though the interior of the cell $C$ where it intersects $L_{i_m}$, otherwise the deformation is trivial.
Since $\lambda(0)$ is a pseudoline extension, $\lambda(0)$ intersects $C$ in a single connected component, so $\lambda(0) \cap C$ is a polygonal path with a pair of well defined endpoints $a,b$, and at most one of these points may be on the horizon $L_{i_1}$.
By Claim \ref{claim-sphericonvex} and $Z_I \subset Y_I$ we know that $S_{i_1},\dots,S_{i_{m-1}}$ subdivide $\sphere^2$ into spherical convex polygons, so $C$ is a convex polygonal region of $\mb{R}^2 = \projplane \setminus L_{i_1}$.  Note that $C$ might be unbounded.  
From the definition of $Z_I$, we know that $L_{i_m} \cap C$ is a segment through $C$ that subdivides $C$ into two convex polygonal regions. 
From the definition of $C$, we know that $\lambda(0)$ crosses $L_{i_m}$ at a single point $p(0) \in C$, which implies that $a,b$ are separated in $C$ by $L_{i_m}$, and therefore the segment $[a,b]_{C}$ intersects $L_{i_m}$ at a single point $p(1) \in C$.
Hence $p(t) \in (L_{i_m} \cap C)$ is well defined, and $\lambda(t) \cap C$ is a well defined polygonal path that intersects $L_{i_m}$ at a single point.  Since $\lambda(t)$ is fixed outside of $C$ and $\lambda(t)$ is a path in $C$ between fixed endpoints on the boundary of $C$, $\lambda(t)$ is a simple closed curve.  Since $\lambda(t)$ is a deformation of the pseudoline $\lambda(0)$, $\lambda(t)$ is a pseudoline. 
By definition, $\lambda(0)$ starts as a pseudoline extension of $L_{i_1},\dots,L_{i_m}$, and $\lambda(t)$ only deforms in the interior of the cell $C$ of $(L_{i_1},\dots,L_{i_{m-1}})$ where it meets $L_{i_m}$ at a single point $p(t)$, so $\lambda(t)$ is a pseudoline extension of $(L_{i_1},\dots,L_{i_{m-1}})$ throughout the deformation.
\end{claimproof}

\begin{claim}\label{claim-lambda-continuous}
For $m >3$, $\tilde \lambda$ is continuous on the domain where it is defined.  Hence, $\tilde \lambda$ is an extension deformation process on $\tilde Z_I$.
\end{claim}

\begin{claimproof}
Let $\tilde A_k = (L_{k,i_1},\dots,L_{k,i_m}) \in \tilde Z_{I}$, $\lambda_k(0) \in \curves$ be a pseudoline extension that is linear on cells of $\tilde A_k$, and $t_k \in [0,1]_\mb{R}$ such that $\tilde A_k \to \tilde A_\infty$, $\lambda_k(0) \to \lambda_\infty(0)$, and $t_k \to t_\infty$.
Note that $L_{k,i_1} = L_{i_1}$, $L_{k,i_2} = L_{i_2}$, and $L_{k,i_3} = L_{i_3}$ are fixed at the horizon, horizontal axis, and vertical axis respectively.
Let $\lambda_k(t) = \tilde \lambda(\tilde A_k,\lambda_k(0),t)$ be defined as above.
Our goal is to show that $\lambda_k(t_k) \to \lambda_\infty(t_\infty)$.

Since there are only finitely many rank 3 sign hyperfield-vector sets on $n$ elements, we may assume that $\tilde A_k$ has the same covector set for all $k \in \mb{N}$, otherwise partition into finitely many subsequences by the covector sets of the $\tilde A_k$ and show convergence to $\lambda_\infty(t_\infty)$ for each subsequence separately. 

If $\lambda_k(0) = L_{k,i_m}$, then $\lambda_k$ is fixed throughout the deformation, so the limit converges and we are done.
Otherwise, let $p_k(0) = \lambda_k(0) \cap L_{k,i_m}$.
If $p_{k}(t)$ is on one of the pseudolines $L_{k,i_1},\dots,L_{k,i_{m-1}}$, then $\lambda_k$ is fixed again and we are done.
Otherwise, let $C_k$ be the 2-cell of the subdivision of $\mb{R}^2$ by $L_{k,i_2},\dots,L_{k,i_{m-1}}$ that contains $p_{k}(t)$ in its interior.

By Lemma~\ref{lemma-vertex}, the vertices of $\tilde A_k$ converge to the corresponding vertices of $\tilde A_\infty$.  
Consequently, by Lemma~\ref{lemma-path} the 1-cells of $\tilde A_k$ converge to corresponding 1-cells or vertices of $\tilde A_\infty$.
In particular, $p_k(0) \to p_\infty(0)$ and $\{a_k,b_k\} \to \{a_\infty,b_\infty\}$ where these are defined in the same way as $a,b$ above.
Since $L_{k,i_m} \cap C_k$ is a segment approaching $L_{\infty,i_m} \cap C_\infty$ and $\lambda_k(0) \cap C_k$ is a distinct segment approaching $L_{\infty,i_m} \cap C_\infty$, we have $p_k(1) \to p_\infty(1)$, since the intersection point of a pair of non-parallel segments depends continuously on their end points. 
Since $\lambda_k(t)$ is defined continuously in terms of $\lambda_k(0)$, $a_k$, $b_k$, $p_k(0)$, and $p_k(1)$, we have $\lambda_k(t_k) \to \lambda_\infty(t_\infty)$.
Thus, $\tilde \lambda$ is continuous on the domain where it is defined, and therefore is an extension deformation process on $\tilde Z_I$.
\end{claimproof}

Let $\Lambda(A,t) = \left(\tilde \lambda(\tilde A,L_1,t),\dots,\tilde \lambda(\tilde A,L_n,t)\right)$. 
%where $\tilde A = L_I(A)$ and $L_i=L_i(A)$.

\begin{claim}\label{claim-lambda-arrangement}
For $m>3$, $\Lambda(A,t)$ is a pseudoline arrangement.  Hence, $h_I$ is a well defined $\orth_3$-equivariant deformation of $Z_I$.
\end{claim}

\newcommand{\xj}{{j'}}

\begin{claimproof}%(Proof of Claim \ref{claim-deform-h})
Let $J = \{ j \in \compl{I}: \|\alpha_j\| > 0 \}$, and let $\lambda_j(t) = \tilde\lambda(\tilde A, L_j,t)$.  
Since the pseudolines $\lambda_j(t)$ deform according to their initial positions, any initially identical pairs of pseudolines remain identical throughout the deformation.  Consider $L_j,L_\xj$ distinct.  We may assume that at least one of these is deforming, and it suffices to show that the number of crossings does not change in the cell where it deforms.
Assume $L_j$ crosses $L_{i_m}$ at a point $p(0)$ in the interior of a 2-cell $C$ of $L_{i_1},\dots,L_{i_{m-1}}$.

Suppose $L_j,L_\xj$ meet in the interior of the cell.  Then, $L_j,L_\xj$ alternate around the boundary of $C$, and since points on the boundary of $C$ remain fixed, $\lambda_j(t),\lambda_\xj(t)$ meet in $C$ for all $t$.  Since $\lambda_j(t),\lambda_\xj(t)$ are linear segments in the cells $C_1,C_2 \subset C$ divided by $L_{i_m}$, they can meet in at most one point in each cell $C_1,C_2$.  If $\lambda_j(t),\lambda_\xj(t)$ met at a point in both $C_1$ and $C_2$, then they would alternate around the boundary of both cells, which would imply that they do not alternate around the boundary of $C$.
Thus, $\lambda_j(t),\lambda_\xj(t)$ meet at only a single point in $C$ for all $t$.

Now suppose $L_j,L_\xj$ do not meet in the interior $C \in \mc{C}$.  
Then, they do not alternate around the boundary of $C$, $C_1$, or $C_2$,  
and we may assume $L_j,L_\xj$ both meet the segment $L_{i_m} \cap C$, otherwise they would never alternate around $C$, $C_1$, or $C_2$ as a result of the deformation. 
Since $\lambda_j(t),\lambda_\xj(t)$ are fixed on the boundary of $C$, they still do not alternate around the boundary of $C$ at $t=1$, so $\lambda_j(1),\lambda_\xj(1)$ do not meet in $C$, which implies that they also do not alternate around the boundary of $C_1$ or $C_2$.  
Therefore, the order of the points $p_j(t),p_{\xj}(t)$ where $\lambda_j(t),\lambda_\xj(t)$ meet the segment $L_{i_m} \cap C$ is the same for $t=0$ as for $t=1$. 
That is, $p_{j}(1)-p_{\xj}(1)= r (p_{j}(0)-p_{\xj}(0))$ for some $r >0$, so 
$p_{j}(t)-p_{\xj}(t)= (tr +(1-t)) (p_{j}(0)-p_{\xj}(0))$, which implies $\lambda_j(t),\lambda_\xj(t)$ do not alternate around the boundary of $C_1$ or $C_2$ for all $t$.  Hence, $\lambda_j(t),\lambda_\xj(t)$ never meet in $C$ throughout the deformation.

Thus, each pair of pseudolines $\lambda_j(t),\lambda_\xj(t)$ either coincide or cross exactly once throughout the deformation, so $\Lambda(A,t)$ is a pseudoline arrangement, and therefore by Claims \ref{claim-lambda-continuous} and \ref{claim-pseudoline-lift}, $h_I$ is well defined as the equivariant deformation induced by $\tilde\lambda$.
\end{claimproof}

\begin{claim}\label{claim-deform-h}
For $m >3$, $h_I = h_{(i_1,\dots,i_m)}$ is a strong $\orth_3$-equivariant deformation retraction from $Z_{(i_1,\dots,i_m)}$ to $Z_{(i_1,\dots,i_{m-1})}$.
\end{claim}

\begin{claimproof}
We have already that $h_I$ is an $\orth_3$-equivariant deformation of $Z_I$ by Claim \ref{claim-lambda-arrangement}.
Since $\lambda(t)$ is linear in each cell of $\tilde A = (L_{i_1},\dots,L_{i_{m}})$, we have $h_I(A,t) \in Z_I$, 
and since $\lambda(1)$ is linear in each cell of $(L_{i_1},\dots,L_{i_{m-1}})$, we have $h_I(A,1) \in Z_{(i_1,\dots,i_{m-1})}$,
so $h_I$ is a deformation retraction from $Z_I$ to $Z_{(i_1,\dots,i_{m-1})}$. 
For $A \in Z_{(i_1,\dots,i_{m-1})}$ and $\lambda(0) = L_j$, we have that $\lambda(0)$ is linear in each cell of $(L_{i_1},\dots,L_{i_{m-1}})$, so $p(1) = p(0)$, which implies $\lambda(t) = \lambda(0)$ is a trivial deformation, and therefore $h_I(A,t) = h_I(A,0)$ is trivial.  Thus, $h_I$ is a strong deformation retraction.
\end{claimproof}

\paragraph{The case m$=$3.}
For $m=3$, let $h_I$ be the deformation of $Z_I$ induced by the extension deformation process $\tilde \lambda$ defined by $\lambda(t) = \tilde \lambda(\tilde A, \lambda(0), t)$
for a pseudoline extension $\lambda(0)$ of $\{\tilde A\} = \{(L_{i_1},L_{i_2},L_{i_3})\} = \tilde Z_{I}$ as follows.

If $\lambda(0)$ is a straight line, 
then let $\lambda$ be fixed. 
Equivalently, $\lambda$ is fixed unless $\lambda(0)$ intersects $L_{i_1}$, $L_{i_2}$, $L_{i_3}$ at three distinct points that are not collinear. 
Assume this is so, and let $p_k(0) = \lambda(0) \cap L_{i_k}$ for $k \in \{1,2,3\}$.  
% Let $C_i$ be the quadrant of the plane that contains the segment of $\lambda(0)$ with vertices $p_j(0)$ and $p_k(0)$ for $i,j,k$ all distinct.

We define $\lambda(t)$ as a polygonal path with moving points $p_k(t)$ as vertices.
For points $a,b \in \projplane$, let $[a,b]_\oplus$ denote the segment between $a,b$ contained in a single quadrant of $\projplane$ when such a segment exists and is unique, and let $[a,b,c,\dots]_\oplus = [a,b]_\oplus \cup [b,c]_\oplus \cup \dots$.  With this we have $\lambda(0) = [p_1(0),p_2(0),p_3(0),p_1(0)]_{\oplus}$.

Let $p_1 = p_1(t) = p_1(0)$ be fixed throughout the deformation. 
Let $P$ be the line through the origin $0$ that is perpendicular to line though $\{0,p_1\}$.
Let $q_0$ be the point where the line $P$ meets $\lambda(0)$; see Figure \ref{figure-q0}.

We define $\lambda(t)$ as the polygonal path consisting of a segment in one quadrant that pivots around the point $q_0$ and then extends beyond the axis as rays in the fixed direction of $p_1$.
We pivot at such a rate that the exterior angle at the vertices of $\lambda(t)$ is $(\nicefrac{\pi}{2})(1-t)$ once this becomes smaller than the initial exterior angle at $t=0$; see Figure \ref{figure-hI3}.
Specifically, let $\lambda(1)$ be the line though $\{q_0,p_1\}$, let $\phi(0) \in (-\nicefrac{\pi}{2},\nicefrac{\pi}{2})_\mb{R}$ be the signed angle from $\lambda(1)$ to the line though $\{p_2(0),p_3(0)\}$ and $\phi(t) = \min(\phi(0),(\nicefrac{\pi}{2})(1-t))$, let $Q(t)$ be the line though $q_0$ at angle $\phi(t)$ from $\lambda(1)$, and let $p_2(t) = Q(t) \cap L_{i_2}$ and $p_3(t) = Q(t) \cap L_{i_3}$.
Finally, let
\[ \lambda(t) = [p_1,p_2(t),p_3(t),p_1]_{\oplus}. \]

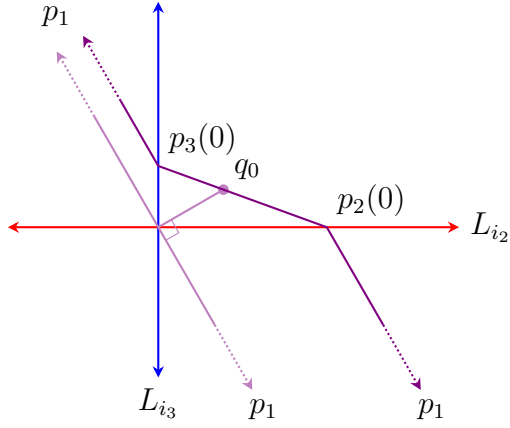
\begin{figure}
\begin{center}

\begin{tikzpicture}

\draw[thick,red,<->]
(-2,0) -- (4,0);

\draw[thick,blue,<->]
(0,-2) -- (0,3);

%\fill[violet!50]
%(.5,.5) circle (2pt);
%\draw[thick,violet,<->]
%(0,.75) ++(-45:-2) -- (0,.75) -- (1.5,0) -- ++(-45:2);

\def\A{-60}
\def\B{-20}

\fill[violet!50]
({\A+90}:1) coordinate (q0) circle (2pt);

\draw[thick,violet!50]
(0,0) coordinate (o) -- (q0)
(o) +(\A:-1.7) coordinate (a) -- ++(\A:1.5) coordinate (b)
;
\draw[thin,violet!50]
(o) -- (\A:.2) -- ++({\A+90}:.2) -- ++(\A:-.2) -- cycle;
\draw[thick,violet!50,densely dotted,->]
(b) -- ++(\A:1) node[shift={(\A:.3)},black] {$p_1$};
\draw[thick,violet!50,densely dotted,->]
(a) -- ++(\A:-1); %node[shift={(\A:-.3)},black] {$p_1$};

\path
(q0) ++(\B:2) coordinate (a)
(o) ++(3,0) coordinate (b) 
(intersection of q0--a and o--b) coordinate (p2) ++(\A:1.5) coordinate (p1b)
(q0) ++(\B:-2) coordinate (a)
(o) ++(0,3) coordinate (b) 
(intersection of q0--a and o--b) coordinate (p3) ++(\A:-1) coordinate (p1a)
;
\draw[thick,violet]
(p1a) -- (p3) -- (p2) -- (p1b);
\draw[thick,violet,densely dotted,->]
(p1b) -- ++(\A:1) node[shift={(\A:.3)},black] {$p_1$};
\draw[thick,violet,densely dotted,->]
(p1a) -- ++(\A:-1);
\path 
(p1a) ++(\A:-1.3) ++(-.2,0) node {$p_1$}
(p2) node[above right] {$p_2(0)$}
(p3) node[above right] {$p_3(0)$}
(q0) node[above right] {$q_0$}
(4,0) node[right] {$L_{i_2}$}
(0,-2) node[below] {$L_{i_3}$}
;

\end{tikzpicture}

\caption{Points used to define $\lambda(t)$ for $m=3$.}
\label{figure-q0}

\end{center}
\end{figure}

\begin{figure}

\begin{center}
\begin{tikzpicture}

\begin{scope}[shift={(-7,0)}]%,scale=.9]

\draw[thick,red,<->]
(-1,0) -- (4.5,0);
\draw[thick,blue,<->]
(0,-1) -- (0,3);
\path 
(0,0) coordinate (o);

\def\A{-55}
\def\B{-10}

\fill[violet!50]
({\A+90}:1) coordinate (q0) circle (3pt);
\path
(q0) ++(\B:2) coordinate (a)
(o) ++(3,0) coordinate (b) 
(intersection of q0--a and o--b) coordinate (p2) ++(\A:1) coordinate (p1b)
(q0) ++(\B:-2) coordinate (a)
(o) ++(0,3) coordinate (b) 
(intersection of q0--a and o--b) coordinate (p3) ++(\A:-1.5) coordinate (p1a)
;
\draw[thick,violet,<->]
(p1a) -- (p3) -- (p2) -- (p1b);
\draw[thin,->]
(q0) ++(\B:-.5) arc ({\B-180}:{\A-180}:.5) node[left] {\tiny $\tfrac{\pi}{4}$}
(q0) -- ++(\A:-.8)
;

\def\A{-70}
\def\B{-40}

\fill[purple!50]
({\A+90}:1.8) coordinate (q0) circle (3pt);
\path
(q0) ++(\B:2) coordinate (a)
(o) ++(3,0) coordinate (b) 
(intersection of q0--a and o--b) coordinate (p2) ++(\A:1) coordinate (p1b)
(q0) ++(\B:-2) coordinate (a)
(o) ++(0,3) coordinate (b) 
(intersection of q0--a and o--b) coordinate (p3) ++(\A:-1) coordinate (p1a)
;
\draw[thick,purple,<->]
(p1a) -- (p3) -- (p2) -- (p1b);
\draw[thin,->]
(q0) ++(\B:-.7) arc ({\B-180}:{\A-180}:.7) node[left,shift={(0,2pt)}] {\tiny $\tfrac{\pi}{6}$}
(q0) -- ++(\A:-1)
;

\node at (2.5,2.5) {$\begin{array}{r@{\ }l@{\vspace{3pt}}} t &=0 \text{ to } \frac{1}{2} \\ \frac{\pi(1{-}t)}{2} &=0 \text{ to } \frac{\pi}{4} \end{array}$};

\end{scope}

\begin{scope}[shift={(0,0)}]%,scale=.9]

\draw[thick,red,<->]
(-1,0) -- (3.5,0);
\draw[thick,blue,<->]
(0,-1) -- (0,3);
\path 
(0,0) coordinate (o);

\def\A{-55}
\def\B{-25}

\fill[violet!50]
({\A+90}:1) coordinate (q0) circle (3pt);
\path
(q0) ++(\B:2) coordinate (a)
(o) ++(3,0) coordinate (b) 
(intersection of q0--a and o--b) coordinate (p2) ++(\A:1) coordinate (p1b)
(q0) ++(\B:-2) coordinate (a)
(o) ++(0,3) coordinate (b) 
(intersection of q0--a and o--b) coordinate (p3) ++(\A:-1.5) coordinate (p1a)
;
\draw[thick,violet,<->]
(p1a) -- (p3) -- (p2) -- (p1b);
\draw[thin,->]
(q0) ++(\B:-.6) arc ({\B-180}:{\A-180}:.6) node[left,shift={(0,1pt)}] {\tiny $\tfrac{\pi}{6}$}
(q0) -- ++(\A:-.9)
;

\def\A{-70}
\def\B{-40}

\fill[purple!50]
({\A+90}:1.8) coordinate (q0) circle (3pt);
\path
(q0) ++(\B:2) coordinate (a)
(o) ++(3,0) coordinate (b) 
(intersection of q0--a and o--b) coordinate (p2) ++(\A:1) coordinate (p1b)
(q0) ++(\B:-2) coordinate (a)
(o) ++(0,3) coordinate (b) 
(intersection of q0--a and o--b) coordinate (p3) ++(\A:-1) coordinate (p1a)
;
\draw[thick,purple,<->]
(p1a) -- (p3) -- (p2) -- (p1b);
\draw[thin,->]
(q0) ++(\B:-.6) arc ({\B-180}:{\A-180}:.6) node[left,shift={(0,1pt)}] {\tiny $\tfrac{\pi}{6}$}
(q0) -- ++(\A:-.9)
;

\node at (2.5,2.5) {$\begin{array}{r@{\ }l@{\vspace{3pt}}} t &=\frac{2}{3} \\ \frac{\pi(1{-}t)}{2} &=\frac{\pi}{6} \end{array}$};

\end{scope}

\begin{scope}[shift={(-7,-5)}]%,scale=.9]

\draw[thick,red,<->]
(-1,0) -- (3.5,0);
\draw[thick,blue,<->]
(0,-1) -- (0,3);
\path 
(0,0) coordinate (o);

\def\A{-55}
\def\B{-40}

\fill[violet!50]
({\A+90}:1) coordinate (q0) circle (3pt);
\path
(q0) ++(\B:2) coordinate (a)
(o) ++(3,0) coordinate (b) 
(intersection of q0--a and o--b) coordinate (p2) ++(\A:1) coordinate (p1b)
(q0) ++(\B:-2) coordinate (a)
(o) ++(0,3) coordinate (b) 
(intersection of q0--a and o--b) coordinate (p3) ++(\A:-1) coordinate (p1a)
;
\draw[thick,violet,<->]
(p1a) -- (p3) -- (p2) -- (p1b);
\draw[thin,->]
(q0) ++(\B:-.7) arc ({\B-180}:{\A-180}:.7) node[above left,shift={(1pt,0)}] {\tiny $\tfrac{\pi}{12}$}
(q0) -- ++(\A:-.9)
;

\def\A{-70}
\def\B{-55}

\fill[purple!50]
({\A+90}:1.8) coordinate (q0) circle (3pt);
\path
(q0) ++(\B:2) coordinate (a)
(o) ++(3,0) coordinate (b) 
(intersection of q0--a and o--b) coordinate (p2) ++(\A:1) coordinate (p1b)
(q0) ++(\B:-2) coordinate (a)
%(o) ++(0,3) coordinate (b) 
%(intersection of q0--a and o--b) coordinate (p3) ++(\A:-2) coordinate (p1a)
;
\draw[thick,purple,<->]
(a) -- (p2) -- (p1b);
\draw[thin,->]
(q0) ++(\B:-.7) arc ({\B-180}:{\A-180}:.7) node[above left,shift={(2pt,0)}] {\tiny $\tfrac{\pi}{12}$}
(q0) -- ++(\A:-.9)
;

\node at (2.5,2.5) {$\begin{array}{r@{\ }l@{\vspace{3pt}}} t &=\frac{5}{6} \\ \frac{\pi(1{-}t)}{2} &=\frac{\pi}{12} \end{array}$};

\end{scope}

\begin{scope}[shift={(0,-5)}]%,scale=.9]

\draw[thick,red,<->]
(-1,0) -- (3.5,0);
\draw[thick,blue,<->]
(0,-1) -- (0,3);
\path 
(0,0) coordinate (o);

\def\A{-55}

\fill[violet!50]
({\A+90}:1) coordinate (q0) circle (3pt);
\path
(q0) ++(\A:1.8) coordinate (a)
(q0) ++(\A:-2) coordinate (b)
;
\draw[thick,violet,<->]
(a) -- (b);
\draw[thin,->]
(q0) -- ++(\A:-.9)
;

\def\A{-70}

\fill[purple!50]
({\A+90}:1.8) coordinate (q0) circle (3pt);
\path
(q0) ++(\A:1.8) coordinate (a)
(q0) ++(\A:-2) coordinate (b)
;
\draw[thick,purple,<->]
(a) -- (b);
\draw[thin,->]
(q0) -- ++(\A:-.9)
;

\node at (2.5,2.5) {$\begin{array}{r@{\ }l@{\vspace{3pt}}} t &=1 \\ \frac{\pi(1{-}t)}{2} &=0 \end{array}$};

\end{scope}

\end{tikzpicture}

\caption{An example of the deformation $h_I$ for $m=3$.}
\label{figure-hI3}

\end{center}

\end{figure}
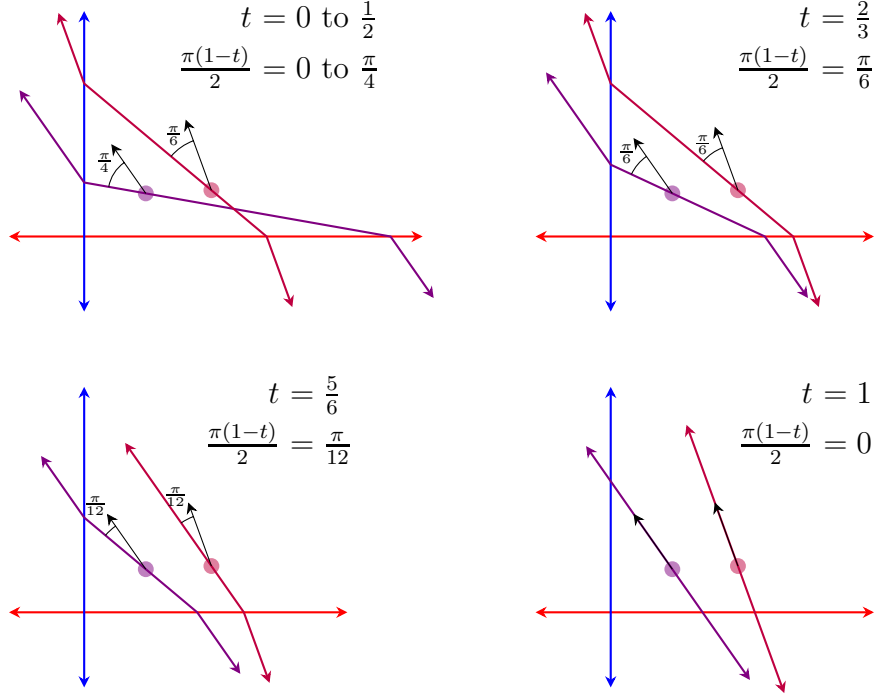

\Needspace{5cm}
\begin{claim}\label{claim-lambda-def-3}
For $m =3$, $\lambda(t)$ is well defined and is a pseudoline extension of $L_{i_1},L_{i_2},L_{i_3}$.
\end{claim}

\begin{claimproof}
Assume that we are in the case where $\lambda(0)$ is not straight and that $(\nicefrac{\pi}{2})(1-t) < |\delta|$, since the claim is trivial otherwise. 
We may assume by symmetry that $p_3(0)$ is above the origin on the vertical axis, and $p_2(0)$ is to the right of the origin on the horizontal axis.

We show that $q_0$ is well defined and is on the segment between $p_2(0)$ and $p_3(0)$.  
By our assumption, $\lambda(0)$ is an implicit function with a finite negative slope in the upper-right quadrant.
Also, our assumption implies that $\lambda(0)$ is unbounded on $\mb{R}^2$ in the upper-left and lower-right quadrants, and $\lambda(0)$ is neither vertical nor horizontal there since $p_1$ is not on $L_{i_2}$ or $L_{i_3}$, so $\lambda(0)$ is an implicit function with a finite negative slope in both of these quadrants as well.  
%Since $\lambda(0) \cap \mb{R}^2$ is an implicit function that has a finite negative slope, $\lambda(0)$ cannot bend by $\nicefrac{\pi}{2}$ or more. 
Since lines through $p_1$ have finite negative slope, $P$ is an implicit function with a finite positive slope, and therefore, $\lambda(0)$ and $P$ meet at a unique point $q_0$ in the upper-right quadrant, so the point $q_0$ is between $p_2(0)$ and $p_3(0)$ on $\lambda(0)$.

Any line through $q_0$ with finite negative slope intersects $L_{i_2}$ in $\mb{R}^2$ to the right of the origin and intersects $L_{i_3}$ in $\mb{R}^2$ above the origin.
Therefore, the points $p_2(t)$ and $p_3(t)$ move along $L_{i_2}$ and $L_{i_3}$ respectively without crossing the origin or leaving the plane.
Hence, $\lambda(t)$ is a well defined path consisting of a segment in the upper-left, upper-right, and lower-right quadrants each.
Furthermore, we now have that $\lambda(t)$ crosses $L_{i_1}$, $L_{i_2}$, and $L_{i_3}$ once each, so this is a pseudoline arrangement.
\end{claimproof}

\begin{claim}\label{claim-lambda-continuous-3}
For $m =3$, $\tilde \lambda$ is continuous on the domain where it is defined.  Hence, $\tilde \lambda$ is an extension deformation process on $\tilde Z_I$.
\end{claim}

\begin{claimproof}
For $j \in \{1,\dots,\infty\}$, let $\lambda_j(t)$ be as above and $t_j \in [0,1]_\mb{R}$ such that $\lambda_j(0) \to \lambda_\infty(0)$ and $t_j \to t_\infty$.  
Recall that we use the metric on $\projplane$ induced from Fréchet distance on the sphere.  

If $\lambda_j(0)$ is straight for all $j \in \mb{N}$ large enough, then $\lambda_j(t) = \lambda_j(0) \to \lambda_\infty(0) = \lambda_\infty(t')$ and we are done.
Otherwise we may restrict to a subsequence that is not straight.
Therefore, assume that $\lambda_j(0)$ is not a straight line for $j\neq \infty$. 

Let $p_{1,j}$, $p_{2,j}(t)$, $p_{3,j}(t)$, $P_j$, $q_{0,j}$, and $Q_{j}(t)$ be defined as above for $j < \infty$ and also for $j=\infty$ where appropriate.

We have four cases to consider, $\lambda_\infty(0)$ intersects $L_{i_1},L_{i_2},L_{i_3}$ at 3 distinct points, or $\lambda_\infty(0)$ is vertical or horizontal, or $\lambda_\infty(0)$ is a straight line through the origin that is neither vertical nor horizontal, or $\lambda_\infty(0) = L_{i_1}$ is the horizon.

Suppose that $\lambda_\infty(0)$ intersects $L_{i_1},L_{i_2},L_{i_3}$ at 3 distinct points.  Then by Lemma~\ref{lemma-vertex}, we have $p_{k,j}(0) \to p_{k,\infty}(0)$, so $q_{0,j} \to q_{0,\infty}$, so $\lambda_j(1) \to \lambda_\infty(1)$, so $Q_{j}(t_j) \to Q_{\infty}(t_\infty)$, so $p_{k,j}(t_j) \to p_{k,\infty}(t_\infty)$, so $\lambda_j(t_j) \to \lambda_\infty(t_\infty)$ since these are defined continuously in terms of each other in succession.

Suppose that $\lambda_\infty(0)$ is vertical.  Then, $p_1$ and $p_{3,j}(0)$ both converge to $(L_{i_1} \cap L_{i_3})$, so $P_j \to L_{i_2}$, so $q_{0,j} \to p_{2,\infty}(0) = (\lambda_\infty(0) \cap L_{i_2})$, so $\lambda_j(t_j)$ converges to the vertical line through $p_{2,\infty}(0)$, which is $\lambda_\infty(t_\infty) = \lambda_\infty(0)$.
The argument for $\lambda_\infty(0)$ horizontal is essentially the same.

Suppose that $\lambda_\infty(0)$ contains the origin and is neither vertical nor horizontal.  Then, $p_{2,j}(0) \to 0$ and $p_{3,j}(0) \to 0$, so $q_{0,j} \to 0$.  Since $p_{1,j} \to p_{1,\infty} = (\lambda_\infty(0) \cap L_{i_1})$, we have $p_{1,j}$ bounded away from $L_{i_2}$ and $L_{i_3}$ for $j$ large enough, so $p_{k,j}(1) \to 0$ for $k \in \{2,3\}$. Since $p_{k,j}(t_k)$ is on $L_{i_k}$ between $p_{k,j}(0)$ and $p_{k,j}(1)$ in $\mb{R}^2$, we have $p_{k,j}(t_j) \to 0$, so $\lambda_j(t_j)$ converges to the line through the origin and $p_{1,\infty}$, which is $\lambda_\infty(t_\infty) = \lambda_\infty(0)$.

Suppose that $\lambda_\infty(0)$ is the horizon.
% Then, $p_{1,j}$ might not converge, but $p_{2,j}(0) \to (L_{i_1}\cap L_{i_2})$ and $p_{3,j}(0) \to (L_{i_1}\cap L_{i_3})$.  
% The point on $\lambda_j(t_j)$ closest to the origin is always on the segment between
Then, $\min \{\|x\| : x \in \lambda_j(0)\} \to \infty$, so 
although $q_{0,j}$ might not converge, $\|q_{0,j}\| \to \infty$.
Since $q_{0,j} = (P_{j} \cap \lambda_j(1))$, and $P_{j}$ and $\lambda_j(1)$ are perpendicular, we have $\min \{\|x\| : x \in \lambda_j(1)\} = \|q_{0,j}\| \to \infty$.
Since $\lambda_j(t)$ pivots about the point $q_{0,j}$ in one quadrant, and is parallel in two other quadrants, $\lambda_j(t_j)$ is separated from the origin by the lower envelope of $\lambda_j(0)$ and $\lambda_j(1)$, so $\min \{\|x\| : x \in \lambda_j(t_j)\} \to \infty$, which means $\lambda_j(t_j)$ converges to the horizon.  
\end{claimproof}

Let $\Lambda(A,t) = \left(\tilde \lambda(\tilde A,L_1,t),\dots,\tilde \lambda(\tilde A,L_n,t)\right)$. 

\begin{claim}\label{claim-lambda-arrangement-3}
For $m=3$, $\Lambda(A,t)$ is a pseudoline arrangement.  Hence, $h_I$ is a well defined $\orth_3$-equivariant deformation of $Z_I$.
\end{claim}

\begin{claimproof}
Let $\lambda_j(t)=\tilde \lambda(\tilde A, L_j,t)$, and again let $p_{1,j}$, $p_{2,j}(t)$, $p_{3,j}(t)$, $P_j$, $q_{0,j}$, and $\phi_j(t)$ be defined for $\lambda_j$ as above.

Assume for the sake of contradiction that $\Lambda(A,t)$ is not a pseudoline arrangement for some $t\in [0,1]_\mb{R}$.
Since we start with a pseudoline arrangement $\Lambda(A,0)$, there is some minimum $t_0 > 0$ where $\Lambda(A,t_0)$ is not a pseudoline arrangement, and there is some pair $j,\xj \in [n]_\mb{N}\setminus I$ distinct such that $\lambda_j(t_0)$ and $\lambda_\xj(t_0)$ intersect at more than 1 point.

% We show that $\lambda_j(t)$ and $\lambda_\xj(t)$ meet at a single point.
% This holds at $t=0$ since we start with a pseudoline arrangement $\Lambda(A,0)$.

% Assume for the sake of contradiction that $\lambda_j(t)$ and $\lambda_\xj(t)$ meet at more than one point for some $t \in [0,1]_\mb{R}$
Since we start with a pseudoline arrangement $\Lambda(A,0)$, $\lambda_j(0)$ and $\lambda_\xj(0)$ meet at a single point.
Since $\lambda_j(t)$ and $\lambda_\xj(t)$ are pseudolines, they must meet in at least 1 point, so there is some minimum time $t_0 > 0$ where the pseudolines meet at more than 1 point, and one of these points $x \in (\lambda_j(t_0) \cap \lambda_\xj(t_0))$ is not a limit point of $\lambda_j(t) \cap \lambda_\xj(t)$ for $t \to t_0$ from below.
If $x$ is in the interior of a segment, then either the segments coincide or they cross at $x$, but the segments cannot cross, since that would make $x$ a limit point of $\lambda_j(t) \cap \lambda_\xj(t)$ for $t \to t_0$.
Therefore, we may assume that $x$ is a vertex of $\lambda_j(t_0)$ and of $\lambda_\xj(t_0)$.

We cannot have $x = p_{1,j}(t_0) = p_{1,\xj}(t_0)$, since these points are fixed throughout the deformation.  Therefore, $x$ is one of the other vertices, and by symmetry we may assume that $x = p_{2,j}(t_0) = p_{2,\xj}(t_0)$.

We will show that exterior angles of $\lambda_j(t_0)$ and $\lambda_\xj(t_0)$ at $x$ cannot be equal. 
If the exterior angles of $\lambda_j(t_0)$ and $\lambda_\xj(t_0)$ at $x$ were equal, then either they would cross at $x$, which is impossible, or they would coincide along the segments on both sides of $x$, which would imply that $\lambda_j(t_0) = \lambda_\xj(t_0)$.
If we had $\lambda_j(t_0) = \lambda_\xj(t_0)$, then $p_{1,j}(0) = p_{1,\xj}(0)$, so $P_j = P_\xj$, so $q_{0,j} = (P_j \cap \lambda_j(t_0)) = (P_\xj \cap \lambda_\xj(t_0)) = q_{0,\xj}$, but then $\lambda_j(0)$ and $\lambda_\xj(0)$ would meet at $p_{1,j}$ and at $q_{0,j}$, which contradicts that $\lambda_j(0)$ and $\lambda_\xj(0)$ meet at a single point.
Therefore, we may assume that the exterior angle of $\lambda_j(t_0)$ at $x$ is strictly greater than that of $\lambda_\xj(t_0)$.  Thus, $\lambda_j(t_0)$ has already begun deforming by time $t_0$, while $\lambda_\xj(t_0)$ has been fixed up to time $t_0$.
%, so 
%the exterior angle of $\lambda_j(t_0)$ is $\nicefrac{\pi}{2}(1-t_0)$ while the exterior angle of $\lambda_\xj(t_0)$ is $\phi_\xj(t_0) =  \phi_\xj(0) < \nicefrac{\pi}{2}(1-t_0)$, and 
%$x = p_{2,j}(t_0) = p_{2,\xj}(t)$ for all $t \in [0,t_0]_\mb{R}$.

Let $C_j$ and $C_\xj$ be the cones emanating from $x$ that are respectively generated by the segments of $\lambda_j(t_0)$ and $\lambda_\xj(t_0)$ incident to $x$. 
Since the exterior angle of $\lambda_j(t_0)$ at $x$ is strictly greater than that of $\lambda_\xj(t_0)$, the interior angle of $\lambda_j(t_0)$ is strictly less than that of $\lambda_\xj(t_0)$, so these cones are nested $C_\xj \subset C_j$.

We now have that $t_0 < 1$, since $\lambda_j(1)$ and $\lambda_\xj(1)$ are both lines, and therefore have the same exterior angle, namely $0$.  Also, we know that $\lambda_j(t_0)$ cannot be a line, since it has exterior angle strictly greater than that of $\lambda_\xj(t_0)$, which is at least $0$.

Since $\lambda_j(t_0)$ is not a line, we have $q_{0,j} \neq x$, but $q_{0,j}$ is on the boundary of $C_j$.
Since $\lambda_j(t_0)$ crosses $L_{i_2}$ at $x$, the axis $L_{i_2}$ passes through the interior of $C_j$.
We now have $p_{2,j}(1) \in C_j$, since $\lambda_j(1)$ passes through $p_{2,j}(1)$, is parallel to one side of $C_j$, and intersects $C_j$ at $q_{0,j}$.
%Moreover, $p_{2,j}(1)$ must be in the interior of $C_j$, since $L_{i_2}$ only meets the boundary of $C_j$ at the apex $x$ and $\lambda_j(1)$ crosses into $C_j$ away from the apex $x \neq q_{0,j}$.

Since $\lambda_j(0)$ and $\lambda_\xj(0)$ do not intersect at $x \in \lambda_\xj(0)$, $\lambda_j$ must have started deforming at an earlier time $s < t_0$.
For $t \in [s,1]_{\mb{R}^2}$, the angle of $\lambda_j(t)$ at $q_{0,j}$ changes monotonically at a uniform speed, so the point $p_{2,j}(t)$ moves continuously and monotonically from $p_{2,j}(s) = p_{2,j}(0)$ to $p_{2,j}(1)$ along $L_{i_2}$.
Hence, the point $p_{2,j}(t)$ is outside $C_\xj$ for time $t<t_0$, and then crosses into to $C_\xj$ at time $t_0$, and then stays in $C_\xj$ for time $t \geq t_0$. 
%Hence, for $t \geq t_0$, $p_{2,j}(t)$ is in the ray $C_\xj \cap L_{i_2}$ with $p_{2,j}(t_0)=x$ on the boundary of the ray, while for for $t < t_0$, $p_{2,j}(t) \not\in C_\xj$. 

Let $R(t) = [p_{2,j}(t),p_{1,j})_\oplus$.
The rays $R(t)$ are all parallel
and ray $R(t_0)$ is in the interior of $C_\xj$, so each ray $R(t)$ intersects $C_\xj$.
In particular, the rays $R(t)$ for $t < t_0$ extend from the point $p_{2,j}(t) \not\in C_\xj$ and cross into $C_\xj$ at a point $y(t) \in \lambda_j(t_0)\cap \lambda_\xj(t_0)$, and $y(t) \to y(t_0) = x$ as $t \to t_0$ from below.
But this is a contradiction, since $x$ is not a limit point of $\lambda_j(t) \cap \lambda_\xj(t)$ for $t \to t_0$.
Thus, our assumption cannot hold, so $\Lambda(A,t)$ must be a pseudoline arrangement for all $t \in [0,1]_\mb{R}$.

Therefore, by Claims \ref{claim-lambda-continuous-3} and \ref{claim-pseudoline-lift}, $h_I$ is well defined as the equivariant deformation induced by $\tilde\lambda$.
\end{claimproof}

\begin{claim}\label{claim-deform-h3}
$h_{(i_1,i_2,i_3)}$ is a strong $\orth_3$-equivariant deformation retraction from $Z_{(i_1,i_2,i_3)}$ to $Z_{(i_1,i_2)} = Z_{()}$.
\end{claim}

\begin{claimproof}
Since $\lambda_j(t)$ is geodesic in each cell of $L_{i_1},L_{i_2},L_{i_3}$ throughout the deformation, we have $h_I(t) \in Z_I$, and since $\lambda_j(1)$ is a line, we have $h_I(1) \in Z_{()}$.
Therefore, $h_I$ is a deformation retraction from $Z_I$ to $Z_{()}$.
Since $\lambda$ is the trivial deformation if $\lambda(0)$ is a line, $h_I$ is a strong deformation retraction.
\end{claimproof}

Recall that for $m < 3$, $h_I(A,t) = A$ is trivial.

\begin{claim}\label{claim-stable-h}
For all $p \in \{0,\dots,m{-}1\}$ and all $j \in [n]_\mb{N} \setminus \{i_1,\dots,i_p\}$, 
if $A \in Z_I \cap X_{(i_1,\dots,i_{p},j)}$ then $h_I(A,1) \in X_{(i_1,\dots,i_{p},j)}$.
\end{claim}

\begin{claimproof}
For $p = 0$ or $p \geq 3$, the claim holds since the deformation $h_I$ preserves norms.
For $p = 1$, if $S_{i_1}$ and $S_j$ are distinct, then they remain distinct throughout the deformation, so the claim holds.
For $p = 2$ and $m > 3$, the pseudocircle $S_j$ can only deform in interiors of 2-cells of $S_{i_1},\dots,S_{i_m}$, so the intersections of $S_j$ with $S_{i_1}$ and with $S_{i_2}$ are preserved throughout the deformation, so the claim holds.
For $p = 2$ and $m = 3$, the intersection of $\lambda(t)$ with $L_{i_1}$ is fixed throughout the deformation, so the points where $S_{i_2}$ and $S_j$ meet $S_{i_1}$ are preserved throughout the deformation, so the claim holds.
\end{claimproof}

\subsubsection{The deformation $g_I$ from $Y_I$ to $Z_I$ for $n < \infty$}
\label{subsubsec-g-finite}

Let $(\cdot)$ denote the operation augmenting a sequence, $(x_1, \dots, x_k) \cdot y = (x_1, \dots, x_k, y)$.
Let this also denote concatenation of deformations,
\[\phi_2 \cdot \phi_1 (x,t) \dfeq \left\{ \begin{array}{ll}
\phi_1(x,\ 2t) & t \in [0, \nicefrac{1}{2}]_\mb{R} \\
\phi_2(\phi_1(x,1),\ 2t{-}1) & t \in [\nicefrac{1}{2},1]_\mb{R} \\
\end{array} \right.
\]
with composition left associative.  Note that $\phi_2$ must be a deformation of the range of $\phi_1$ at $t=1$.
We use $\prod_{j \in J} \phi_j$ to denote the concatenation of the deformations $\phi_j$ for $j \in J$ in increasing order from right to left.

We next define the deformations $g_I$ using the $f_I$ and $h_I$.  
Here we assume that $n \in \{3,\dots\}$.  The infinite case will be dealt with in the next subsubsection.

Our situation so far is this. 
If we were only concerned with the space of arrangements where none of the weighted pseudocircles vanish, and the first three $\alpha_1,\alpha_2,\alpha_3$ are always a basis, then we could straighten all pseudocircles with the deformation given by
%${A} \in X$ where $X \subset X_{(1,2,3)}$ such that $\alpha_i \neq 0$ for all $i \in [n]_\mb{N}$, then we could use the deformation retraction from $X$ to $Z_{()}$ given by
\[ h_{(1)} \cdot h_{(1,2)} \cdots h_{(1,\dots,n)} \cdot f_{(1,\dots,n)} \cdots f_{(1,2)} \cdot f_{(1)},  \]
%: X \times [0,1]_\mb{R} \to X.
% which deforms the pseudocircles by the $f_I$ to be piecewise geodesic one at time, then deforms the pseudocircles by the $h_I$ to be geodesic in the cells of successively coarser subdivisions of $\sphere^2$. 
but these assumptions only hold on some proper subset of $\pstief_{3,n}$. 
What we need is a deformation retraction from $Y_{()} = \pstief_{3,n}$ to $Z_{()}$.
Moreover, $f_{(1,\dots,{m-1})}$ is a deformation retraction to $Y_{(1,\dots,{m-1})}$, and the next deformation we would like to use $f_{(1,\dots,m)}$ is only defined on $X_{(1,\dots,m)}$, which is a proper subset of $Y_{(1,\dots,{m-1})}$.
Even worse, the spaces $X_{I\cdot j}$ for $j \not\in I$ do not even cover $Y_I$.
To deal with this, we will use the fact that these deformations approach the trivial deformation near the complement of the union of the $X_{I\cdot j}$.
We will also continuously shift between the deformations corresponding to different sequences $I \cdot j$ by stopping certain deformations early.

Let
\[ U_I = \bigcup_{j\in\compl{I}} X_{I\cdot j} 
= \left\{ (\alpha_1,\dots,\alpha_n) \in Y_I:\ \exists j \in \compl{I}.\ \alpha_j > 0 \right\}. \]
Note that the second equality above holds since every $A \in \pstief_{3,n}$ has a basis.  Moreover, since every independent set can be completed to a basis, we have that $U_I = Y_I$ for $m < 3$.

For spaces $X \subseteq Y$, let
\begin{align*}
&\upto :  (X \times [0,1]_\mb{R} \to Y) \times [0,1]_\mb{R} \to (X \times [0,1]_\mb{R} \to Y) , \\  
%&\Bigl(\upto(\phi,s)\Bigr)(x,t) = \phi(x,\min(s,t)). 
&\upto(\phi,s;x,t) = \phi(x,\min(s,t)). 
\end{align*}
That is, $\upto(\phi,s)$ is the deformation of $X$ that coincides with $\phi$ up to the stopping time $s$ and then remains fixed thereafter. 
We allow $\upto(\phi,0)$ to be the trivial deformation even for input that is outside the domain of $\phi$.
That is, we let $\upto(\phi,0; y,t) = y$ for $y \in Y$.

% \left\{ \begin{array}{ll} 
% \phi(x,t) & t \in [0, s]_\mb{R} \\
% \phi(x,s) & t \in [s,1]_\mb{R}. \\
% \end{array} \right.
% \]

Recall $m=|I|$.
If $m=n$, then $g_I$ is just the trivial deformation $g_I({A},t) = {A}$.
Otherwise for $m<n$, $g_I$ is defined recursively by 
\begin{equation}\label{equation-g}
 g_{I}({A},t) = \left( \prod_{j \in \compl{I} }  \upto\left(h_{I \cdot j} \cdot g_{I \cdot j} \cdot f_{I \cdot j},\ s_{I \cdot j} ({A}) \right)\right) ({A},t), 
\end{equation}
where $s_{I \cdot j}$ is defined for $j \in \compl{I}$ as follows,
% First, let
\[ r_{I\cdot j}(A) = \inf \left\{\dist(B,A): B \in Y_{I} \setminus X_{I\cdot j} \right\}, \]
\[ s_{I \cdot j}(A) = \begin{cases} 
0 & A \not\in U_{I} \\
\left(2\left(\frac{\displaystyle  r_{I\cdot j}({A})}{\displaystyle \max_{k \in \compl{I}} r_{I\cdot k}({A}) }\right) -1 \right)^+ & A \in U_{I}, \\
\end{cases}\]
where $(\phi(x))^+ = \max(0,\phi(x))$ is the positive part of a function $\phi$.

Here, we let $\upto\left(h_{I \cdot j} \cdot g_{I \cdot j} \cdot f_{I \cdot j},\ s \right)$ be the trivial deformation on all of $Y_I$ when $s=0$.  Note that if $s > 0$, then this deformation is only defined on $X_{I\cdot j}$.

% Let 
% $ \phi_{I\cdot j}(s) = \upto\left(h_{I \cdot j} \cdot g_{I \cdot j} \cdot f_{I \cdot j},\ s \right).$
% % \begin{remark}\label{upto-0}
% For $s > 0$, 
% $\phi_{I\cdot j}(s) : X_{I\cdot j} \times [0,1]_\mb{R} \to Y_I$
% is a deformation of $X_{I\cdot j} \subset Y_I$, whereas we let
% $\phi_{I\cdot j}(0) : Y_I \times [0,1]_\mb{R} \to Y_I$
% be the trivial deformation on all of $Y_I$ given by $\phi_{I\cdot j}(0,A,t) = A$.
% \end{remark}

% \[ \tilde g_{I, j_1}({A},t) = \left( \prod_{j \in \compl{I},\ j \leq j_1 }  \upto\left(h_{I \cdot j} \cdot g_{I \cdot j} \cdot f_{I \cdot j},\ s_{I \cdot j} ({A}) \right)\right) ({A},t). \]

\begin{claim}\label{claim-s}
For all $A \in Y_I$ and $j \in \compl{I}$, $s_{I\cdot j}(A)$ is well defined, and $s_{I\cdot j}$ is continuous on $U_{I}$.
\end{claim}

Note that $s_{I\cdot j}$ is not continuous on all of $Y_I$.

\begin{proof}
If $A \not\in U_{I}$, then $s_{I\cdot j}(A)=0 $ is well defined.
Alternatively, if $A \in U_{I}$, then there is some $j_0 \in \compl{I}$ such that $A \in X_{I\cdot j_0}$, which implies that $r_{I\cdot j_0}(A) > 0$ by Claim \ref{claim-open-X}.  
Therefore, $\max_{k \in \compl{I}} r_{I\cdot k}({A}) > 0$, which implies that $s_{I\cdot j}(A)$ is well defined for all $j \in \compl{I}$.

For the second part, observe that distance to a closed set in a metric space is continuous, so each $r_{I\cdot j}$ is continuous on $Y_I$.  We have that $\max_{k \in \compl{I}} r_{I\cdot k}(A)$ is strictly positive on $A \in U_I$, so $(\max_{k \in \compl{I}} r_{I\cdot k}(A))^{-1}$ is continuous there, which implies that $s_{I\cdot j}$ is continuous on $U_I$.
\end{proof}

% \newcommand{\complto}[2]{\rlap{${#1}^{\text{c}}$}\phantom{#1}_{#2}}

%{
\begin{claim}\label{claim-g}
$g_I$ is a well defined strong equivariant deformation retraction from $Y_I$ to $Z_I$.

Moreover, let $\compl{I} = \{j_1,\dots,j_{n-m}\}$ and for $e \in \{0,\dots,n-m\}$ let 
\[ \tilde g_{I, e}({A},t) = \left( \prod_{j =j_1}^{j_e}  \upto\left(h_{I \cdot j} \cdot g_{I \cdot j} \cdot f_{I \cdot j},\ s_{I \cdot j} ({A}) \right) \right) ({A},t) \]
be the initial part of the deformation $g_I$ up to the $e$-th element of $\compl{I}$.
For all $A \in Y_I$, $t \in [0,1]_\mb{R}$, and $e \in \{0,\dots,n{-}m\}$, we have the following.

\vbox{
\begin{enumerate}
\item\label{part-g-defined}
$\tilde g_{I, e}(A,t) \in Y_I$ is well defined and equivariant.
\item\label{part-g-continuous}
$\tilde g_{I, e}$ is continuous.
\item\label{part-g-strong}
If $A \in Z_I$, then $\tilde g_{I,e}(A,t) \in Z_I$.
\item\label{part-g-stable}
For all $p \in \{0,\dots,m\}$ and $j \in [n]_\mb{N} \setminus \{i_1,\dots,i_{p}\}$, 
if $A \in X_{(i_1,\dots,i_{p},j)}$ then $\tilde g_{I,e}(A,1) \in X_{(i_1,\dots,i_{p},j)}$.
\item\label{part-g-proj}
$\proj_I(\tilde g_{I,e}({A},t)) = \proj_I({A})$ and $\tilde g_{I,e}$ preserves norms.
%\item\label{part-g-end}
%$g_I(A,1) \in Z_I$.
\end{enumerate}
}
\end{claim}

\begin{proof}
We proceed by nested induction arguments on $m$ decreasing from $m=n$ and on $e$ increasing from $e=0$.  To show continuity, we consider $A_k \to A$ and $t_k \to t$.

For $m = n$, $g_I$ is just the trivial deformation and $Z_I = Y_I$, so the claim holds.  Let $m < n$.  Our first inductive assumption is that for all $j\in \compl{I}$, the claim holds for $g_{I\cdot j} = \tilde g_{I\cdot j,n-m-1}$.

Also for $e=0$, $\tilde g_{I,e}$ is just the trivial deformation, so the parts of the claim for $\tilde g_{I,e}$ hold.
Let $e > 0$, and let $A' = \tilde g_{I,e-1}(A,1)$.  Our second inductive assumption is that the claim holds for $\tilde g_{I,e-1}$.
%, except for part~\ref{part-g-end} and the conclusion, which are only relevant to $g_I$.

Suppose first that $A \in X_{I\cdot j_e}$.

We start by showing parts \ref{part-g-defined}, \ref{part-g-strong}, and \ref{part-g-proj}.
Here is the crucial reason why we need part \ref{part-g-stable} of the claim to show that $g_I$ is well defined: by the second inductive assumption for part~\ref{part-g-stable} of the claim, we have that $A' \in X_{I\cdot j_e}$.  This implies that $f_{I\cdot j_e}(A',t) \in X_{I\cdot j_e}$ is well defined, equivariant, and trivial on $Y_{I\cdot j_e}$ by Claim~\ref{claim-deform-f}, and $f_{I\cdot j_e}(A',1) \in Y_{I\cdot j_e}$.
Therefore, $g_{I\cdot j_e} \cdot f_{I\cdot j_e}(A',t) \in Y_{I\cdot j_e}$ is well defined, equivariant, and trivial on $Z_{I\cdot j_e}$ by the first inductive assumption, and $g_{I\cdot j_e} \cdot f_{I\cdot j_e}(A',1) \in Z_{I\cdot j_e}$, so $h_{I\cdot j_e}\cdot g_{I\cdot j_e}\cdot f_{I\cdot j_e}(A',t)  \in Z_{I\cdot j_e}$ is well defined, equivariant, and trivial on $Z_I$ by Claims~\ref{claim-deform-h} and \ref{claim-deform-h3}. 
% and $h_{I\cdot j_e}\cdot g_{I\cdot j_e}\cdot f_{I\cdot j_e}(A',1)  \in Z_{I}$, so part~\ref{part-g-end} holds.  
% $h_{I\cdot j}\cdot g_{I\cdot j}\cdot f_{I\cdot j}(A',1) \in Z_{I}$.
Recall that $Z_I \subset Z_{I\cdot j_e} \subseteq Y_{I\cdot j_e} \subset X_{I\cdot j_e} \subset Y_I$, so $\tilde g_{I, j_1}({A},t) \in Y_I$ is well defined, equivariant, and trivial on $Z_I$, which means parts~\ref{part-g-defined} and~\ref{part-g-strong} hold.
% Each of these maps act trivially on $A \in Z_I$, so part~\ref{part-g-strong} holds.
Similarly, we have part~\ref{part-g-proj} by the inductive assumptions and the definitions of $f_{I\cdot j_e}$ and $h_{I\cdot j_e}$.

Next we show part \ref{part-g-stable}.
Consider now the case where $A \in X_{(i_1,\dots,i_{p},j)}$.
Again, $A' \in X_{(i_1,\dots,i_{p},j)}$ by the second inductive assumption, so $f_{I\cdot j_e}(A',1) \in X_{(i_1,\dots,i_{p},j)}$ by Claim~\ref{claim-stable-f}, and so $g_{I\cdot j_e}\cdot f_{I\cdot j_e}(A',1) \in X_{(i_1,\dots,i_{p},j)}$ by the first inductive assumption, and so $h_{I\cdot j_e}\cdot g_{I\cdot j_e}\cdot f_{I\cdot j_e}(A',1) \in X_{(i_1,\dots,i_{p},j)}$ by Claim~\ref{claim-stable-h}.
Thus, part~\ref{part-g-stable} holds.
%$\tilde g_{I, j_1}({A},1) \in X_{(i_1,\dots,i_{p},j)}$.

Next we show part \ref{part-g-continuous}.
By the second inductive assumption, we have $A'_k = g_{I,e-1}(A,1) \to A' \in X_{I\cdot j_e}$. 
Since $X_{I\cdot j_e}$ is an open subset of $Y_I$ by Claim~\ref{claim-open-X}, $A'_k \in X_{I\cdot j_e}$ for $k$ sufficiently large, and since $f_{I\cdot j_e}$ is continuous on $X_{I\cdot j_e}$, we have $f_{I,e}(A'_k,t_k) \to f_{I,e}(A',t)$.
Since $g_{I\cdot j_e}$, $h_{I\cdot j_e}$, and $s_{I\cdot j_e}$ are respectively continuous on $Y_{I\cdot j_e}$, $Z_{I\cdot j_e}$, and $U_{I} \supset X_{I\cdot j_e}$, we have that $\tilde g_{I,e}$ is continuous. Thus, part~\ref{part-g-continuous} holds.

Suppose next that $A \not\in X_{I\cdot j_e}$.

Then, $s_{I\cdot j_e}(A) = 0$, so $\tilde g_{I,e}(A,t)$ is the concatenation of $\tilde g_{I,e-1}(A,t)$ with a trivial deformation.
Specifically, $\tilde g_{I,e}(A,t) = \tilde g_{I,e-1}(A,\tau(t))$ where $\tau(t) = \min(2t,1)$ 
%\[ \tau(t) = \begin{cases} 2t & t \in [0,\nicefrac{1}{2}]_\mb{R} \\ 1 & t \in [\nicefrac{1}{2},1]_\mb{R} \end{cases} \]
is a non-decreasing surjective transformation of the unit interval, which means that parts \ref{part-g-defined}, \ref{part-g-strong}, \ref{part-g-stable}, and \ref{part-g-proj} hold by the second inductive assumption.  It remains to show part \ref{part-g-continuous} in this case.

Suppose also that $A \in U_I$.  Then, there is some $j$ such that $A \in X_{I\cdot j}$, so $r_{I\cdot j}(A_k)$ is bounded below by some $r > 0$ for $k$ sufficiently large, whereas $r_{I\cdot j_e}(A_k) \to 0$, which implies that $s_{I\cdot j_e}(A_k) = 0$ for $k$ sufficiently large.  Therefore, $\tilde g_{I,e}(A_k,t_k) = \tilde g_{I,e-1}(A_k,\tau(t_k))$, so part~\ref{part-g-continuous} holds.

Suppose instead that $A \not\in U_I$.  
This implies that for all $j \in \compl{I}$, $s_{I\cdot j}(A) = 0$, so $\tilde g_{I\cdot e}(A,t) = A$.
Since every $A \in \pstief_{3,n}$ includes a basis among its elements, and any independent set can be completed to a basis, we have $U_I = Y_I$ for $m \leq 2$.  Therefore, we must have $m \geq 3$ in this case, which means that each of the $X_{I\cdot j}$ consists of the weighted pseudocircle arrangements in $Y_I$ where the $j$-th element does not vanish.  Since $A$ is in none of the $X_{I\cdot j}$, we have $\proj_I(A) = A$, which implies that $\proj_I(A_k) \to A$. 
%Therefore, $\|\alpha_{k,j}\| \to 0$ where $A_k = (\alpha_{k,1},\dots,\alpha_{k,n})$, so $\dist(A_k,\proj_I(A_k)) \to 0$.
We have already shown that part~\ref{part-g-proj} of the claim holds for $\tilde g_{I,e}(A_k,t_k)$ (in both cases $A_k \not\in X_{I\cdot j_e}$ and $A_k \in X_{I\cdot j_e}$), so $\proj_I(\tilde g_{I,e}(A_k,t_k)) = \proj_I(A_k)$ and $\tilde g_{I,e}$ preserves norms.  Therefore, for all $j \in \compl{I}$, the $j$-th element of $\tilde g_{I,e}(A_k,t_k)$ converges to 0, since the $j$-th element of $A_k$ converges to 0.  This implies that \[ \dist(\tilde g_{I,e}(A_k,t_k),\proj_I(A_k)) = \dist(\tilde g_{I,e}(A_k,t_k),\proj_I(\tilde g_{I,e}(A_k,t_k))) \to 0, \]
so $\tilde g_{I,e}(A_k,t_k) \to A = \tilde g_{I\cdot e}(A,t)$.
Thus, part~\ref{part-g-continuous} holds.

This completes the induction on $e$.  So far we have established parts \ref{part-g-defined}, \ref{part-g-continuous}, \ref{part-g-strong}, \ref{part-g-stable}, and \ref{part-g-proj} of the claim.  

It remains to show that $g_I(A,1) \in Z_I$. %part~\ref{part-g-end}.
If $A \not\in U_I$, then we must have $m \geq 3$, which means $\alpha_j = 0$ for all $j \in \compl{I}$, so $A \in Z_I$ trivially, so $g_I(A,1) \in Z_I$.
Alternatively, if $A \in U_I$, then there is some $j_e \in \compl{I}$ such that $r_{I\cdot j_e}(A)$ is maximal among all $r_{I\cdot j}(A)$ for $j \in \compl{I}$.
Therefore, $s_{I\cdot j_e}(A) = 1$, so $\tilde g_{I\cdot j_e}(A,1) \in Z_I$ by Claims \ref{claim-deform-h} and \ref{claim-deform-h3}.  Since $f_{I\cdot j}$, $g_{I\cdot j}$, and $h_{I\cdot j}$ are all trivial on $Z_I$, once $g_I(A,t)$ attains a value in $Z_I$ for some $t \in [0,1]_\mb{R}$, it is trivial thereafter.  That is, $g_I(A,t') = g_I(A,t) \in Z_I$ for all $t' \in [t,1]_\mb{R}$, so we have $g_I(A,1) \in Z_I$. %part \ref{part-g-end} holds.

Finally, parts \ref{part-g-defined}, \ref{part-g-continuous}, and \ref{part-g-strong} applied to $\tilde g_{I,n-m} = g_I$ together with $g_I(A,1) \in Z_I$ imply that $g_I$ is a strong equivariant deformation retraction from $Y_I$ to $Z_I$.
\end{proof}

\subsubsection{The deformation $g_{()}$ from $\pstief_{3,\infty}$ to $Z_{()}$ for $n = \infty$}

We will define a strong equivariant deformation retraction $g_{()}$ from $\pstief_{3,\infty}$ to $Z_{()}$.
For any finite non-repeating sequence $I$ of natural numbers, we have $X_I,Y_I,Z_I \subset \pstief_{3,\infty}$ defined in the same way as in Subsubsection \ref{subsubsec-spaces} for $n = \infty$.
We also have the deformations $f_I$ from $X_I$ to $Y_I$ from Subsubsection \ref{subsubsec-f} and the deformation $h_I$ from $Z_{I}$ to $Z_{(i_1,\dots,i_{m-1})}$ from Subsubsection \ref{subsubsec-h} in the infinite case.

For $\hat n \geq 3$, we define strong equivariant deformation retractions $\hat g_{I,\hat n}$ from $Y_I \cap \pstief_{3,\hat n}$ to $Z_I \cap \pstief_{3,\hat n}$, which will be a little different from the deformation of Subsubsection \ref{subsubsec-g-finite}.
For $m = |I| \geq \hat n$, we again let $\hat g_{I,\hat n}$ be the trivial deformation.  Otherwise, $\hat g_{I,\hat n}$ is defined recursively by 
\begin{equation}\label{equation-g-infty}
 \hat g_{I,\, \hat n}({A},t) = \left( \prod_{j \in \mb{N} \setminus I }  \upto\left(h_{I \cdot j} \cdot \hat g_{I \cdot j,\, \hat n} \cdot f_{I \cdot j},\ s_{I \cdot j} ({A}) \right)\right) ({A},t), 
\end{equation}
where $s_{I\cdot j}$ is the same as in Subsubsection \ref{subsubsec-g-finite} with $n = \infty$.
Finally, let $g_{()}(A,t) = \hat g_{(),\hat n}(A,t)$ for $A \in \pstief_{3,\hat n} \subset \pstief_{3,\infty}$.

Note that the composition of deformations is defined to be left associative, so that an infinite composition of deformations is well defined for $t<1$.  Specifically, a composition of deformations 
\[\phi = \left(\prod_{j \in \mb{N}} \phi_j\right) (A,t) \]
performes the deformation $\phi_1$ twice as fast up to time $t = \nicefrac{1}{2}$, and then $\phi_2$ four times as fast up to time $t = \nicefrac{3}{4}$, etcetera.
When $\phi(A,t)$ converges as $t \to 1$ from below, the final state is defined as the limit $\phi(A,1) = \lim_{t \to 1} \phi(A,t)$.  

\begin{claim}\label{claim-consistent-g}
$\hat g_{I,\hat n}$ is a well defined strong equivariant deformation retraction from $Y_I \cap \pstief_{3,\hat n}$ to $Z_I \cap \pstief_{3,\hat n}$.
Also, for $\hat n_1 < \hat n_2$ and $A \in Y_I \cap \pstief_{3,\hat n_1}$, $\hat g_{I,\hat n_1}(A,t) = \hat g_{I,\hat n_2}(A,t)$.
\end{claim}

\begin{proof}
First observe that if there is some $i_k \in I$ such that $i_k > \hat n$, then $Y_I \cap \pstief_{3,\hat n} = Z_I \cap \pstief_{3,\hat n} = \emptyset$, so we may assume all entries of $I$ are no greater than $\hat n$.
For $j > \hat n$ and $A \in Y_I \cap \pstief_{3,\hat n}$, we have $A \not \in X_{I\cdot j}$, so $r_{I\cdot j}(A) = 0$, which makes $s_{I\cdot j}(A) = 0$.  Hence, $\hat g_{I,\hat n}$ is a composition of deformations that become trivial after $j > \hat n$, and therefore $\hat g_{I,\hat n}$ is a well defined equivariant deformation retraction by the same argument as in Claim \ref{claim-g}.  

For the second part, we proceed by induction on $\hat n_1 -m$.
For $m = \hat n_1$, we have that $I$ is a permutation of $[\hat n_1]_\mb{N}$, so $A \in (Y_I \cap \pstief_{3,\hat n_1}) = (Z_I \cap \pstief_{3,\hat n_1})$.  Since $\hat g_{I,\hat n_2}$ is a strong deformation retraction, $\hat g_{I,\hat n_2}(A,t) = A = \hat g_{I,\hat n_1}(A,t)$.  For $m < n_1$, we have $\hat g_{I,\hat n_2}(A,t) = \hat g_{I,\hat n_1}(A,t)$ by induction.
\end{proof}

\begin{claim}\label{claim-infinite-g}
$g_{()}$ is a well defined strong equivariant deformation retraction from $\pstief_{3,\infty}$ to $Z_{()}$.
\end{claim}

\begin{proof}
Let $\iota_{\hat n} : \pstief_{3,\hat n} \hookrightarrow \pstief_{3,\infty}$ denote inclusion.
By the second part of Claim \ref{claim-consistent-g}, there is a unique function $g_{()}$ that makes the following diagram commute for all $\hat n \in \{3,\dots\}$,
\[
\begin{tikzpicture}
\matrix (m) [matrix of math nodes, row sep=4em,
column sep=5em, text height=1.5ex, text depth=0.25ex]
{ 
\pstief_{3,\hat n} \times [0,1]_\mb{R} & \pstief_{3,\hat n} \\
\pstief_{3,\infty} \times [0,1]_\mb{R} & \pstief_{3,\infty} \\
};
\path[->]%,font=\scriptsize]
(m-1-1) edge node[auto] {$\hat g_{(),\hat n} $} (m-1-2)
(m-1-2) edge node[auto] {$\iota_{\hat n}$} (m-2-2)
(m-1-1) edge node[auto] {$\iota_{\hat n} \times \id$} (m-2-1)
(m-2-1) edge node[auto] {$ g_{()} $} (m-2-2) 
;
\path[dashed,->]
(m-1-1) edge (m-2-2)
;
\end{tikzpicture} %\qedhere
\]
so $g_{()}$ is well defined. By the first part, it only remains to show that $g_{()}$ is continuous.
We have $g_{()} \circ (\iota_{\hat n} \times \id) = \iota_{\hat n} \circ \hat g_{(),\hat n}$ is continuous by the universal property of the subspace topology, since $\hat g_{(),\hat n}$ is continuous.  Therefore, $g_{()}$ is continuous by the universal property of the direct limit topology.
\end{proof}

\subsubsection{Orthonormalization}

To complete the deformation, we perform a continuous orthonormalization.  This may be accomplished in a variety of ways, of which a continuous analog of the Gram–Schmidt process may be the most familiar.  For our purposes, a continuous deformation using the polar decomposition of a matrix is more directly relevant, so that is what we do here.

We may represent $A = (\alpha_1,\dots,\alpha_n) \in Z_{()} \subset (\Rpol{3})^n \simeq \mb{R}^{n\times 3}$ as the $(n\times 3)$-matrix (also denoted $A$) where the entries of the $j$-th row are given by the coordinates of $\pol^{-1}(\alpha_j) \in \mb{R}^3$.
In this way, we will simply treat $Z_{()}$ as the space of all full-rank $(n\times 3)$-matrices where it is convenient to do so, and make use of the standard matrix operations of matrix multiplication and taking roots of symmetric positive definite matrices.

For $n = \infty$, we treat $Z_{()}$ as the union of the ascending chain of spaces of full-rank $(\hat n\times 3)$-matrices with the direct limit topology.  Here matrices that differ by a tail of rows of zeros are identified.

\newcommand{\I}{\text{I}}

Let $q$ be the deformation of $Z_{()}$ by
\[ q(A,t) = A \left(t(A^*A)^{-\nicefrac{1}{2}}+(1-t)\I\right),  \]
where $\I$ is the identity $(3\times 3)$-matrix.

\begin{claim}\label{claim-orthonormalization}
$q$ is a well defined strong equivariant deformation retraction from $Z_{()}$ to $\stief_{3,n}$.
\end{claim}

\begin{proof}
Since $A$ is full-rank, $A^*A$ is positive definite symmetric, so it has a well defined square root that is also positive definite symmetric, so $q$ is well defined. Also, $q$ is defined by a composition of continuous functions for $n < \infty$ so $q$ is continuous.  For $n=\infty$, $q$ is continuous on $Z_{()}$, since $q$ is continuous on each subspace of $(\hat n\times 3)$-matrices.

% For the strong condition,
if $A \in \stief_{3,n}$, then $A$ has orthonormal rows, so $A^*A = \I$, and $\I^{-\nicefrac{1}{2}} = \I$, which gives $q(A,t) = A$.

% For the equivariant condition, 
For $Q \in \orth_3$, we have 
\begin{align*}
q(A*Q,t) %&= q(Q^*A,t) \\
&= AQ \left(t(Q^*A^*AQ)^{-\nicefrac{1}{2}}+(1-t)\I\right) \\
&= AQ \left(tQ^*(A^*A)^{-\nicefrac{1}{2}}Q+(1-t)\I\right) \\
&= A \left(t(A^*A)^{-\nicefrac{1}{2}}+(1-t)\I\right) Q \\
&= q(A,t)*Q. 
\end{align*}

Let $A_1 = q(A,1)$.  We have 
\[ A_1{}^* A_1 
= (A^*A)^{-\nicefrac{1}{2}} A^* A (A^*A)^{-\nicefrac{1}{2}} 
= (A^*A)^{-\nicefrac{1}{2}} (A^* A)^{\nicefrac{1}{2}} (A^* A)^{\nicefrac{1}{2}} (A^*A)^{-\nicefrac{1}{2}} 
= \I, \]
so $q(A,1) \in \orth_3$.

Finally, since $(A^*A)^{-\nicefrac{1}{2}}$ and $\I$ are both symmetric positive definite, positive linear combinations of $(A^*A)^{-\nicefrac{1}{2}}$ and $\I$ are also symmetric positive definite, which implies that $\left(t(A^*A)^{-\nicefrac{1}{2}}+(1-t)\I\right)$ is full-rank, so $q(A,t)$ is full-rank, so $q(A,t) \in Z_{()}$.
Thus, $q$ is a strong equivariant deformation retraction from $Z_{()}$ to $\stief_{3,n}$.
\end{proof}

\begin{proof}[Proof of Theorem \ref{theorem-grassman-pseudograssman}]
The deformation $q \cdot g_{()}$ is a strong $\orth_3$-equivariant deformation retraction from $\pstief_{3,n}$ to $\stief_{3,n}$ by Claims \ref{claim-orthonormalization} and \ref{claim-g} for $n < \infty$ or \ref{claim-infinite-g} for $n = \infty$.
Hence, taking the quotient of this deformation by the $\orth_3$-action on $\pstief_{3,n}$ provides a strong deformation retraction from $\pg_{3,n}$ to $\g_{3,n}$, and the quotient by $\sorth_3$ provides a strong deformation retraction from $\pog_{3,n}$ to $\og_{3,n}$.
\end{proof}

% That is, 
% %$P = \sqrt{AA^*}$ is a positive definite $(3,3)$-matrix, and $Q = P^{-1}A$ 
% $A = U \Sigma V^*$ is the compact singular value decomposition, $U \in \orth(3)$, $\Sigma$ is positive (3,3)-diagonal, and $V^*$ is a $_{3,n}$-matrix with orthonormal rows.
% 
% $P = U \Sigma U^* = \sqrt{AA^*}$
% 
% $Q = U V^*$

\subsection{Universal vector bundles and classifying spaces.}

Recall that a real rank $k$ vector bundle is a space that locally has the structure of a product with $\mb{R}^k$;
for a precise definition see \cite[page 24]{husemoller1994fibre}. 
%For a group $G$, a principal $G$-bundle is a topological space $X$ where $G$ acts freely, i.e.\ $\forall g \in G$ if $\exists x \in X.\ xg = x$ then $g=\id$, and there is a continuous translation function $\tau$, i.e. $\forall x,x' \in X$ if $\exists g \in G. \ x' = xg$, then $x' = x\tau(x,x')$.
%Note that if $G$ acts freely, then $\tau$ always exists and is unique, the second requirement is really that $\tau$ be continuous.
One way to obtain a vector bundle is as the quotient space $(X \times \mb{R}^k)/\orth_3$ where $X$ is a locally trivial principal $\orth_3$-bundle \cite[Capter~5]{husemoller1994fibre}.
%Also, the quotient by $(X \times \mb{R}^k)/\sorth_3$ is an oriented vector bundle.

%The local existence of sections on $X$ is enough to show $X$ is locally trivial \cite[Section 4.7]{husemoller1994fibre}.

%with a system of transition functions associated to an open cover of $X$ \cite[Capter 5, Theorem 3.2]{husemoller1994fibre}. 

%These transition functions can be obtained from local sections defined on $X$.
%More precisely, this is a map $\xi:E \to B$ where $\forall b \in B$, $\xi^{-1}(b)$ is a real $k$-dimensional vector space, and $\exists U \ni b$ open such that $\xi^{-1}(U)$ is homeomorphic to $U \times \mb{R}^k$ that is linear on fibers.
%If $X$ is a locally trivial principal $\orth_k$-bundle then $(X \times \mb{R}^k)/\orth_k$ is a rank $k$ vector bundle, and if $X$ is a principal $\sorth_k$-bundle then $(X \times \mb{R}^k)/\sorth_k$ is an oriented rank $k$ vector bundle.  Note that every principal $\orth_k$-bundle is also principal $\sorth_k$-bundle.  

%Let 
%\begin{align*}
%\pe_{3,n} = (\pstief_{3,n} \times \mb{R}^3)/{\orth_3}
%\end{align*}

\begin{samepage}
\begin{lemma}\label{lemma-principal-bundle}
$\pstief_{3,n}$ is a locally trivial principal $\orth_3$-bundle. 
Hence, the projection maps 
$$ \xi_{3,n}: \pe_{3,n} \to \pg_{3,n} \quad \text{and} \quad \widetilde\xi_{3,n}: \poe_{3,n} \to \pog_{3,n} $$ 
are respectively a vector bundle and oriented vector bundle. %respectively.
\end{lemma}
\end{samepage}

\begin{proof}
As a consequence of Lemma \ref{lemma-basis}, $\pstief_{3,n}$ is a free $\orth_3$-space.  To see this, let $Q \in \orth_3$ and $A = A*Q \in \pstief_{3,n}$, and $I$ be an independent set of $A$.  Then, $\coord(I;A) = \coord(I;A*Q) = Q^* \coord(I;A)$, so $Q = \id$.
This implies that there is a translation function
\[ \tau : \{(A,A*Q): A \in \pstief_{3,n}, Q \in \orth_3 \} \to \orth_3 \] 
where $\tau(A,B)$ is the unique orthogonal transformation such that $B = A*\tau(A,B)$.
To check uniqueness, if  $B = A*Q_1 = A*Q_2$, then $A = A*(Q_2 Q_1^{-1})$, and since the $\orth_3$ action is free $Q_2 Q_1^{-1} = \id$, so $Q_2=Q_1$.

Next we will verify the continuity of $\tau$ and define local sections associated to an open cover of $\pstief_{3,n}$.
Let $\mc{I}$ be the set of non-repeating ordered triples with entries among $[n]_\mb{N}$, and for each $I = (i_1,i_2,i_3) \in \mc{I}$ let $U_I$ denote the subset of $\pstief_{3,n}$ where $I$ is a basis.
If pseudocircles $S_{i_1},S_{i_2},S_{i_3}$ have no common point of intersection, then any triple of pseudocircles that are sufficiently close in Fréchet distance will also not have a common point of intersection, so $U_I$ is open.  Since every $A \in \pstief_{3,n}$ has a basis, $\{U_I : I \in \mc{I}\}$ is an open cover of $\pstief_{3,n}$.  

For $B = A*\tau(A,B)$ with basis $I$, we have $\coord(I;B) = \coord(I;A*\tau(A,B)) = \tau(A,B)^*\coord(I;A)$, so $\tau(A,B) = \coord(I;A)\coord(I;B)^*$, which is continuous on $A,B \in U_I$, so $\tau$ is continuous.
Thus, $\pstief_{3,n}$ is a principal $\orth_3$-bundle. 

Let $\mc{U}_I = U_I/\orth_3$, and define a local section on $U_I$ by
\[ s_I : \mc{U}_I \to U_I, \quad s_I(\mc{A}) = A *\coord(I;A) \text{ for } A \in \mc{A},  \]
Observe that $s_I$ does not depend on the choice of $A \in \mc{A}$, since for $B = A*Q$, we have
$B *\coord(I;B) = A*Q \coord(I;A*Q) = A*Q Q^* \coord(I;A) = A*\coord(I;A) $, 
and $s_I$ is continuous since $\coord(I)$ is continuous and by definition of the topology on $\pg_{3,n}$. 
We may equivalently let $s_I(\mc{A})$ be the unique element of $\mc{A}$ such that $\coord(I;s_I(\mc{A})) = \id$.

%We define a local $\orth_3$-bundle trivialization on $U_I$ by 
%\[ h_I : U_I \to (U_I/\orth_3) \times \orth_3, \quad h_I(A) = (A*\orth_3,\coord(I;A)). \]
%We have that $h_I$ is continuous, since $\coord(I)$ is continuous, and the local section gives us a continuous inverse, 
%$h_I^{-1}(\mc{A},Q) = s_I(\mc{A})*Q^*$.
%\[h_I^{-1}\circ h_I(A) = s_I(A*\orth_3)*\coord(I;A)^* = A *\coord(I;A) *\coord(I;A)^* = A \]
%\[h_I\circ h_I(\mc{A},Q) = (s_I(\mc{A})*Q^* *\orth_3, \coord(I;s_I(\mc{A})*Q^*)) = (\mc{A},Q) \]

Since $\pstief_{3,n}$ has local sections, $\pstief_{3,n}$ is locally trivial, which implies that $\pe_{3,n}$ is locally trivial \cite[Section 4.7]{husemoller1994fibre}. 
Specifically, we have a local trivialization on $\xi_{3,n}^{-1}(\mc{U}_I) \subset \pe_{3,n}$ given by 
\[ h_I : \xi_{3,n}^{-1}(\mc{U}_I) \to \mc{U}_I \times \mb{R}^3, \quad h_I(\mc{X}) = (\xi_{3,n}(\mc{X}),\coord(I;A)^*x) \text{ for } (A,x) \in \mc{X} \]
Observe that $h_I(\mc{X})$ does not depend on the choice of $(A,x)$, since for $(B,y) \in \mc{X}$, we have $B = A*Q$ and $y = Q^* x$, so 
$\coord(I;B)^*y = \coord(I;A*Q)^*Q^*x = (Q^*\coord(I;A))^*Q^*x = \coord(I;A)^*QQ^*x = \coord(I;A)^*x $.  Equivalently, we may choose $h_I(\mc{X})=(\mc{A},x)$ such that $(s_I(\mc{A}),x) \in \mc{X}$.  Again, $h_I$ is continuous since $\coord(I)$ is continuous, and we have a continuous inverse given by 
$h_I^{-1}(\mc{A},x) \in \pe_{3,n}$ such that $(s_I(\mc{A}),x) \in h_I^{-1}(\mc{A},x)$.  
For $\mc{A} \in \mc{U}_I \cap \mc{U}_J$, we have $h_I \circ h_J^{-1}(\mc{A},x) = (\mc{A},y)$ where $y = \coord(I;s_J(\mc{A}))^*x$, so $h_I \circ h_J^{-1}$ acts as a linear isometry on fibers. 
Thus, $\xi_{3,n}$ is a vector bundle.  

The same argument with $\sorth_3$ instead of $\orth_3$ shows that $\widetilde\xi_{3,n}$ is an oriented vector bundle.
\end{proof}

Recall that for a principal $\orth_3$-bundle $\xi: E \to B$ (or vector bundle or oriented vector bundle) and a continuous map $f: B' \to B$ between paracompact spaces, there is a pullback bundle (in the same category) $f^*(\xi)$ with base space $B'$, and if $f_0,f_1$ are homotopy equivalent then their pull back bundles $f_0^*(\xi),f_1^*(\xi)$ are isomorphic. 
This induces a map $\text{pb}(\xi,B')$ from homotopy classes of maps $B' \to B$ to isomorphism classes of bundles with base space $B'$ (in the same category as $\xi$), and we say $\xi$ is \df{universal} when $\text{pb}(\xi,B')$ is a bijection for every paracompact space $B'$.
In particular, the canonical bundles on $\stief_{3,\infty}$, $\e_{3,\infty}$, and $\oe_{3,\infty}$ are respectively universal for principal $\orth_3$-bundles, vector bundles, and oriented vector bundles \cite{husemoller1994fibre}.

\begin{proof}[Proof of Corollary \ref{corollary-classifying}]
Since $\e_{3,\infty}$ is a subspace of $\pe_{3,\infty}$, the map $\text{pb}(\xi_{3,\infty},B)$ is surjective for every paracompact space $B$.  That is every vector bundle over $B$ is isomorpic to the pullback of a map $f : B \to \e_{3,\infty}$, which is also a map into $\pe_{3,\infty}$. 
Suppose the pullbacks of a pair of maps $f_1,f_4 : B \to \pg_{3,\infty}$ are isomorphic vector bundles.  
Since $\e_{3,\infty}$ is a deformation retract of $\pe_{3,\infty}$, $f_1,f_4$ are respectively homotopic to a pair $f_2,f_3 : B \to \pg_{3,\infty}$, which are homotopic to each other since $\e_{3,\infty}$ is a universal vector bundle, so $f_1,f_4$ are homotopic to eachother.
Thus, $\pe_{3,\infty}$ is a universal vector bundle. 
The same argument applies to $\pstief_{3,\infty}$ and $\poe_{3,\infty}$.
\end{proof}

%\begin{observation}
%As a consequence of Lemma \ref{lemma-basis}, $\orth_3$ acts freely on $\pstief_{3,n}$, and therefore $\pe_{3,n} = (\pstief_{3,n} \times \mb{R}^3)/{\orth_3}$ is a vector bundle over $\pg_{3,n}$, which restricts to the canonical bundle over $\g_{3,n}$.  Also, $\pe_{3,n}$ is isomorphic to the pull back bundle of the retraction given by Theorem \ref{theorem-grassman-pseudograssman}. 
%Similarly, this holds for $\poe_{3,n} \dfeq (\pstief_{3,n} \times \mb{R}^3)/{\sorth_3}$ and $\og_{3,n}$.
%Since $\g_{3,\infty}$ and $\og_{3,\infty}$ are classifying spaces, we have Corollary \ref{corollary-classifying}.
%\end{observation}

\section*{Acknowledgments}

The author would like to thank 
Laura Anderson,
Pavle Blagojević,
Kenneth Clarkson,
\hbox{Andreas Holmsen},
and Alfredo Hubard 
for helpful discussions and insights.
This research was partly supported by the National Research Foundation of Korea NRF grant 2011-0030044, SRC-GAIA.

\bibliographystyle{plain}
\bibliography{OM7}

\end{document}